\documentclass[12pt]{article}
\usepackage{graphicx}
\usepackage{amsmath}
\usepackage{amsfonts}
\usepackage{enumerate}
\usepackage{latexsym}
\usepackage{epsfig}
\usepackage{float}
\usepackage{color}
\usepackage{amscd}
\usepackage{amssymb}
\usepackage{epstopdf}

\newtheorem{theorem}{Theorem}[section]
\newtheorem{lemma}[theorem]{Lemma}

\def\proofbox{\begin{picture}(6.5,6.5)
\put(0,0){\framebox(6.5,6.5){}}\end{picture}}
\newenvironment{proof}{\noindent{\it Proof.\quad}}{\hfill\proofbox}

\begin{document}

\title{Edge Preserving Maps of the Curve Graphs 
in Low Genus\\}
\author{Elmas Irmak}

\maketitle

\renewcommand{\sectionmark}[1]{\markright{\thesection. #1}}

\thispagestyle{empty}
\maketitle
\begin{abstract} Let $R$ be a compact, connected, orientable surface of genus $g$ with $n$ boundary components. Let $\mathcal{C}(R)$ be the curve graph of $R$. 
We prove that if $g=0, n \geq 5$ or $g=1, n \geq 3$, and $\lambda : \mathcal{C}(R) \rightarrow\mathcal{C}(R)$ is an edge preserving map, then $\lambda$ is induced by a homeomorphism of $R$, and this homeomorphism is unique up to isotopy. \end{abstract}
  
{\small Key words: Mapping class groups, edge preserving maps, orientable surfaces

MSC: 57N05, 20F38}
 
\section{Introduction} 
Let $R$ be a compact, connected, orientable surface of genus $g$ with $n$ boundary components. The mapping class group, $Mod_R$, of $R$ is defined to be the group of 
isotopy classes of orientation preserving self-homeomorphisms of $R$. The extended mapping class group, $Mod^*_R$, of $R$ is defined to be the group of 
isotopy classes of all self-homeomorphisms of $R$. Abstract simplicial complexes on surfaces have been studied to get information about the algebraic structure of the extended mapping class groups of the surfaces. One of these complexes is the \textit{complex of curves}, on $R$. The vertex set of the complex of curves
is the set of isotopy classes of nontrivial simple closed curves on $R$, where nontrivial means 
the curve does not bound a disk and it is not isotopic to a boundary component of $R$. A set of vertices forms a simplex in the complex of curves if
they can be represented by pairwise disjoint simple closed curves on the surface. Let $\mathcal{C}(R)$ be the curve graph, the 1st skeleton of the complex of curves on $R$. 
The main result is the following:

\begin{theorem} \label{A1} Let $R$ be a compact, connected, orientable surface with $g=0, n \geq 5$ or $g=1, n \geq 3$. If $\lambda :\mathcal{C}(R) \rightarrow \mathcal{C}(R)$ is an edge preserving map, 
then there exists a homeomorphism $h : R \rightarrow R$ such that $H(\alpha) = \lambda(\alpha)$
for every vertex $\alpha$ in $\mathcal{C}(R)$ where $H=[h]$ (i.e. $\lambda$ is induced by $h$) and this homeomorphism is unique up to isotopy.\end{theorem}
 
Ivanov gave a classification of isomorphisms between any two finite index subgroups of the extended mapping class group of $R$ and proved that every automorphism of the complex of curves is induced by a homeomorphism of $R$ if the genus is at least two in \cite{Iv1}. These   
results were extended for surfaces of genus zero 
and one by Korkmaz in \cite{K1} and indepedently by Luo in \cite{L}. The author proved that the superinjective simplicial maps of the complexes of curves on a compact, connected, orientable surface are induced by homeomorphisms if the genus is at least two, and using this result she gave a classification of injective homomorphisms from finite index subgroups to the extended 
mapping class group in \cite{Ir1}, \cite{Ir2}, \cite{Ir3}. These results 
were extended to lower genus cases by Behrstock-Margalit and Bell-Margalit in \cite{BM} and in \cite{BeM}. We remind that  
superinjective simplicial maps are simplicial maps that preserve geometric intersection zero and nonzero properties.  
After these results, Shackleton proved that locally injective simplicial maps of the curve complex are induced by homeomorphisms in \cite{Sh}.  

Aramayona-Leininger proved that there is an exhaustion of the curve complex by a sequence of finite rigid sets in \cite{AL2}. Irmak-Paris proved that superinjective simplicial maps of the two-sided curve complex are induced by homeomorphisms on compact, connected, nonorientable surfaces when the genus is at least 5 in \cite{IrP1}. They also gave a classification of injective homomorphisms from finite index subgroups to the whole mapping class group on these surfaces in \cite{IrP2}. In this paper we use some techniques given by Irmak-Paris in \cite{IrP1} and some techniques given by Aramayona-Leininger in \cite{AL2}. 

In \cite{H2} Hern\'andez proved that if $S_1$ and $S_2$ are orientable surfaces of finite topological type such that $S_1$ has genus at least 3 and the complexity of $S_1$ is an upper bound of the complexity of $S_2$, and $\theta : \mathcal{C}(S_1) \rightarrow \mathcal{C}(S_2)$ is an edge-preserving map, then $S_1$ is homeomorphic to $S_2$ and $\theta$ is induced by a homeomorphism. In \cite{Ir4} the author gave a new proof of this result for edge preserving maps of $\mathcal{C}(R)$ when $g \geq 2$, $n \geq 0$ by first proving the result on the nonseparating curve graph. 
Since superinjective simplicial maps are edge preserving this improved the results of the author given in \cite{Ir1}, \cite{Ir2}, \cite{Ir3}. We 
also note that edge preserving maps of the curve graphs were used to get information about the maps of Hatcher-Thurston graphs, see \cite{H3}, \cite{Ir4}. Automorphisms of the Hatcher-Thurston complex were classified by Irmak-Korkmaz in \cite{IrK}.
 
In this paper the author proves the remaining cases on the edge preserving maps of the curve graphs when $g=0, n\geq 5$ or $g=1, n \geq 3$. We note that 
when $g=0$ and $n \in \{1, 2, 3\}$ the curve graph is empty. For the other cases, when $g=0, n =4$ or $g=1, n \in \{0, 1, 2\}$ the statement is not true. When $g=0, n =4$ or $g=1, n \in \{0, 1\}$, the curve graph is represented by the Farey graph (see Figure \ref{fig-0}) (by putting edges between vertices that have geometric intersection two in $g=0, n =4$ case, and by putting edges between vertices that have geometric intersection one in the other two cases). 
It is easy to see that there are edge preserving maps of the Farey graph that are not induced by homeomorphisms of the corresponding surfaces in these cases. When $g=1, n=2$, the curve graph is isomorphic to the curve graph of the surface $M$ with $g=0, n=5$, see Lemma 2.1 given by Luo in \cite{L}. There are automorphisms of the curve graph of $M$ switching vertices that correspond to nonseparating and separating curves on the surface with $g=1, n=2$. So, the statement is not true for $g=1, n=2$.
      
\begin{figure} 
	\begin{center}    
		
		\epsfxsize=2.75in \epsfbox{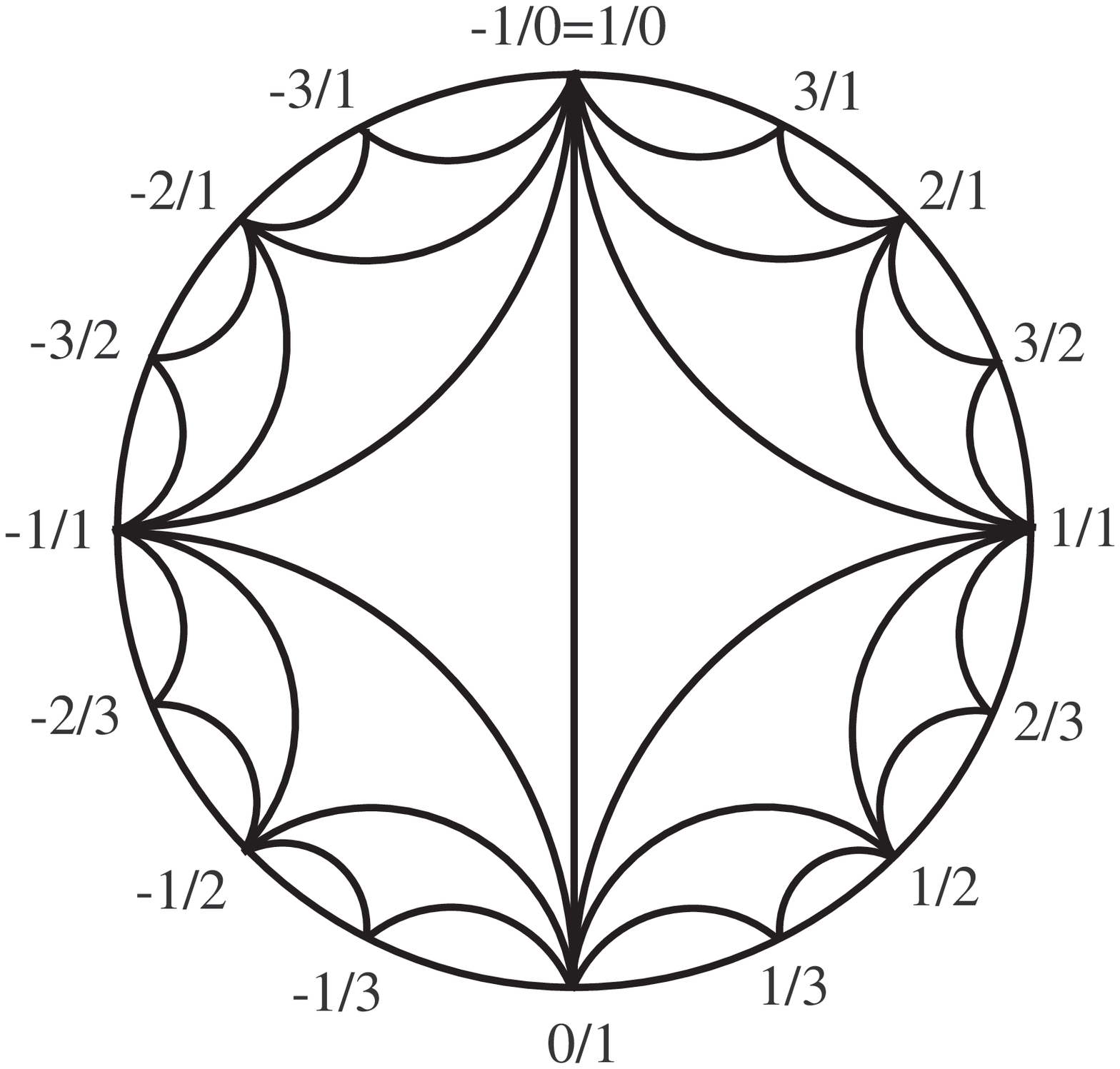} 
		\hspace{0.2cm}  	\epsfxsize=1.9in \epsfbox{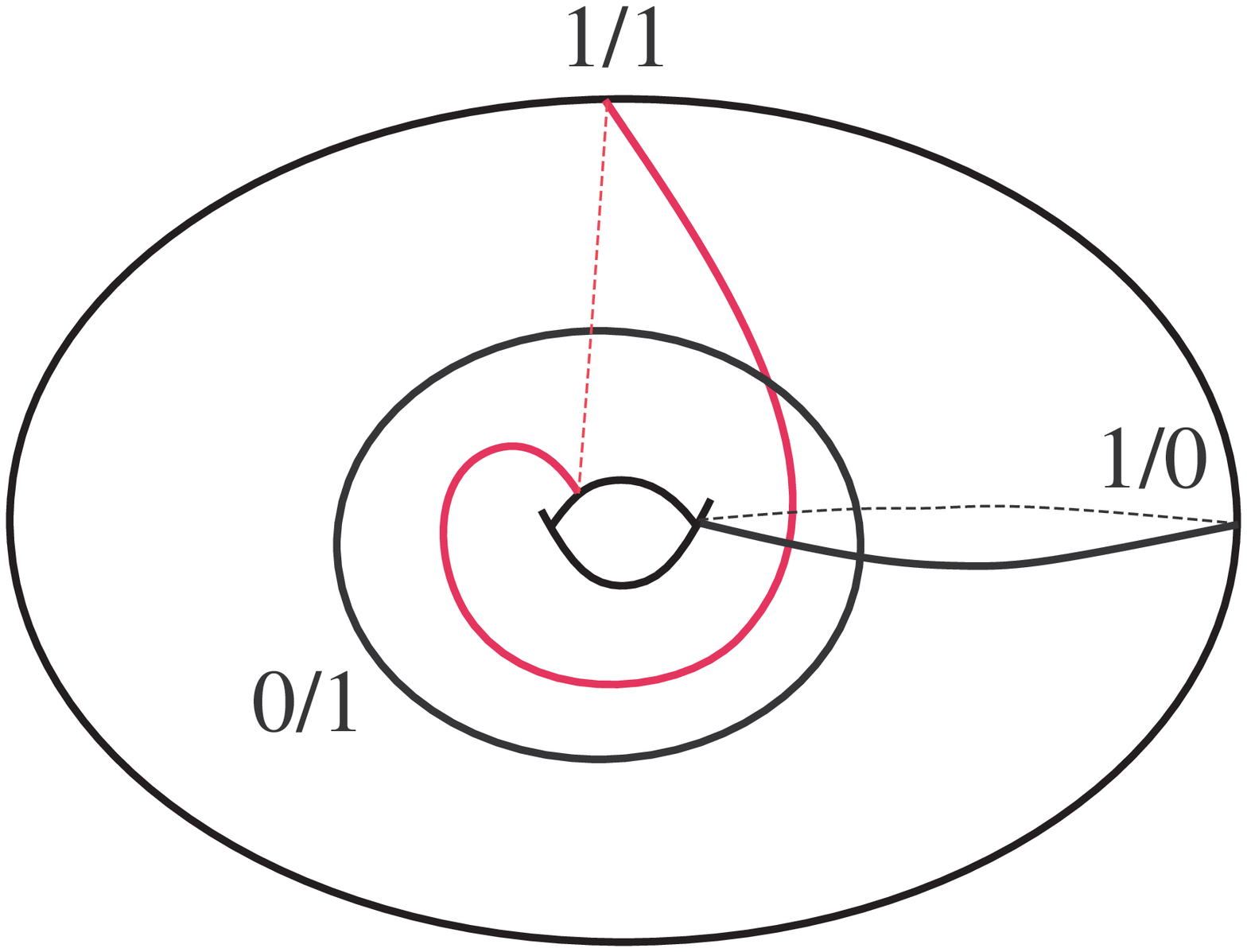} 
		\caption{Farey graph} \label{fig-0}
	\end{center} 
\end{figure}

\section{Edge Preserving Maps of $\mathcal{C}(R)$ when $g=1, n \geq 3$}

In this section we will always assume that $g=1$, $n \geq 3$ and $\lambda :\mathcal{C}(R) \rightarrow \mathcal{C}(R)$ is an edge preserving map. 

We first give some definitions. 
Let $P$ be a set of pairwise disjoint nontrivial simple closed curves on $R$. The set $P$ is
called a pair of pants decomposition of $R$, if $R_P$ (the surface obtained from $R$ by cutting along $P$) is disjoint union of genus zero surfaces with three boundary components, pairs of pants. A pair of pants of a pants
decomposition is the image of one of these connected components under the quotient map $q:R_P \rightarrow R$. Let $a$ and $b$ be two distinct elements in a pair of pants decomposition $P$ on $R$. Then $a$ is called {\it adjacent} to $b$ w.r.t. $P$ iff there exists a pair of pants in $P$ which has $a$ and $b$ on its boundary. 
 
\begin{lemma}
	\label{inj-1} The map $\lambda$ is injective on every set of vertices in $\mathcal{C}(R)$ if each pair in the set has geometric intersection zero.
\end{lemma}
 
\begin{proof} Let $\mathcal{A}$ be a set of vertices in $\mathcal{C}(R)$ such that each pair in the set has geometric 
	intersection zero. Let $\alpha$ and $\beta$ be distinct elements in $\mathcal{A}$. Since $i(\alpha, \beta) =0$ there is an edge between 
	$\alpha$ and $\beta$. Since $\lambda$ is edge preserving, there is an edge between $\lambda(\alpha)$ and $\lambda(\beta)$, so $\lambda(\alpha) \neq \lambda(\beta)$.\end{proof}

\begin{lemma} \label{pd-inj-1} Let $P$ be a pants decomposition on $R$. 
	A set of pairwise disjoint representatives of $\lambda([P])$ is a pants decomposition on $R$.\end{lemma}

\begin{proof} The proof follows from Lemma \ref{inj-1}.\end{proof} 
 
 \begin{figure} 
 	\begin{center}    
 		
 		\epsfxsize=2.2in \epsfbox{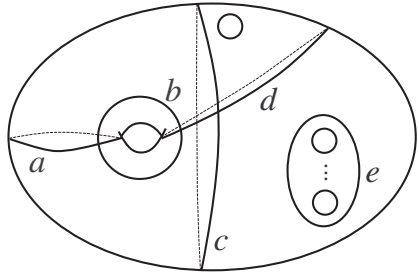} 
 		\hspace{0.2cm}  \epsfxsize=2.2in  
 		\caption{Intersection one} \label{fig-44}
 	\end{center} 
 \end{figure}
	
\begin{lemma}
	\label{22}   Let $\alpha_1, \alpha_2$ be two vertices of 
	$\mathcal{C}(R)$. If $i(\alpha_1, \alpha_2) = 1$, then 
	 $i(\lambda(\alpha_1), \lambda(\alpha_2)) \neq 0$.\end{lemma}

\begin{proof} Let $a, b$ be minimally intersecting representatives of $\alpha_1, \alpha_2$ respectively. We complete $a, b$ to a curve configuration $\{a, b, c, d, e\}$ as shown in Figure \ref{fig-44}. Then we complete $\{a, c, e\}$ to a pants decomposition $P$ on $R$. Let $P'$ be a set of pairwise disjoint representatives of $\lambda([P])$. The set $P'$ is a pants decomposition on $R$. We see that $i([b], [x]) = 0$ for all $x \in P \setminus \{a\}$ and there is an edge between $[b]$ and $[x]$ for all $x \in P \setminus \{a\}$. Since 
$\lambda$ is edge preserving we have $i(\lambda([b]), \lambda([x])) = 0$ for all $x \in P \setminus \{a\}$ and there is an edge between $\lambda([b])$ and $\lambda([x])$ for all $x \in P \setminus \{a\}$. This implies that either $i(\lambda([a]), \lambda([b])) \neq 0$ or $\lambda([a]) = \lambda([b])$. With a similar argument we can see that either $i(\lambda([d]), \lambda([b])) \neq 0$ or 
$\lambda([d]) = \lambda([b])$. If $\lambda([a]) = \lambda([b])$ then we couldn't have $i(\lambda([d]), \lambda([b])) \neq 0$ and
$\lambda([d]) = \lambda([b])$ since $\lambda$ is edge preserving. Hence, $i(\lambda([a]), \lambda([b])) \neq 0$.\end{proof}
 
\begin{figure} 
	\begin{center}	  	  
		\epsfxsize=2.3in \epsfbox{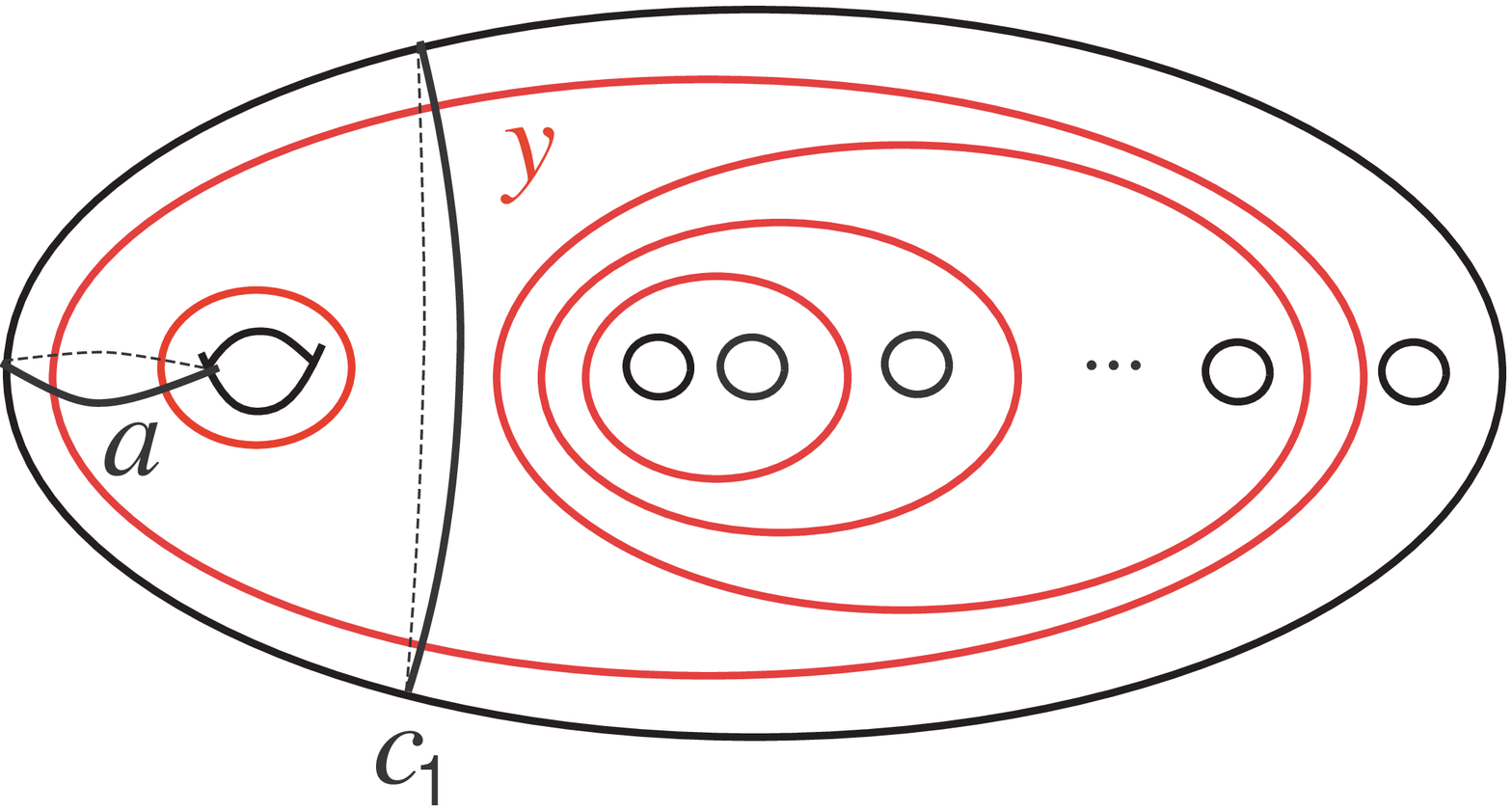}  \hspace{0.2cm} 
		\epsfxsize=2.3in \epsfbox{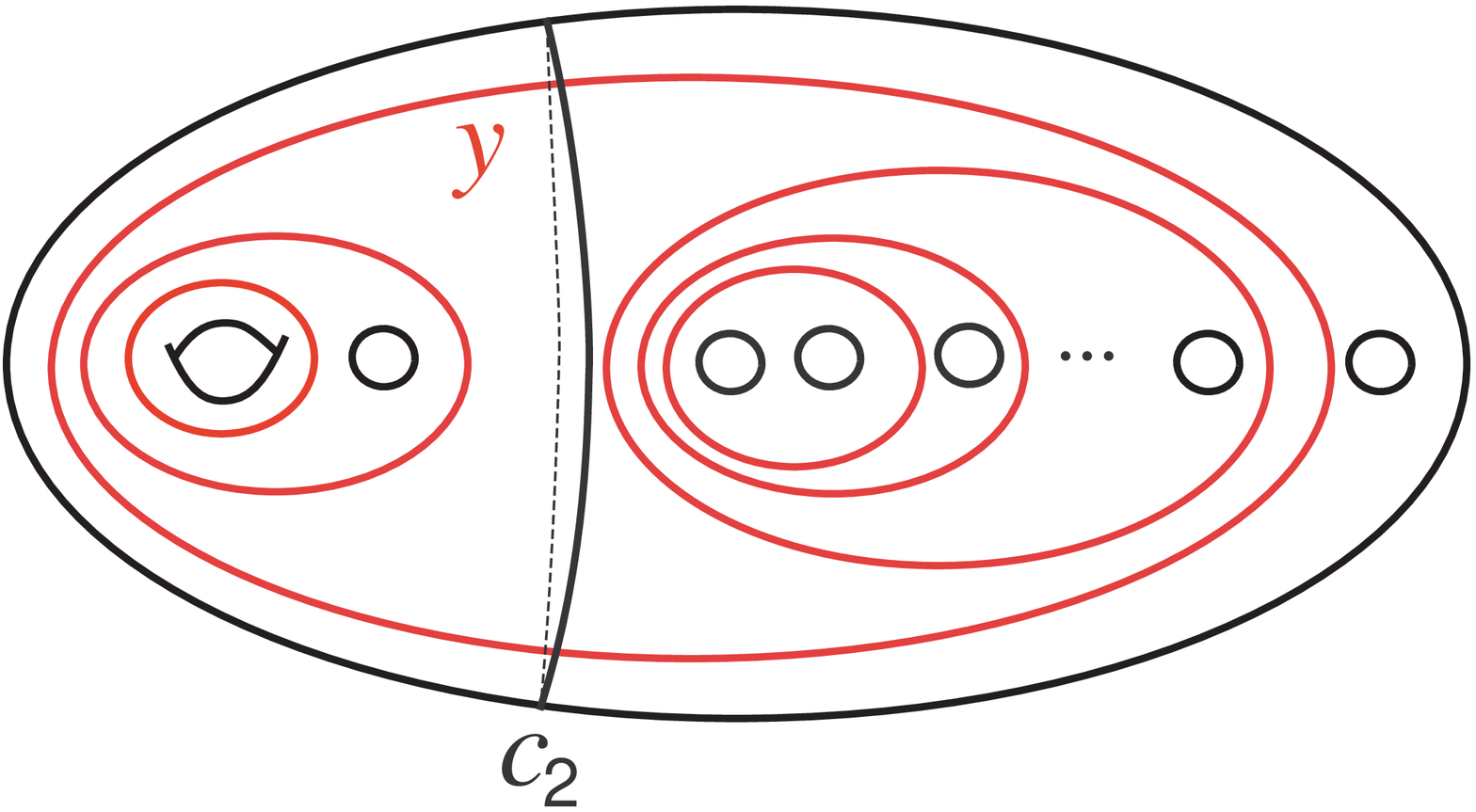}
		
		(i)   \hspace{5.6cm} (ii)  \vspace{0.3cm}
		
		\epsfxsize=2.3in \epsfbox{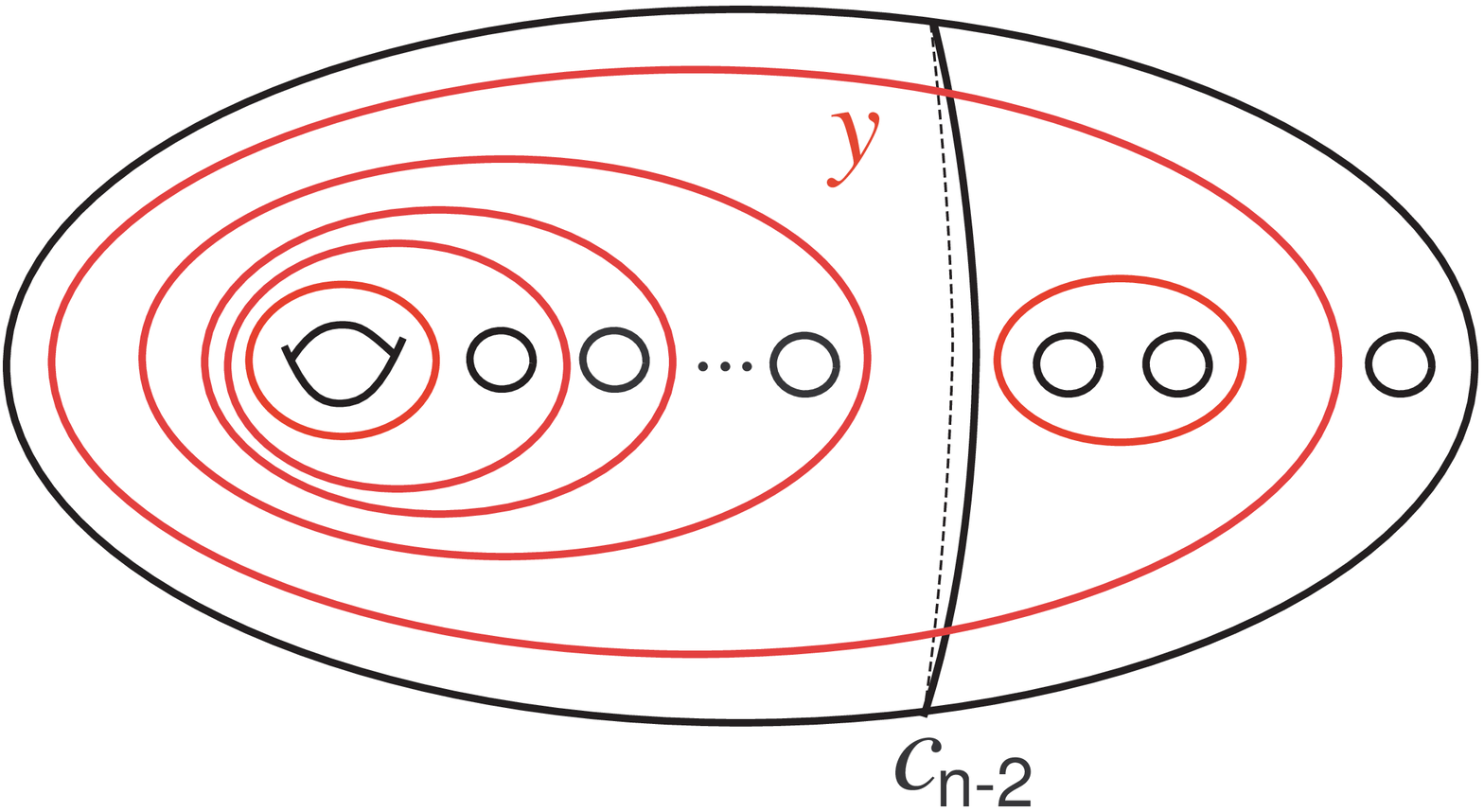}  \hspace{0.2cm} 
		\epsfxsize=2.3in \epsfbox{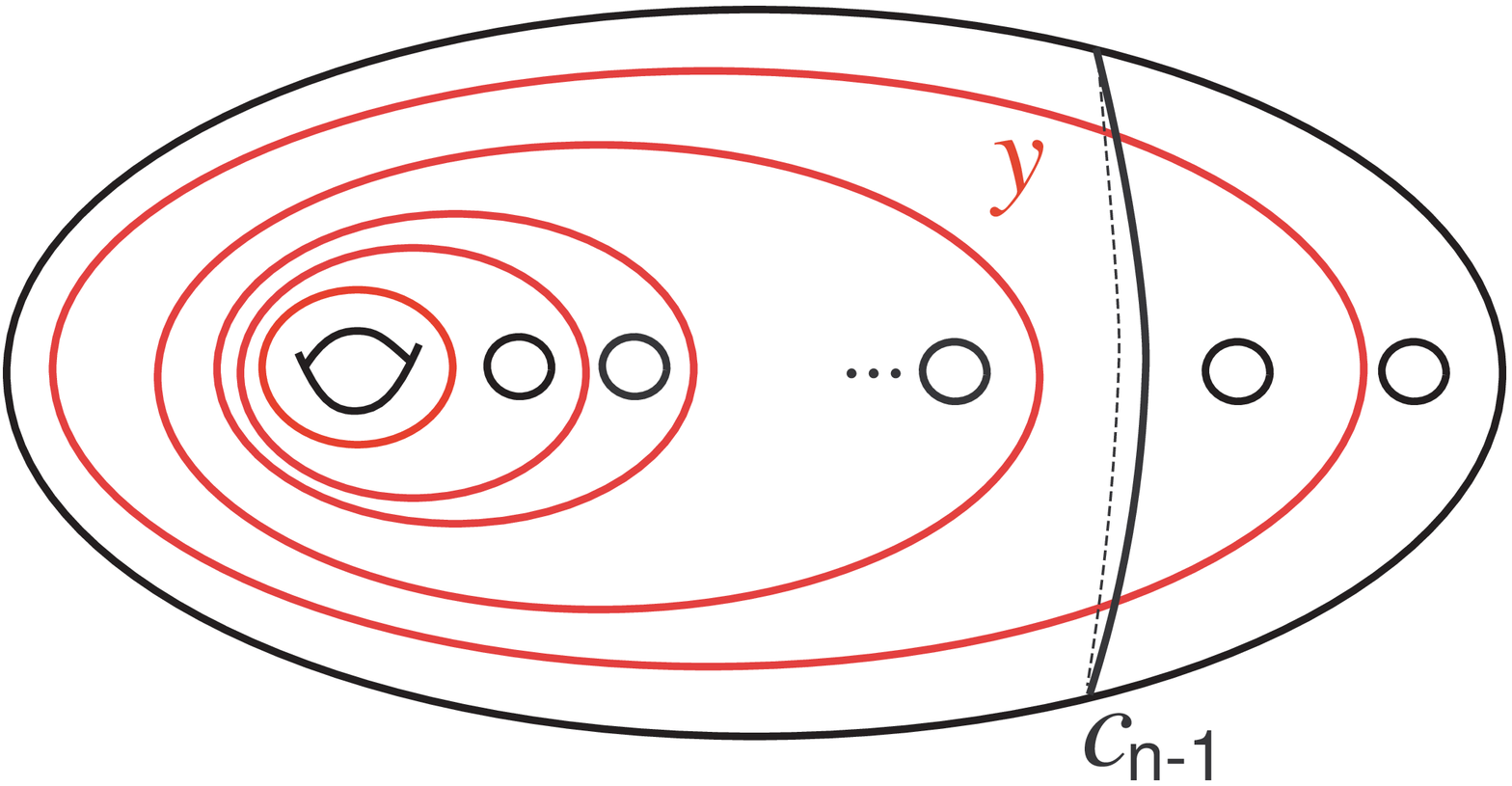} 
		
		(iii)   \hspace{5.6cm} (iv)  \vspace{0.3cm}
		
			\epsfxsize=2.3in \epsfbox{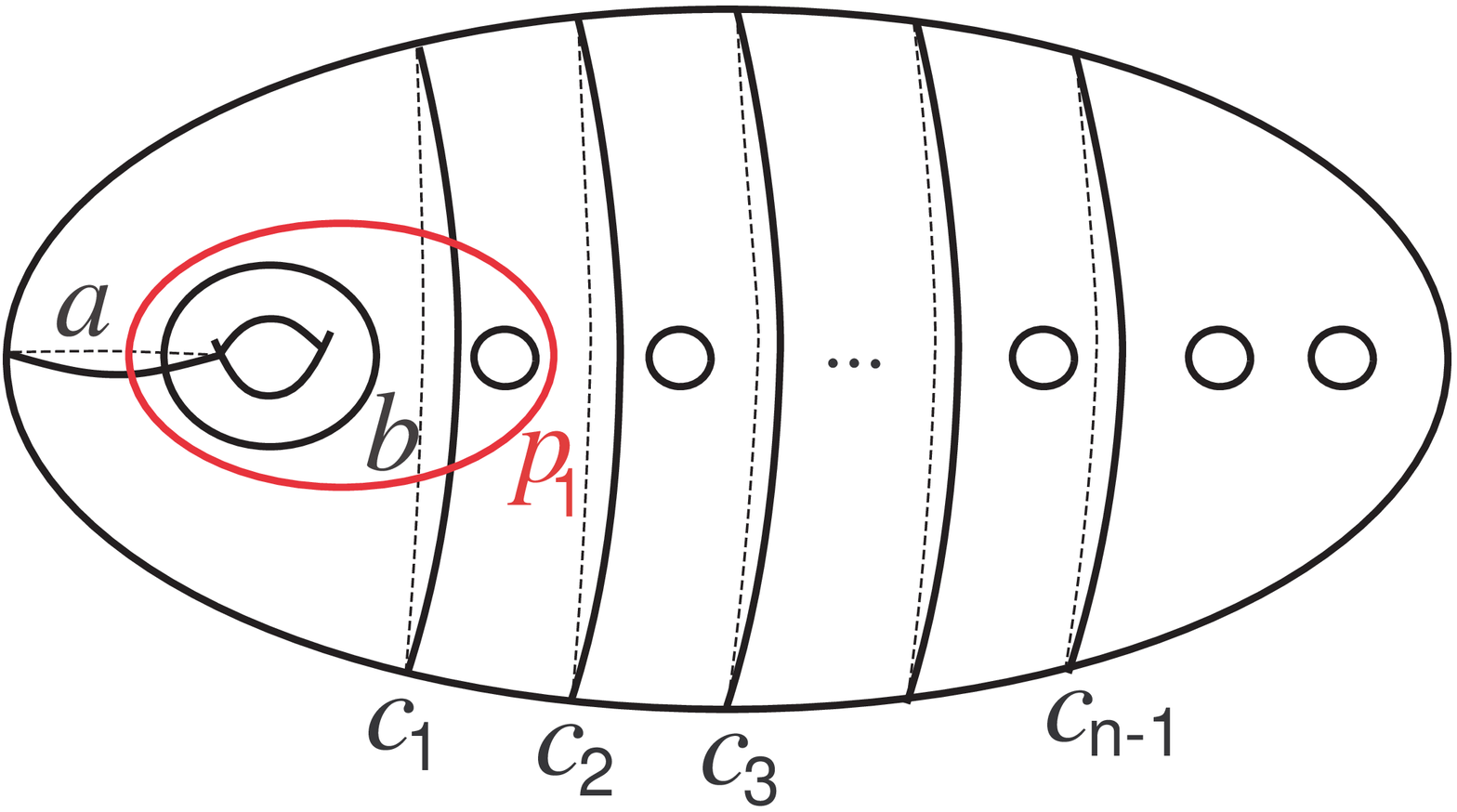}  \hspace{0.2cm} 
		\epsfxsize=2.3in \epsfbox{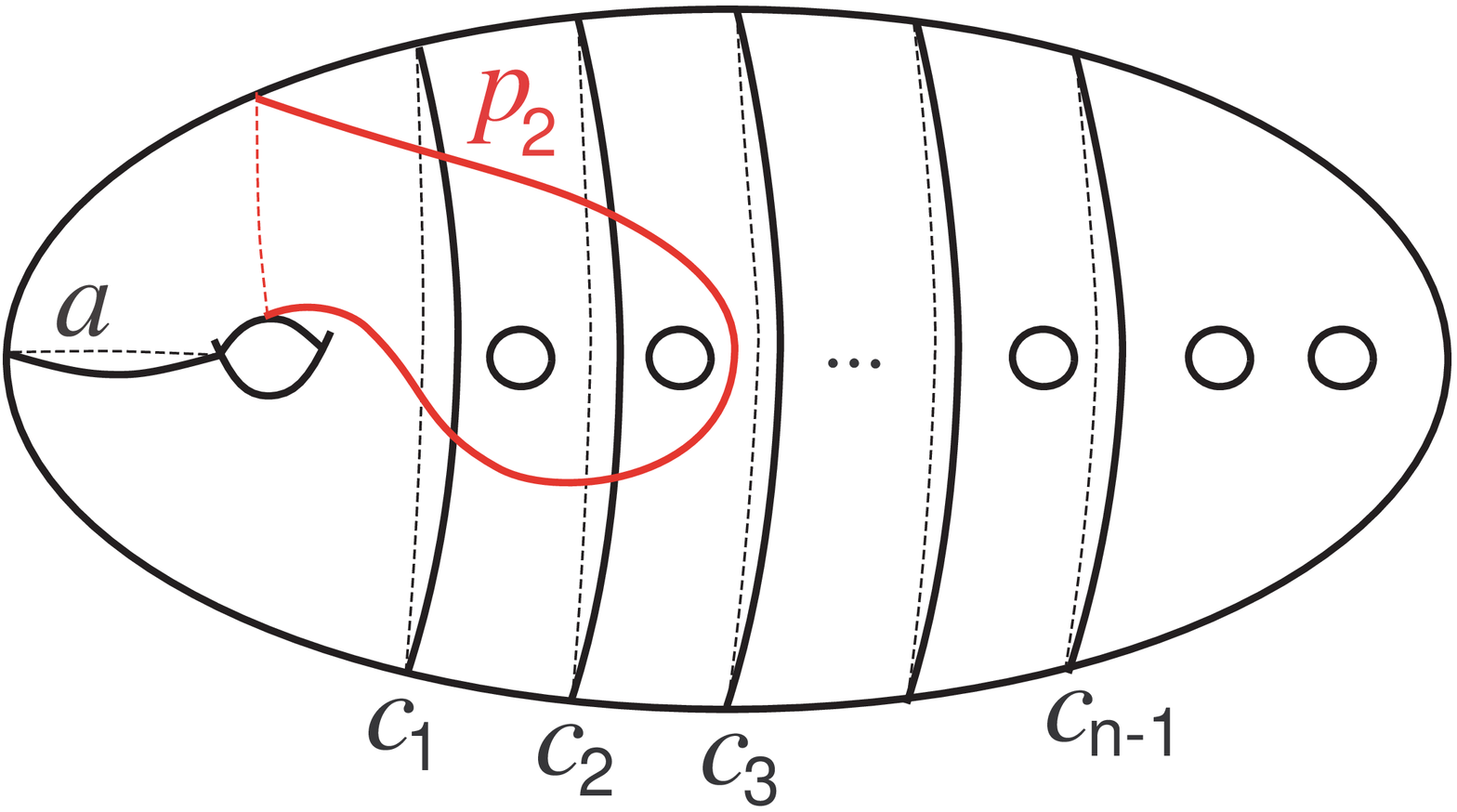}  
		
		(v)   \hspace{5.6cm} (vi)  \vspace{0.3cm}
		
		\epsfxsize=2.3in \epsfbox{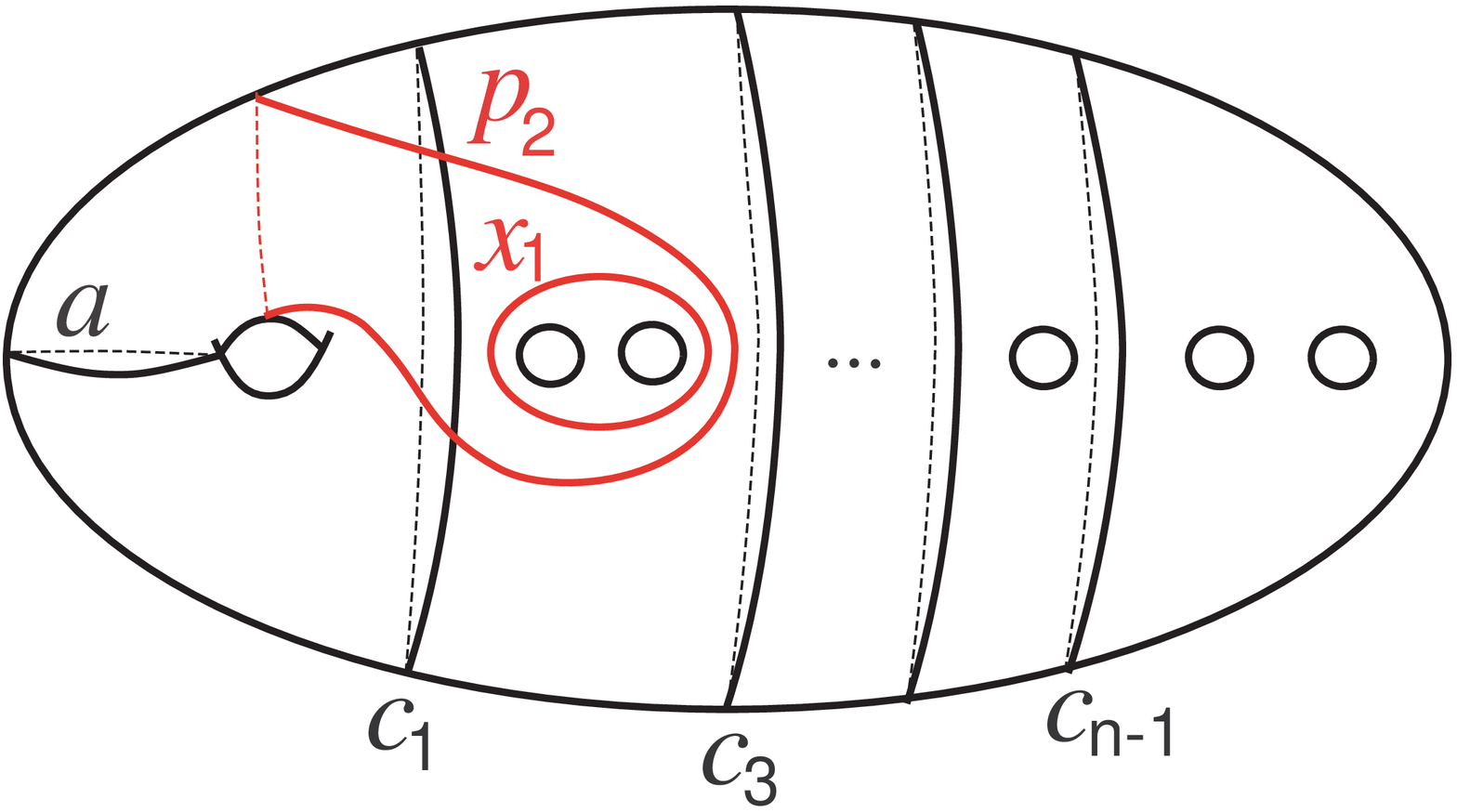}  \hspace{0.2cm} 
		\epsfxsize=2.3in \epsfbox{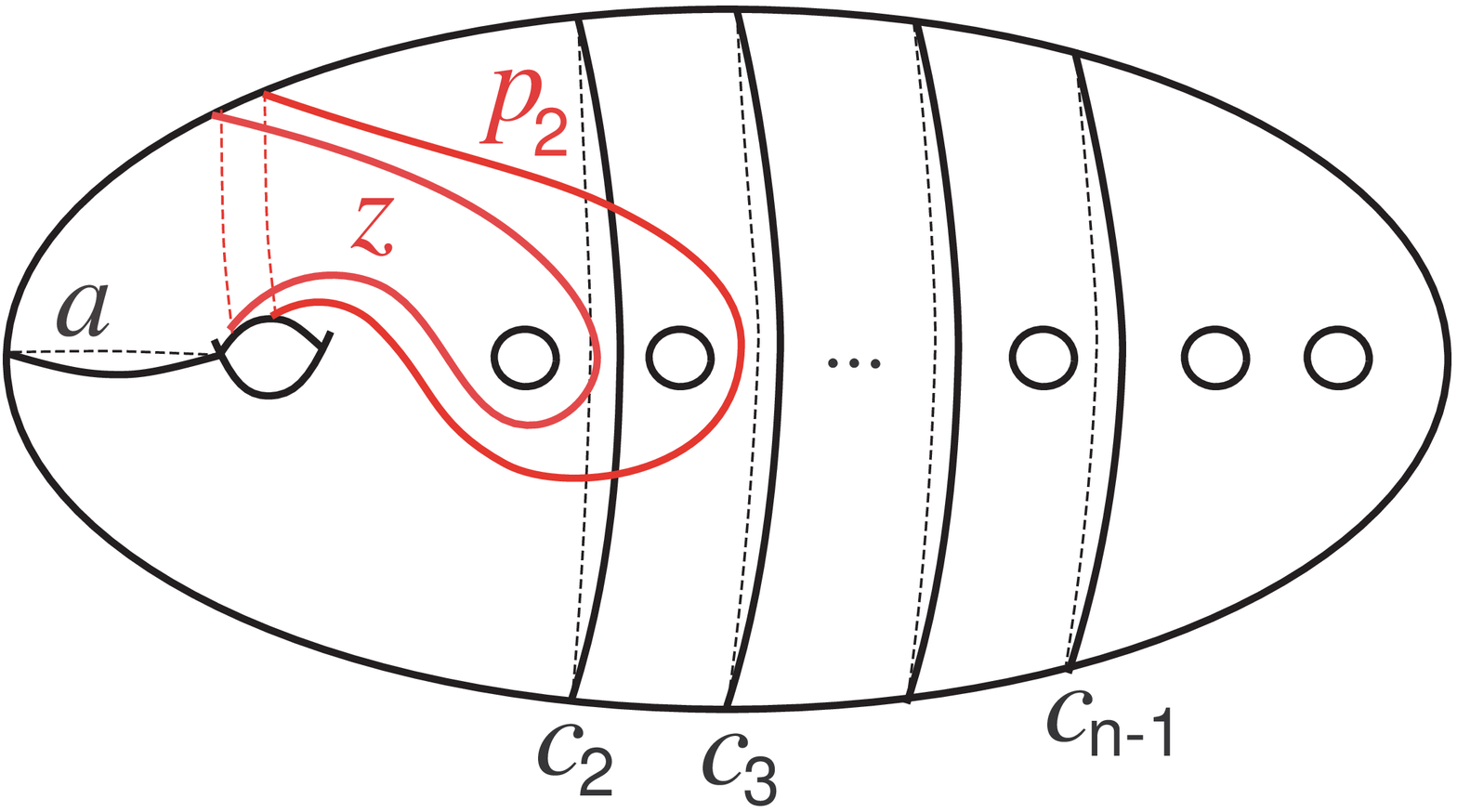}
		
		(vii)   \hspace{5.6cm} (viii)  \vspace{0.3cm}
		
		\caption{adjacency} \label{fig-5s} 
	\end{center} 
\end{figure}

\begin{lemma}
	\label{intersection} Let $\{y, c_1, c_2, \cdots c_{n-1}\}$ be the curves shown in Figure \ref{fig-5s}. Then we have $i(\lambda([y]), \lambda([c_i])) \neq 0$ for all $i= 1, 2, \cdots n-1$.\end{lemma}

\begin{proof} To see that $i(\lambda([y]), \lambda([c_1])) \neq 0$, we complete $y$ to a pants decomposition $P$ using all the red curves given in Figure \ref{fig-5s} (i). Let $P'$ be a set of pairwise disjoint representatives of $\lambda([P])$. The set $P'$ is a pants decomposition on $R$. We see that $i([c_1], [x]) = 0$ for all $x \in P \setminus \{y\}$ and there is an edge between $[c_1]$ and $[x]$ for all $x \in P \setminus \{y\}$. Since $\lambda$ is edge preserving we have $i(\lambda([c_1]), \lambda([x])) = 0$ for all $x \in P \setminus \{y\}$ and there is an edge between $\lambda([c_1])$ and $\lambda([x])$ for all $x \in P \setminus \{y\}$. This implies that either $i(\lambda([c_1]), \lambda([y])) \neq 0$ or $\lambda([c_1]) = \lambda([y])$. Let $a$ be the curve shown in Figure \ref{fig-5s} (i). Since $i([a]), [y]) =1$, by Lemma \ref{22} we know $i(\lambda([a]), \lambda([y])) \neq 0$. But since $i([a]), [c_1]) =0$, we have $i(\lambda([a]), \lambda([c_1])) = 0$. So, $\lambda([c_1])$ cannot be equal to $\lambda([y])$. Hence,  $i(\lambda([y]), \lambda([c_1])) \neq 0$. With similar arguments we see that $i(\lambda([y]), \lambda([c_i])) \neq 0$ for all $i = 2, 3, \cdots, n-1$ (see Figures \ref{fig-5s} (ii)-(iv).) \end{proof} 	

\begin{lemma} \label{adj-2} Let $P = \{a, c_1, c_2, c_3, \cdots, c_{n-1}\}$ 
	where the curves are as shown in Figure \ref{fig-5s}. Let $P'$ be a pair of pants
	decomposition of $R$ such that $\lambda([P]) = [P']$. If $x, y \in P$ and $x$ is adjacent to $y$ w.r.t. $P$, then $\lambda([x]), \lambda([y])$ have representatives in $P'$ which are adjacent to each other w.r.t. $P'$.\end{lemma}

\begin{proof} We see that $a$ is adjacent to $c_1$ w.r.t. $P$. To see that 
 $\lambda([a]), \lambda([c_1])$ have representatives in $P'$ which are adjacent to each other w.r.t. $P'$ it is enough to find a curve $p_1$ shown in Figure \ref{fig-5s} (v) which intersects only $a$ and $c_1$ and not any other curve in $P$ and control that $i(\lambda([p_1]), \lambda([a])) \neq 0$, $i(\lambda([p_1]), \lambda([c_1])) \neq 0$ and $i(\lambda([p_1]), \lambda([x])) = 0$ for every $x \in P \setminus \{a, c_1\}$. Since $a$ and 
 $p_1$ have geometric intersection one, by using Lemma \ref{22} we see that $i(\lambda([p_1]), \lambda([a])) \neq 0$. Since $\lambda$ is edge preserving
 $i(\lambda([p_1]), \lambda([x])) = 0$ for every $x \in P \setminus \{a, c_1\}$. To see that $i(\lambda([p_1]), \lambda([c_1])) \neq 0$ we consider the 
 following: Let $Q= (P \setminus \{a\}) \cup \{b\}$ where the curve $b$ is as shown in Figure \ref{fig-5s} (v). Then $Q$ is a pants decomposition on $R$ and  $i(\lambda([p_1]), \lambda([x])) = 0$ for every $x \in Q \setminus \{c_1\}$.
So, either $i(\lambda([p_1]), \lambda([c_1])) \neq 0$ or $\lambda([p_1]) = \lambda([c_1])$. Since   $i([a]), [p_1]) =1$, by Lemma \ref{22}   $i(\lambda([a]), \lambda([p_1])) \neq 0$. But since $i([a]), [c_1]) =0$, we have $i(\lambda([a]), \lambda([c_1])) = 0$. So, $\lambda([p_1])$ cannot be equal to $\lambda([c_1])$. Hence, $i(\lambda([p_1]), \lambda([c_1])) \neq 0$.
This gives us that $\lambda([a]), \lambda([c_1])$ have representatives in $P'$ which are adjacent to each other w.r.t. $P'$.

To see that $\lambda([c_1]), \lambda([c_2])$ have representatives in $P'$ which are adjacent to each other w.r.t. $P'$ it is enough to find a curve $p_2$ shown in Figure \ref{fig-5s} (vi) which intersects only $c_1$ and $c_2$ and not any other curve in $P$ and control that $i(\lambda([p_2]), \lambda([c_1])) \neq 0$, $i(\lambda([p_2]), \lambda([c_2])) \neq 0$ and $i(\lambda([p_2]), \lambda([x])) = 0$ for every $x \in P \setminus \{c_1, c_2\}$. Since $\lambda$ is edge preserving
$i(\lambda([p_2]), \lambda([x])) = 0$ for every $x \in P \setminus \{c_1, c_2\}$. To see that $i(\lambda([p_2]), \lambda([c_1])) \neq 0$ we consider the 
following: Let $Q= (P \setminus \{c_2\}) \cup \{x_1\}$ where the curve $x_1$ is as shown in Figure \ref{fig-5s} (vii). Then $Q$ is a pants decomposition on $R$ and  $i(\lambda([p_2]), \lambda([x])) = 0$ for every $x \in Q \setminus \{c_1\}$. So, either $i(\lambda([p_2]), \lambda([c_1])) \neq 0$ or $\lambda([p_2]) = \lambda([c_1])$. Since   $i([b]), [p_2]) =1$, by Lemma \ref{22}   $i(\lambda([b]), \lambda([p_2])) \neq 0$. But since $i([b]), [c_1]) =0$, we have $i(\lambda([b]), \lambda([c_1])) = 0$. So, $\lambda([p_2])$ cannot be equal to $\lambda([c_1])$. Hence, $i(\lambda([p_2]), \lambda([c_1])) \neq 0$. To see that $i(\lambda([p_2]), \lambda([c_2])) \neq 0$ we consider $T= (P \setminus \{c_1\}) \cup \{z\}$ where the curve $z$ is as shown in Figure \ref{fig-5s} (viii). Then $T$ is a pants decomposition on $R$ and  $i(\lambda([p_2]), \lambda([x])) = 0$ for every $x \in T \setminus \{c_2\}$. So, either $i(\lambda([p_2]), \lambda([c_2])) \neq 0$ or $\lambda([p_2]) = \lambda([c_2])$. Since $i([b]), [p_2]) =1$, by Lemma \ref{22}   $i(\lambda([b]), \lambda([p_2])) \neq 0$. But since $i([b]), [c_2]) =0$, we have $i(\lambda([b]), \lambda([c_2])) = 0$. So, $\lambda([p_2])$ cannot be equal to $\lambda([c_2])$. Hence, $i(\lambda([p_2]), \lambda([c_2])) \neq 0$.
This gives us that $\lambda([c_1]), \lambda([c_2])$ have representatives in $P'$ which are adjacent to each other w.r.t. $P'$.

\begin{figure} 
	\begin{center}
		\epsfxsize=2.3in \epsfbox{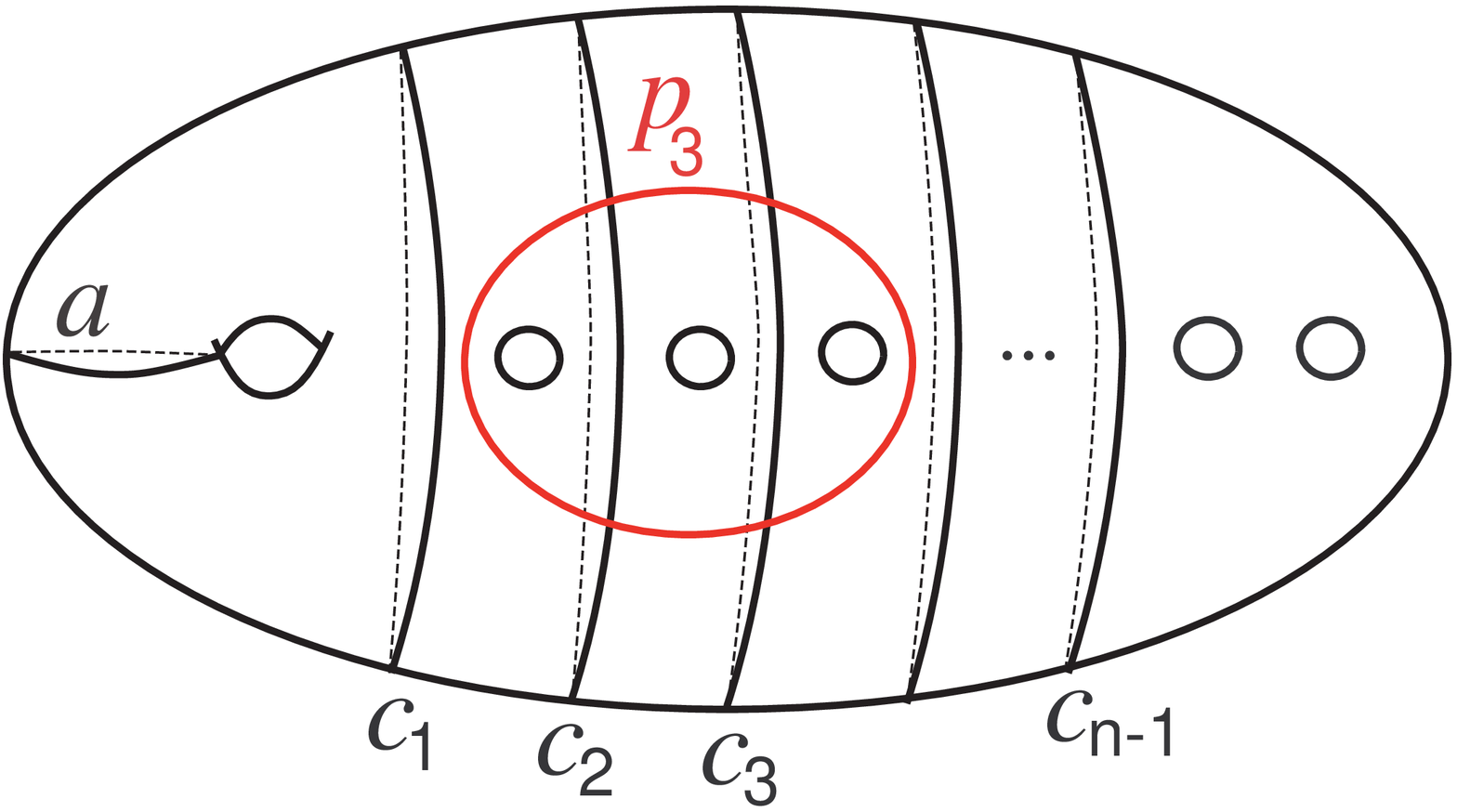}  \hspace{0.2cm} 
		\epsfxsize=2.3in \epsfbox{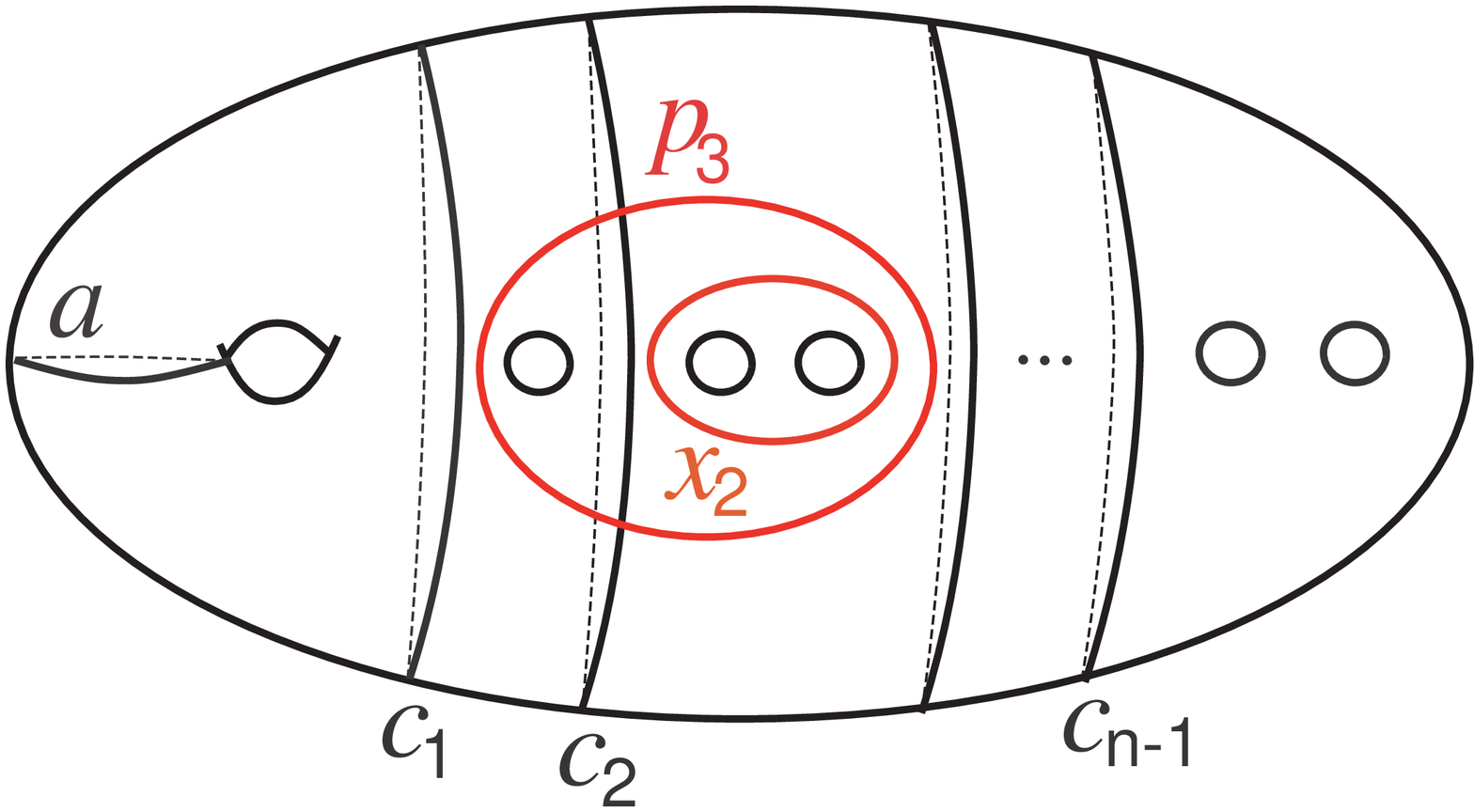}  
		
		(i)   \hspace{5.6cm} (ii)  \vspace{0.3cm}
		
		\epsfxsize=2.3in \epsfbox{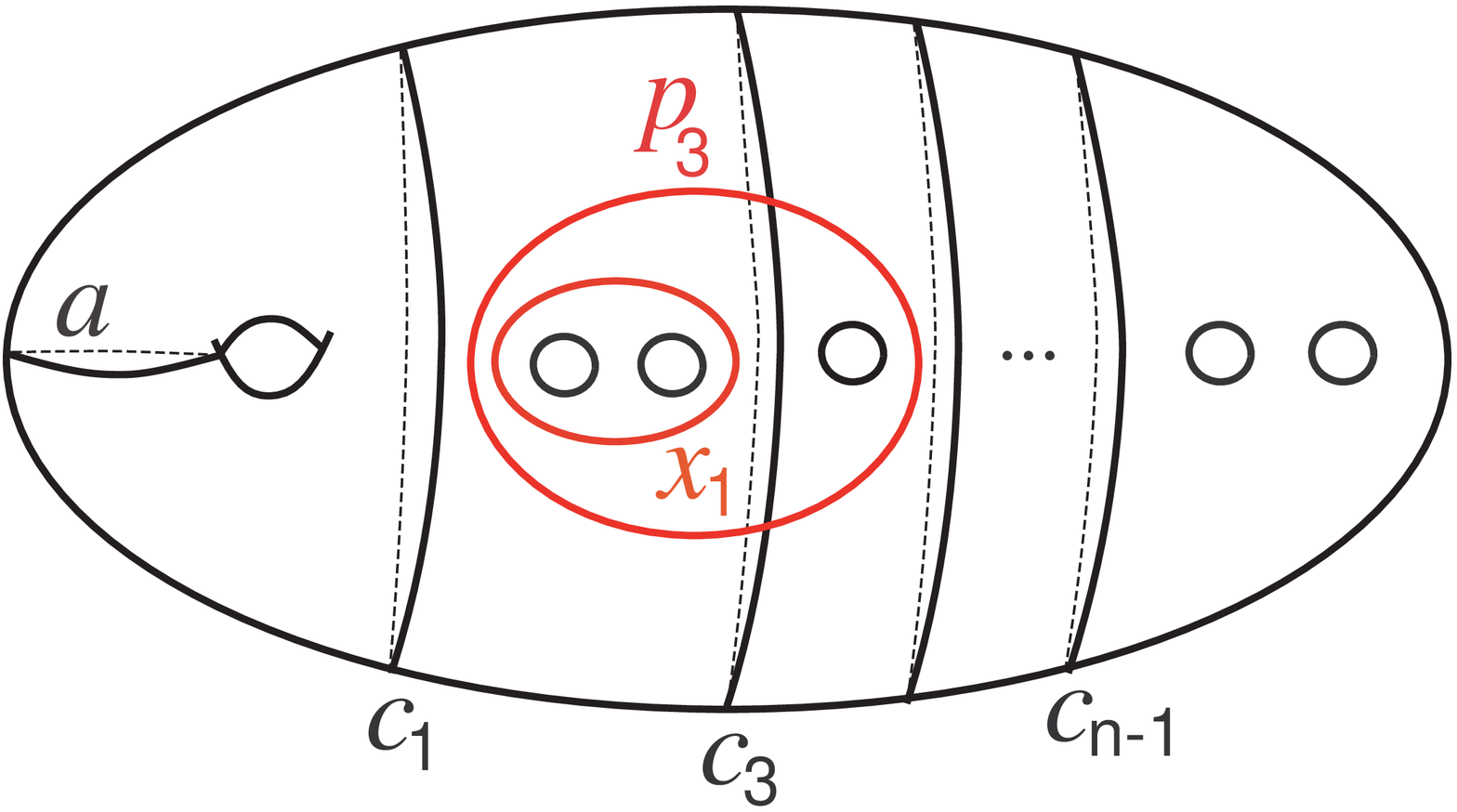}  \hspace{0.2cm} 
		\epsfxsize=2.3in \epsfbox{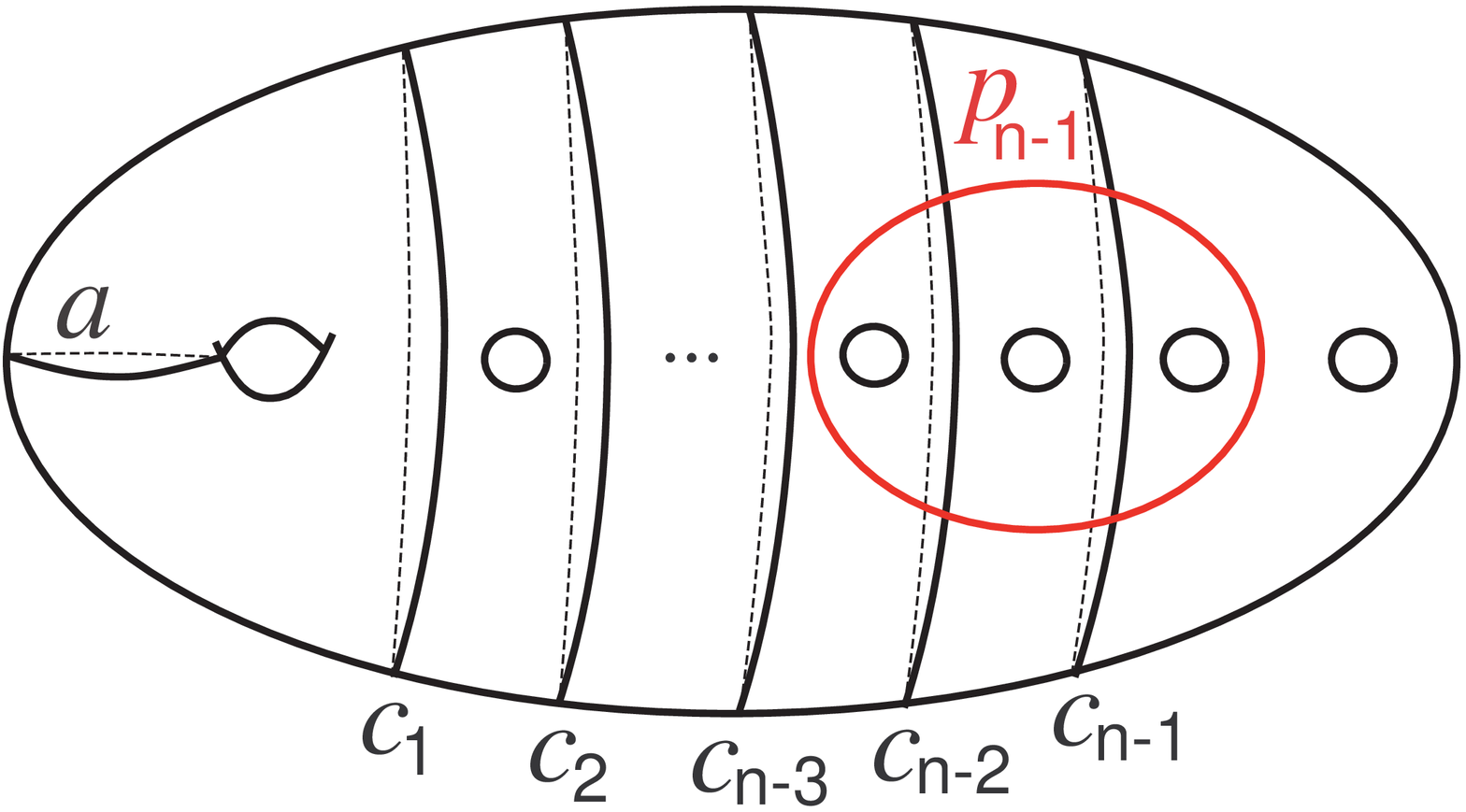}  
		
		(iii)   \hspace{5.6cm} (iv)  \vspace{0.3cm}	
		
		\epsfxsize=2.3in \epsfbox{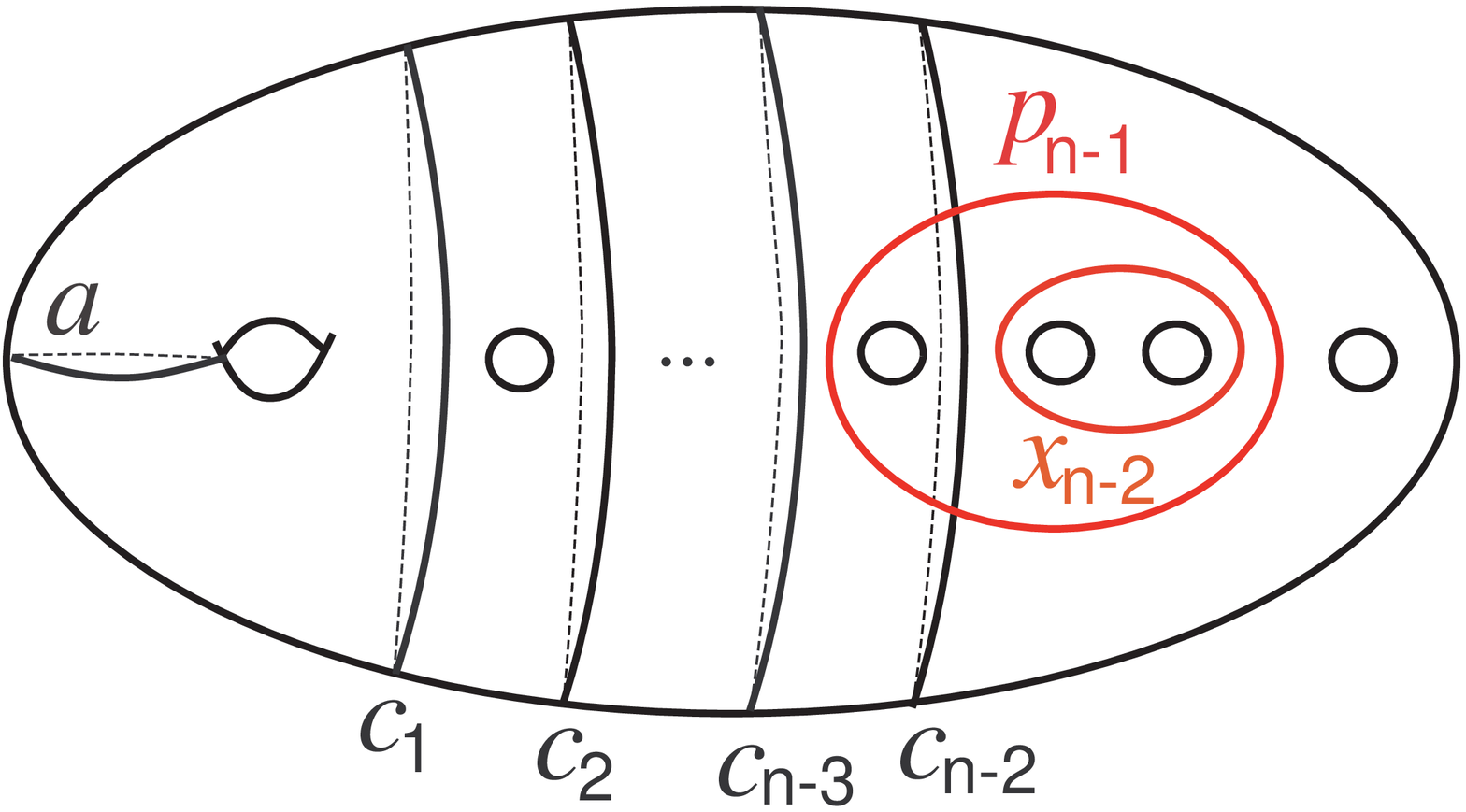}  \hspace{0.2cm} 	\epsfxsize=2.3in \epsfbox{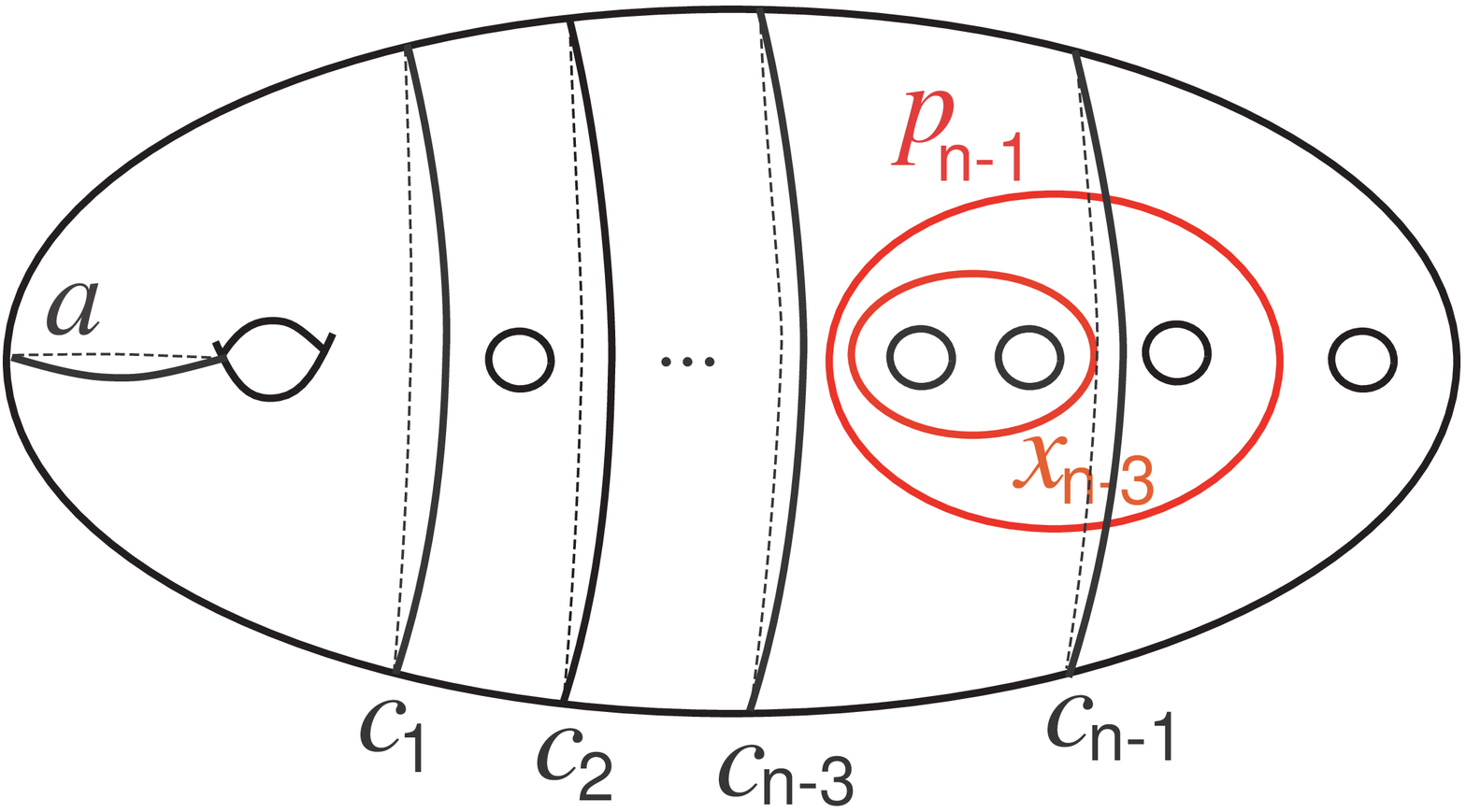}  
		
		(v)   \hspace{5.6cm} (vi)  \vspace{0.3cm}
		
		\epsfxsize=2.3in \epsfbox{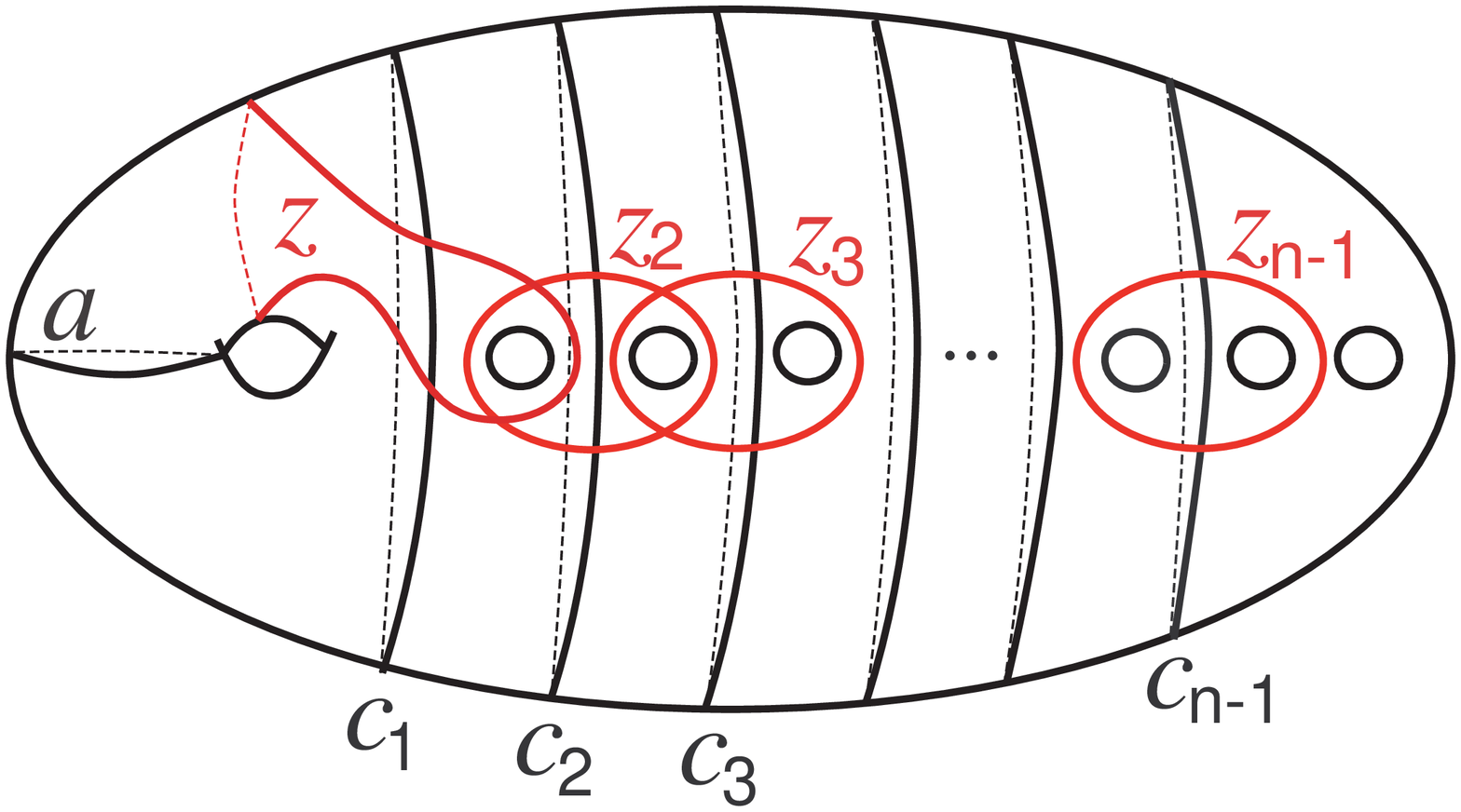}
		
		(vii)   
		
		\caption{adjacency, nonadjacency} \label{fig-6s}
	\end{center} 
\end{figure}	
 
To see that $\lambda([c_2]), \lambda([c_3])$ have representatives in $P'$ which are adjacent to each other w.r.t. $P'$ it is enough to find a curve $p_3$ shown in Figure \ref{fig-6s} (i) which intersects only $c_2$ and $c_3$ and not any other curve in $P$ and control that $i(\lambda([p_3]), \lambda([c_2])) \neq 0$, $i(\lambda([p_3]), \lambda([c_3])) \neq 0$ and $i(\lambda([p_3]), \lambda([x])) = 0$ for every $x \in P \setminus \{c_2, c_3\}$. Since $\lambda$ is edge preserving
$i(\lambda([p_3]), \lambda([x])) = 0$ for every $x \in P \setminus \{c_2, c_3\}$. To see that $i(\lambda([p_3]), \lambda([c_2])) \neq 0$ we consider 
$U= (P \setminus \{c_3\}) \cup \{x_2\}$ where the curve $x_2$ is as shown in Figure \ref{fig-6s} (ii). Then $U$ is a pants decomposition on $R$ and  $i(\lambda([p_3]), \lambda([x])) = 0$ for every $x \in U \setminus \{c_2\}$. So, either $i(\lambda([p_3]), \lambda([c_2])) \neq 0$ or $\lambda([p_3]) = \lambda([c_2])$. By Lemma \ref{intersection} $i(\lambda([y]), \lambda([c_2])) \neq 0$. Since $i(([y]), [p_3]) =0$, we have $i(\lambda([y]), \lambda([p_3])) = 0$. So, $\lambda([p_3])$ cannot be equal to $\lambda([c_2])$. Hence, $i(\lambda([p_3]), \lambda([c_2])) \neq 0$. To see that $i(\lambda([p_3]), \lambda([c_2])) \neq 0$, we consider $V= (P \setminus \{c_2\}) \cup \{x_1\}$ where the curve $x_1$ is as shown in Figure \ref{fig-6s} (iii). Then $V$ is a pants decomposition on $R$ and $i(\lambda([p_3]), \lambda([x])) = 0$ for every $x \in V \setminus \{c_3\}$. So, either $i(\lambda([p_3]), \lambda([c_3])) \neq 0$ or $\lambda([p_3]) = \lambda([c_3])$. By Lemma \ref{intersection} $i(\lambda([y]), \lambda([c_3])) \neq 0$. Since $i(([y]), [p_3]) =0$, we have $i(\lambda([y]), \lambda([p_3])) = 0$. So, $\lambda([p_3])$ cannot be equal to $\lambda([c_3])$. Hence, $i(\lambda([p_3]), \lambda([c_3])) \neq 0$. 
This gives us that $\lambda([c_2]), \lambda([c_3])$ have representatives in $P'$ which are adjacent to each other w.r.t. $P'$. Proof of $\lambda([c_i]), \lambda([c_{i+1}])$ have representatives in $P'$ which are adjacent to each other w.r.t. $P'$ for $i= 2, 3, \cdots, n-1$ is similar to the proof of this last case (see  Figure \ref{fig-6s} (iv)-(vi)).\end{proof}
  
\begin{lemma}
	\label{nonadj-2} Let $P = \{a, c_1, c_2, c_3, \cdots, c_{n-1}\}$ 
	where the curves are as shown in Figure \ref{fig-6s} (vii). Let $P'$ be a pair of pants
	decomposition of $R$ such that $\lambda([P]) = [P']$. If $x, y \in P$ and $x$ is not adjacent to $y$ w.r.t. $P$, then $\lambda([x]), \lambda([y])$ have representatives in $P'$ which are not adjacent to each other w.r.t. $P'$.\end{lemma}
 
\begin{proof} Consider the curves $z, z_i$ given in Figure \ref{fig-6s} (vii).
We will first show that $i(\lambda([c_1]), \lambda([z])) \neq 0$. We see that $i(\lambda([z]), \lambda([x])) = 0$ for every $x \in P \setminus \{c_1\}$. So, either $i(\lambda([z]), \lambda([c_1])) \neq 0$ or $\lambda([z]) = \lambda([c_1])$. Since $i([b], [z]) =1$, we have $i(\lambda([b]), \lambda([z)) \neq 0$ by Lemma \ref{22}. But since $i([b]), [c_1]) =0$, we have $i(\lambda([b]), \lambda([c_1])) = 0$. So, $\lambda([z])$ cannot be equal to $\lambda([c_1])$. Hence, $i(\lambda([z]), \lambda([c_1])) \neq 0$.
Next we will show that $i(\lambda([c_2]), \lambda([z_2])) \neq 0$. We see that $i(\lambda([z_2]), \lambda([x])) = 0$ for every $x \in P \setminus \{c_2\}$. So, either $i(\lambda([z_2]), \lambda([c_2])) \neq 0$ or $\lambda([z_2]) = \lambda([c_2])$. We have $i(\lambda([y]), \lambda([c_2])) \neq 0$ by Lemma \ref{intersection}. But since $i([y]), [z_2]) =0$, we have $i(\lambda([y]), \lambda([z_2])) = 0$. So, $\lambda([z_2])$ cannot be equal to $\lambda([c_2])$. Hence, $i(\lambda([z_2]), \lambda([c_2])) \neq 0$. Similarly 
we see that $i(\lambda([z_i]), \lambda([c_i])) \neq 0$ for all $i = 3, 4, \cdots, n-1$. 

To see that if $x, y \in P$ and $x$ is not adjacent to $y$ w.r.t. $P$, then $\lambda([x]), \lambda([y])$ have representatives in $P'$ which are not adjacent to each other w.r.t. $P'$, it is enough to find two disjoint curves $w, t$ such that $w$ intersects only $x$ nontrivially and not the other curves in $P$, $t$ intersects only $y$ nontrivially and not the other curves in $P$ and $i(\lambda([w]), \lambda([x])) \neq 0$, $i(\lambda([t]), \lambda([y])) \neq 0$,  $i(\lambda([w]), \lambda([q])) = 0$, for all $q \in P \setminus \{ x\}$, $i(\lambda([t]), \lambda([q])) = 0$, for all $q \in P \setminus \{y\}$,
$i(\lambda([t]), \lambda([w])) = 0$. For the pair $a, c_i$, when $i = 2, 3, \cdots, n-1$, the curves $b, z_i$ 
would satisfy it, where the curve $b$ is as shown in Figure \ref{fig-5s} (v). For the pair $c_1, c_i$, when $i = 3, 4, \cdots, n-1$, the curves $z, z_i$ 
would satisfy it. For the pair $c_2, c_i$, when $i = 4, 5, \cdots, n-1$,  the curves $z_2, z_i$ 
would satisfy it. Similarly, we see that nonadjacency is preserved for every nonadjacent pair in $P$.\end{proof}

\begin{figure} 
	\begin{center}    
				\epsfxsize=2.2in \epsfbox{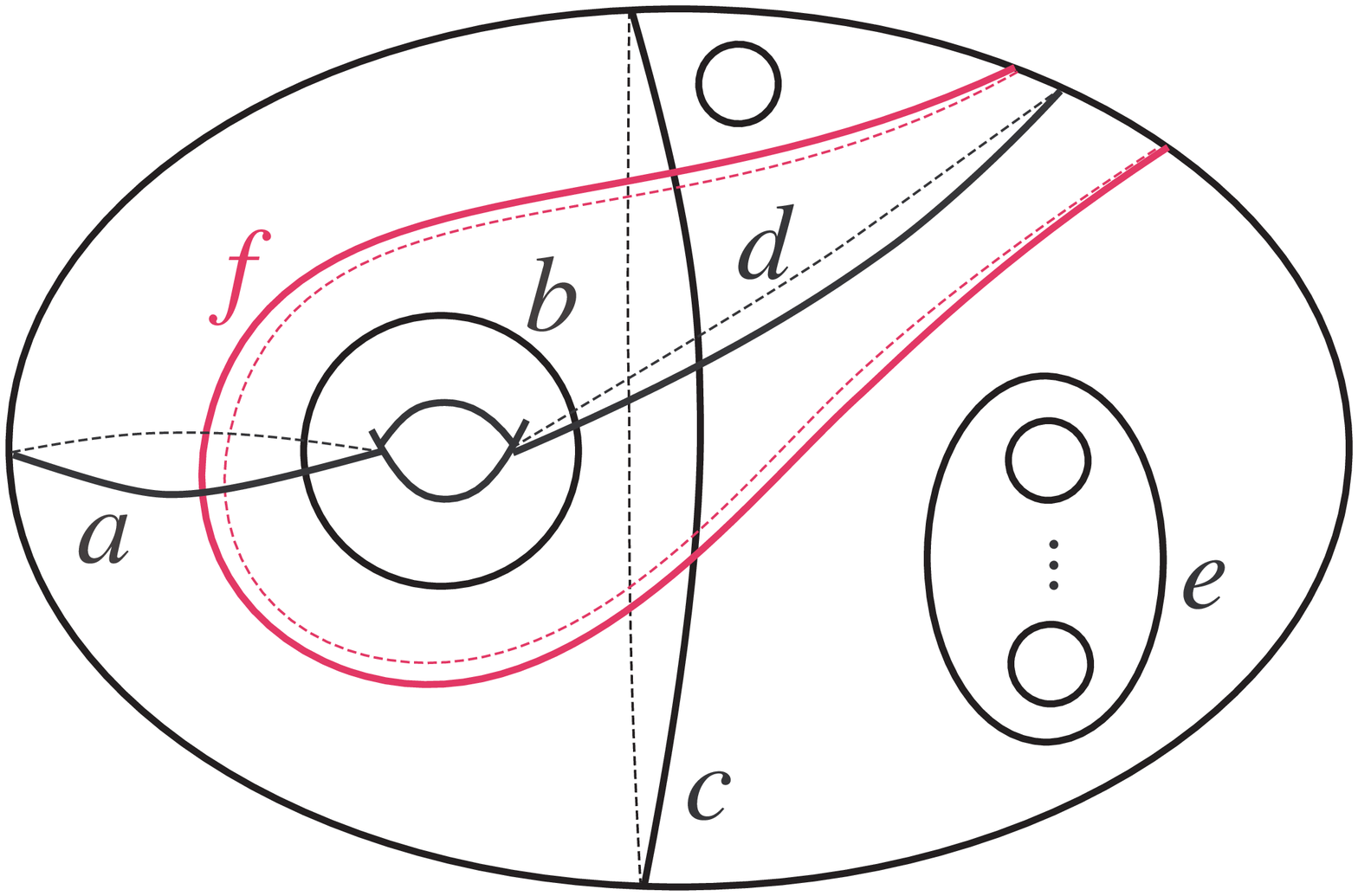} 
		\hspace{0.2cm}  \epsfxsize=2.2in \epsfbox{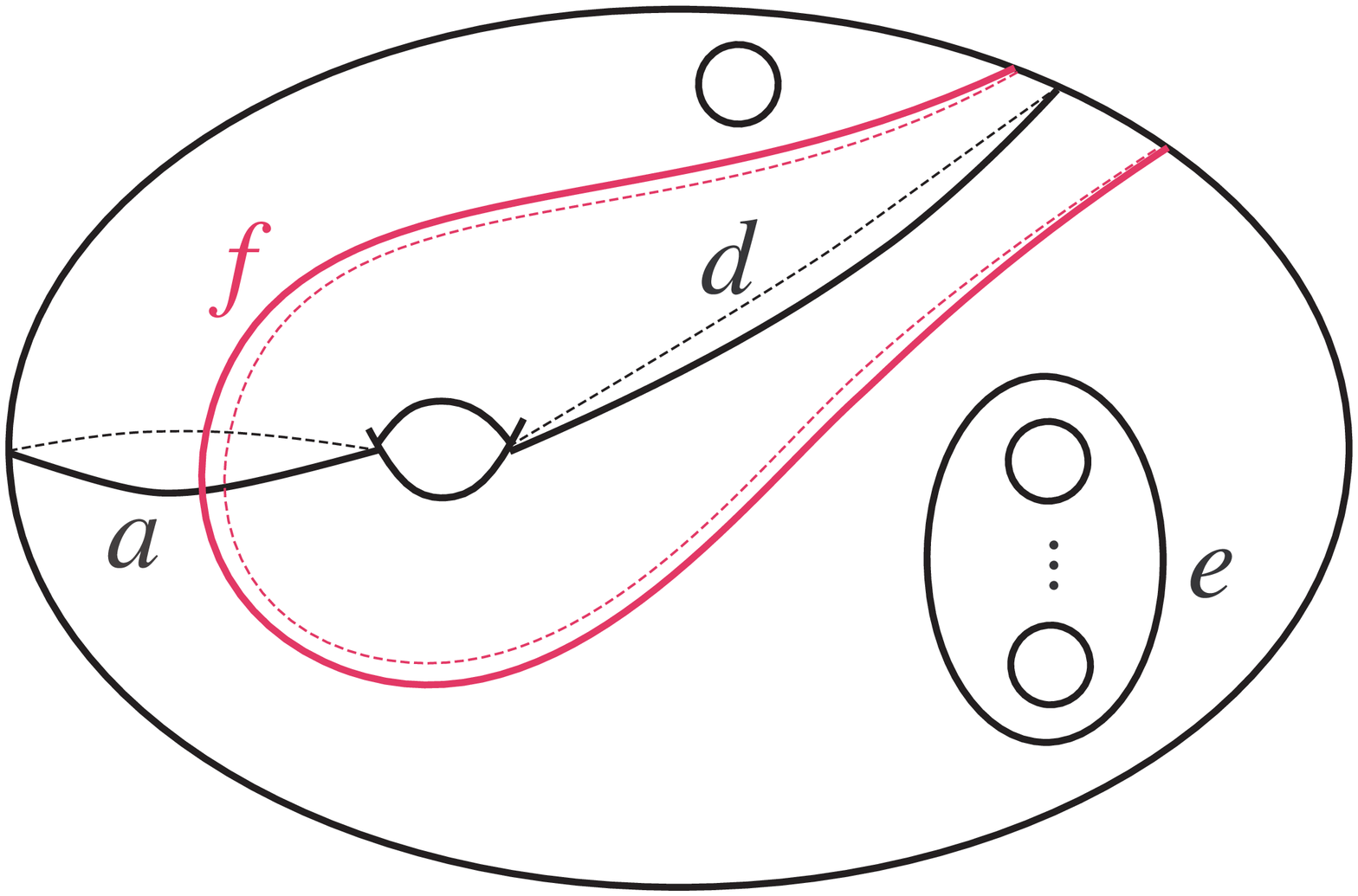}
		
		(i)   \hspace{5.6cm} (ii)  \vspace{0.3cm}
		
		\epsfxsize=2.2in \epsfbox{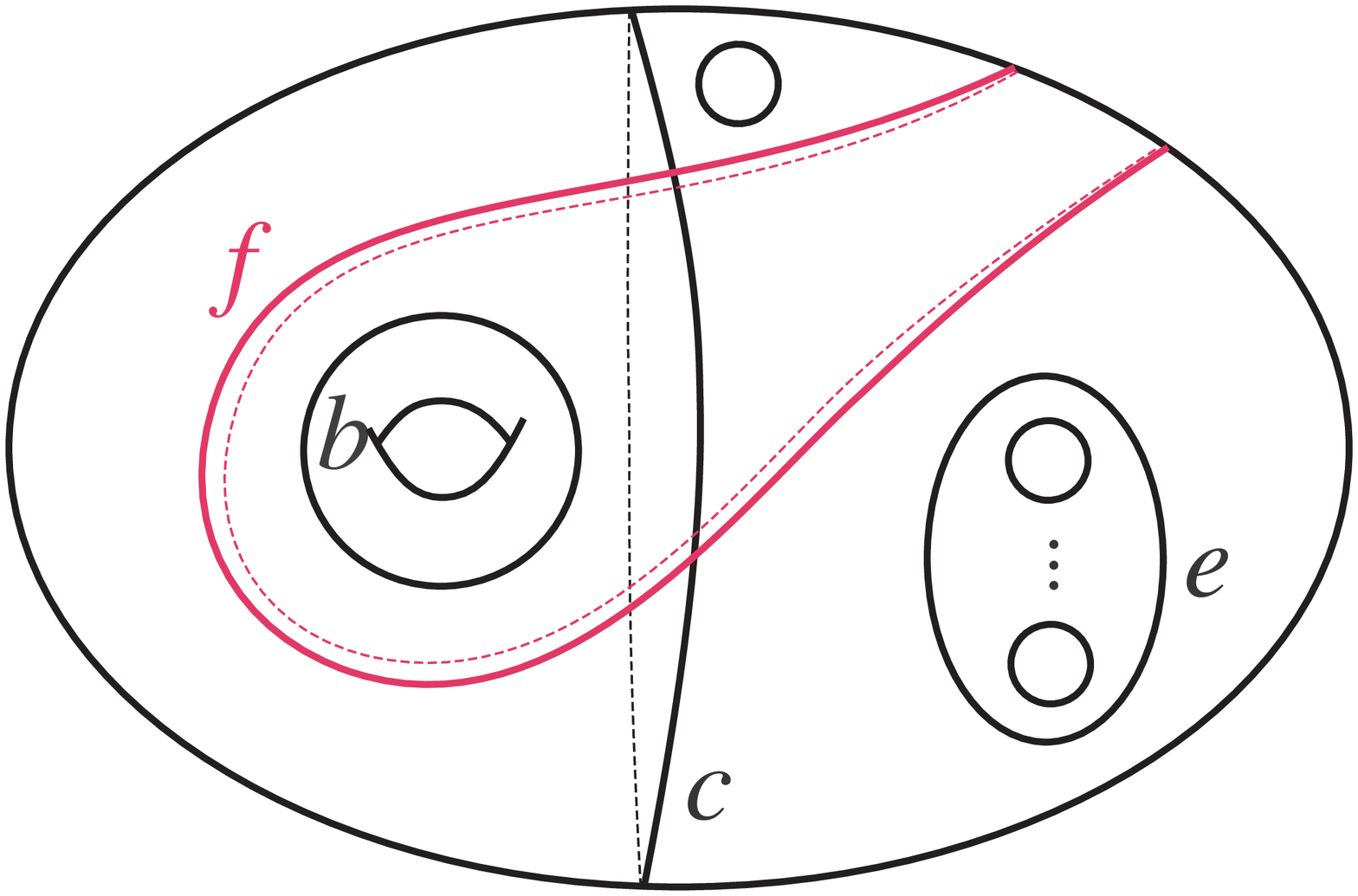} 
		\hspace{0.2cm}  \epsfxsize=2.2in \epsfbox{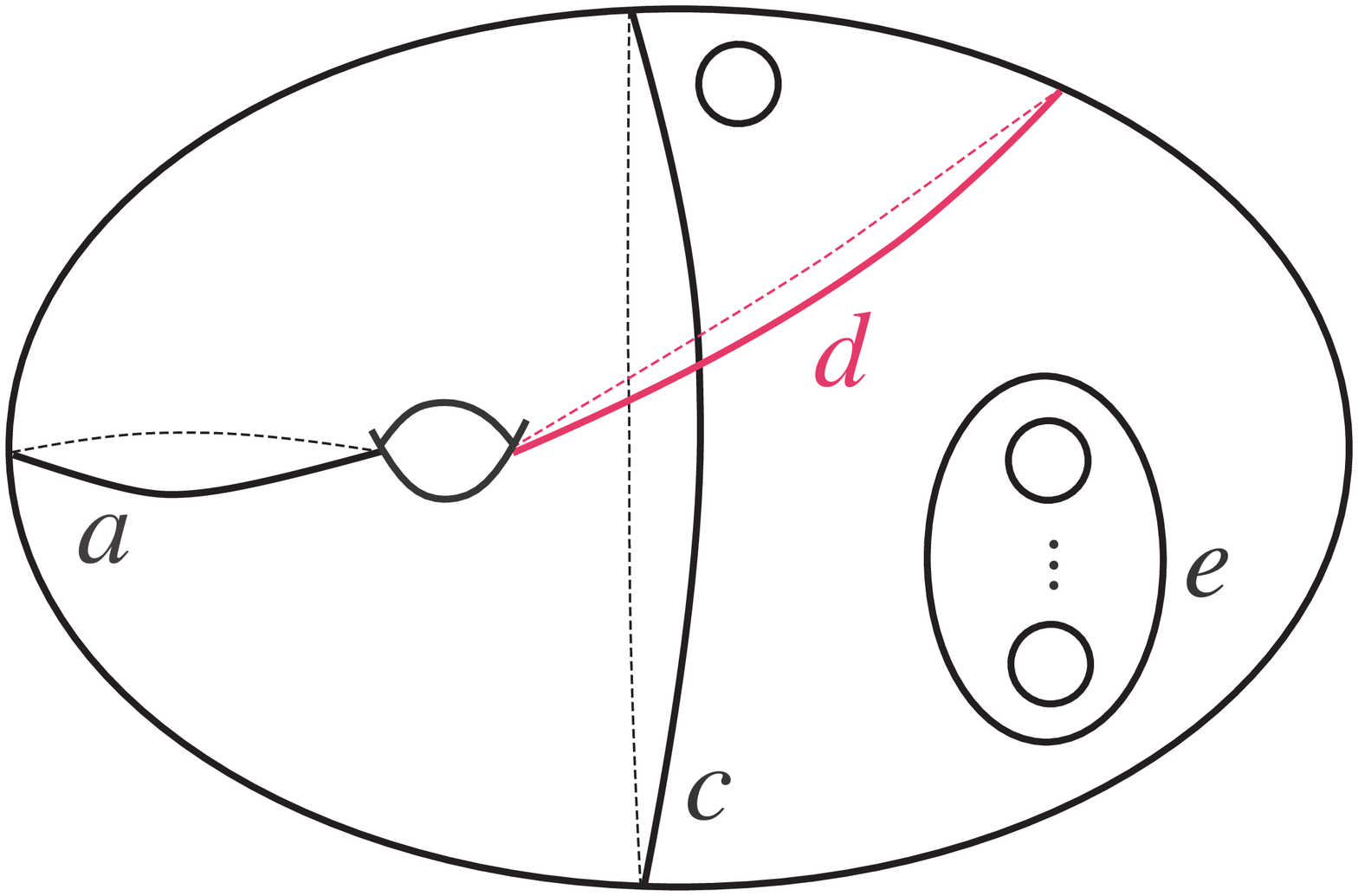}
		
		(iii)   \hspace{5.6cm} (iv)  \vspace{0.3cm}
		
		\caption{Intersection one} \label{fig-7s}
	\end{center} 
\end{figure}

\begin{lemma}
	\label{int-one} If $\alpha_1, \alpha_2$ are two vertices of 
	$\mathcal{C}(R)$ with $i(\alpha_1, \alpha_2) = 1$, then $i(\lambda(\alpha_1), \lambda(\alpha_2)) =1 $.\end{lemma}

\begin{proof} Let $a, b$ be representatives of $\alpha_1, \alpha_2$ respectively. We will complete $a, b$ to a curve configuration $\{a, b, c, d, e, f \}$ as shown in Figure \ref{fig-7s} (i). We can let $c_1=c, c_2=e$ and complete $\{a, c, e\}$ to a pants decomposition $P$ as in Lemma \ref{adj-2}, and using that adjacency and nonadjadjacency are preserved w.r.t. $P'$ by Lemma \ref{adj-2} and 
Lemma \ref{nonadj-2}, we can see that $\lambda([c])$ has a representative $c'$ which is a separating curve that separates the surface into two pieces and one of this is a torus $T$ with one boundary component and $\lambda([a])$ has a nonseparating represantative, say $a'$, in $T$. Let $b', d', e', f'$ be minimally intersecting represantatives of $\lambda([b]), \lambda([d])$ $\lambda([e])$ $\lambda([f])$ respectively such that all the curves 
$a', b', c', d', e', f'$ minimally intersect each other. By Lemma \ref{22} we know that $i([a'], [b']) \neq 0$, $i([b'], [d']) \neq 0$.

\begin{figure} 
	\begin{center}
		\epsfxsize=2.2in \epsfbox{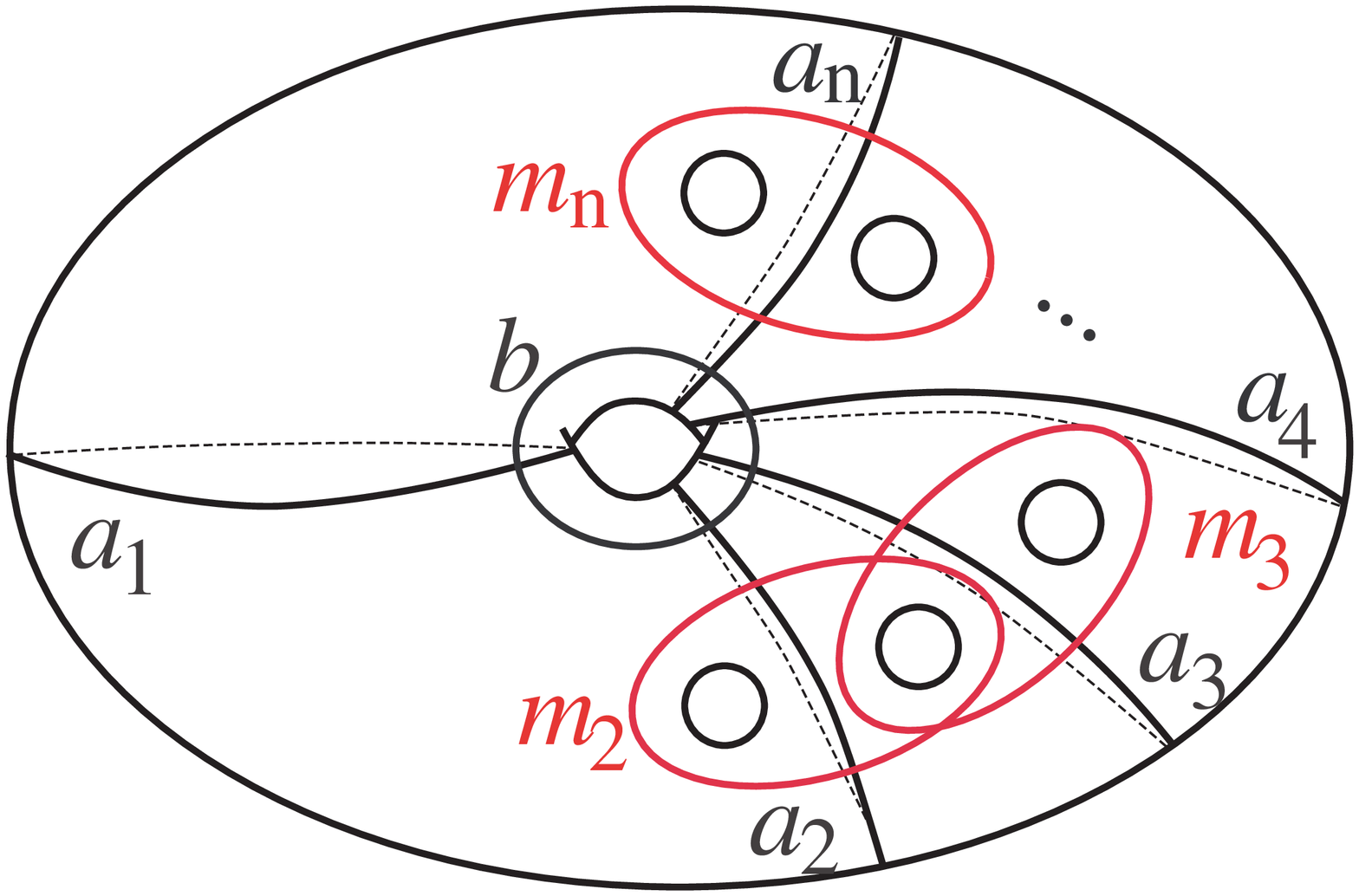} \hspace{0.2cm} 	\epsfxsize=2.2in \epsfbox{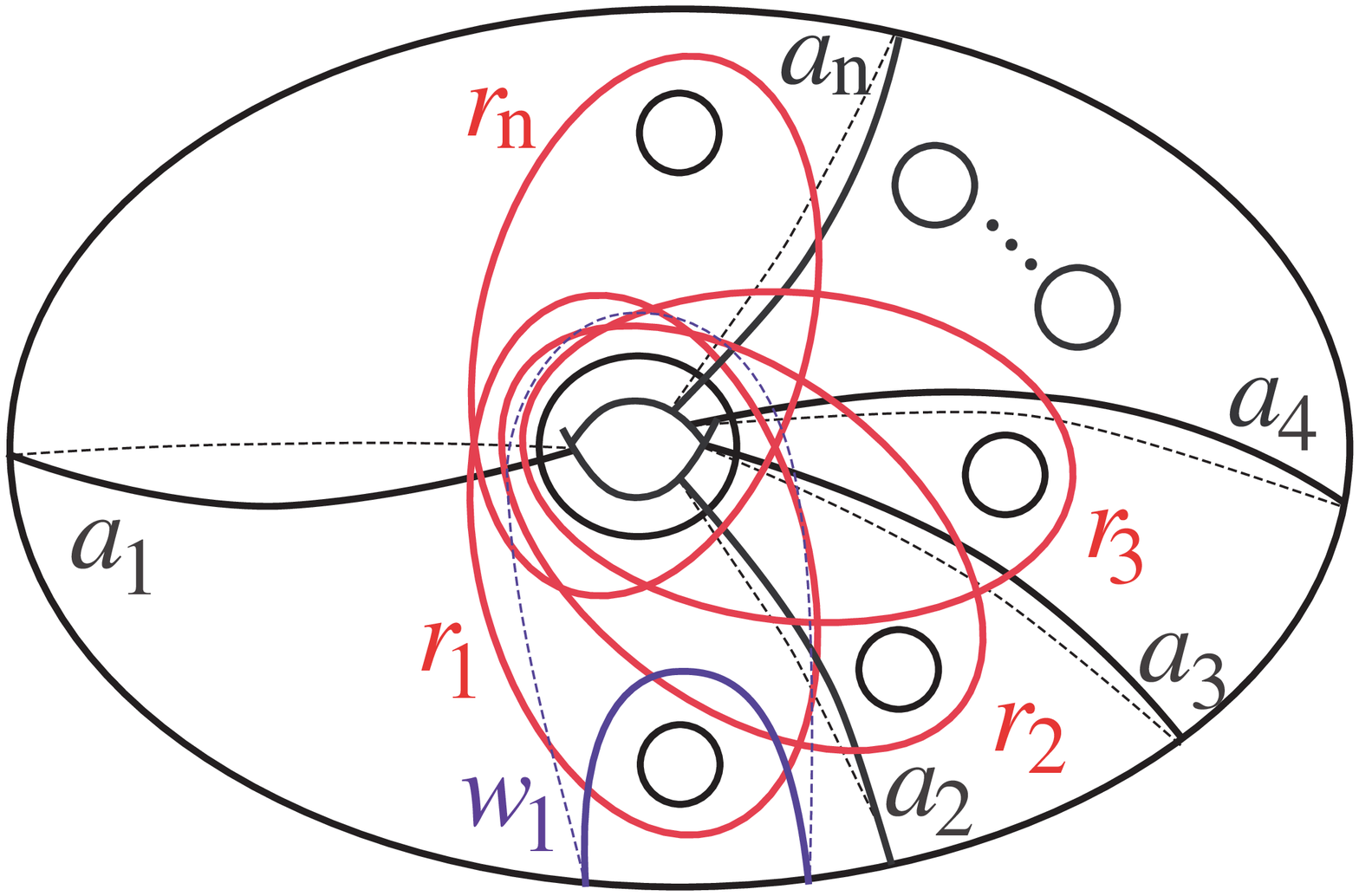}
		
		(i)   \hspace{5.6cm} (ii)  \vspace{0.3cm}
		
		\epsfxsize=2.2in \epsfbox{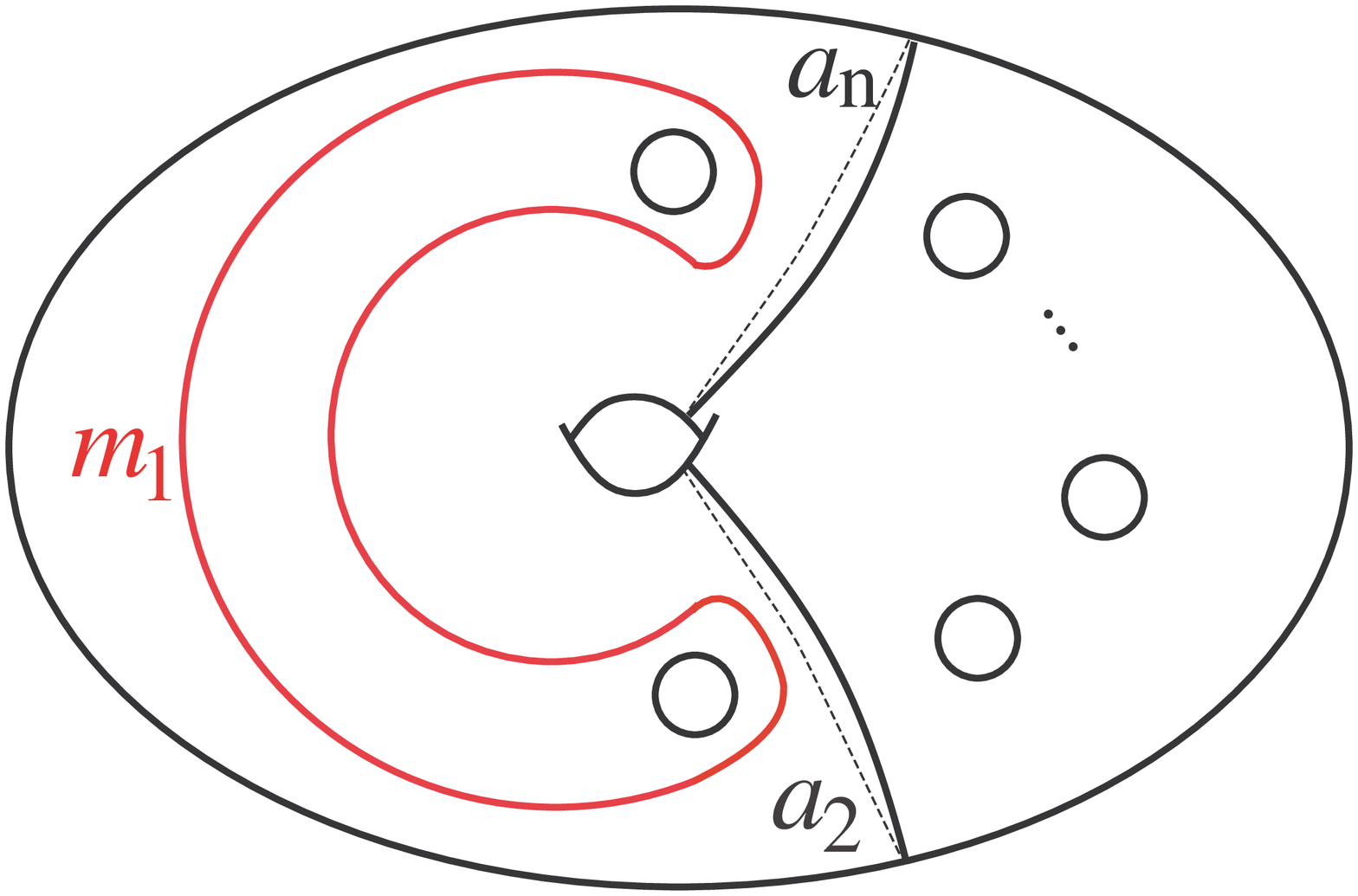} 
		\hspace{0.2cm} 
		\epsfxsize=2.2in \epsfbox{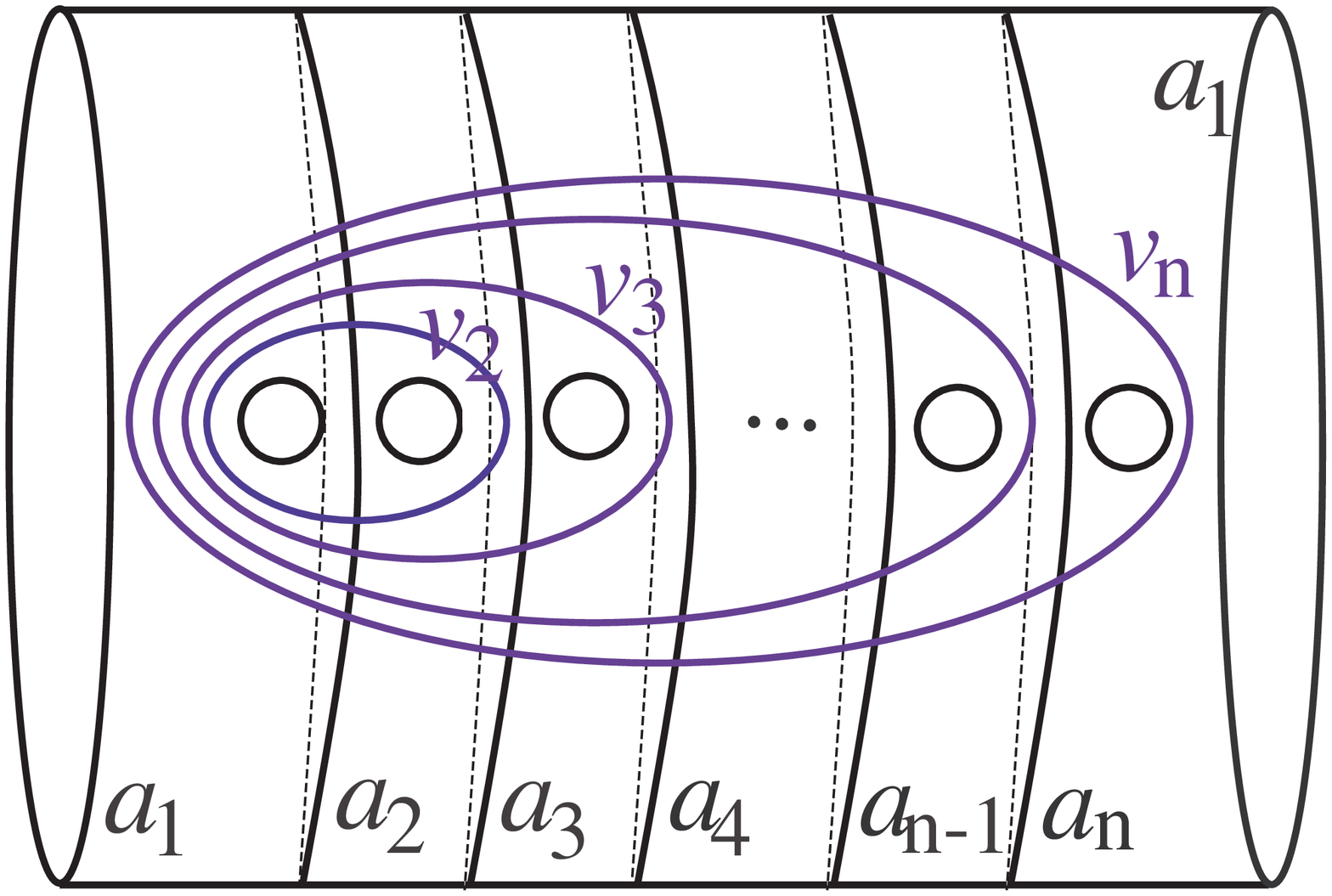}
		
		(iii)   \hspace{5.6cm} (iv) \vspace{0.3cm}
		\caption{Curves in $\mathcal{C}$} \label{fig-8s}
	\end{center} 
\end{figure}	

We will prove that $i([a'], [f']) \neq 0$, $i([c'], [f']) \neq 0$, $i([d'], [c']) \neq 0$. To see $i([a'], [f']) \neq 0$, let $U= (P \setminus \{c_1\}) \cup \{d\}$. Then $U$ is a pants decomposition on $R$ and $i(\lambda([f]), \lambda([x])) = 0$ for every $x \in U \setminus \{a\}$, see Figure \ref{fig-7s} (i). So, either $i(\lambda([f]), \lambda([a])) \neq 0$ or $\lambda([f]) = \lambda([a])$. By Lemma \ref{22} $i(\lambda([a]), \lambda([b])) \neq 0$. Since $i(([f]), [b]) =0$, we have $i(\lambda([f]), \lambda([b])) = 0$. So, $\lambda([f])$ cannot be equal to $\lambda([a])$. Hence, $i(\lambda([f]), \lambda([a])) \neq 0$. 

To see $i([c'], [f']) \neq 0$, let $V=(P \setminus \{a\}) \cup \{b\}$. Then $V$ is a pants decomposition on $R$ and $i(\lambda([f]), \lambda([x])) = 0$ for every $x \in V \setminus \{c\}$, see Figure \ref{fig-7s} (iii). So, either $i(\lambda([f]), \lambda([c])) \neq 0$ or $\lambda([f]) = \lambda([c])$. By
the above paragraph we know that $i(\lambda([a]), \lambda([f])) \neq 0$. Since $i(([a]), [c]) =0$, we have $i(\lambda([a]), \lambda([c])) = 0$. So, $\lambda([f])$ cannot be equal to $\lambda([c])$. Hence, $i(\lambda([f]), \lambda([c])) \neq 0$.

To see $i([c'], [d']) \neq 0$ we observe that $i(\lambda([d]), \lambda([x])) = 0$ for every $x \in P \setminus \{c\}$, see Figure \ref{fig-7s} (iv). So, either $i(\lambda([d]), \lambda([c])) \neq 0$ or $\lambda([d]) = \lambda([c])$. By
Lemma \ref{22} we know that $i(\lambda([b]), \lambda([d])) \neq 0$. Since $i(([b]), [c]) =0$, we have $i(\lambda([b]), \lambda([c])) = 0$. So, $\lambda([d])$ cannot be equal to $\lambda([c])$. Hence, $i(\lambda([d]), \lambda([c])) \neq 0$.

The above intersection information imply that there is an arc of $d'$, say $\gamma_1$, in $T$ that starts and ends at $c'$ (the boundary of $T$) such that $\gamma_1$ is disjoint from $a'$. Also there is an arc of $f'$, say $\gamma_2$, in $T$ that is disjoint from $\gamma_1$ and 
starts and ends at $c'$. Then, since $b'$ is disjoint from $\gamma_2 \cup c'$, and $b'$ intersects $a'$ by Lemma \ref{22}, we see that $i(a', b')=1$.\end{proof}\medskip

If $f: R \rightarrow R$ is a homeomorphism, then we will use the same notation for $f$ and $[f]$.
Let $\mathcal{C} = \{a_1, a_2, \cdots, a_n, b, m_1, m_2, \cdots, m_n, r_1, r_2, \cdots, r_{n}, v_2, v_3, \cdots, v_n\}$ where the curves are as shown in Figure \ref{fig-8s}.  

\begin{lemma} \label{curves} There exists a 
	homeomorphism $h: R \rightarrow R$ such that $h([x]) = \lambda([x])$ $\forall \ x \in \mathcal{C}$.\end{lemma}

\begin{proof} We will consider all the curves in  $\mathcal{C}$ 
	as shown in Figure \ref{fig-8s}. Let $a'_i  \in \lambda([a_i]), b' \in \lambda([b]), m'_i  \in \lambda([m_i]), r_i' \in \lambda([r_i]), v'_j  \in \lambda([v_j])$ where 
	$i =1, 2, \cdots, n$, $j = 2, 3, \cdots, n$ be minimally intersecting representatives.
	
By using Lemma \ref{int-one} and that $\lambda$ is edge preserving we see that a regular neighborhood of $a'_1 \cup a'_2 \cup \cdots \cup a'_n \cup b$ is a torus with $n$ boundary components as shown in Figure \ref{fig-9s}. So, there exists a homeomorphism $h$ such that $h([x]) = \lambda([x])$ for all $x \in \{a_1, a_2, \cdots, a_n, b\}$. This implies that if two nonseparating curves $x, y$ and a boundary component of $R$ bound a pair of pants then $\lambda([x]), \lambda([y])$ have representatives $x', y'$ such that $x', y'$ and a boundary component of $R$ bound a pair of pants on $R$. 

	\begin{figure} 
	\begin{center}
		\epsfxsize=2.2in \epsfbox{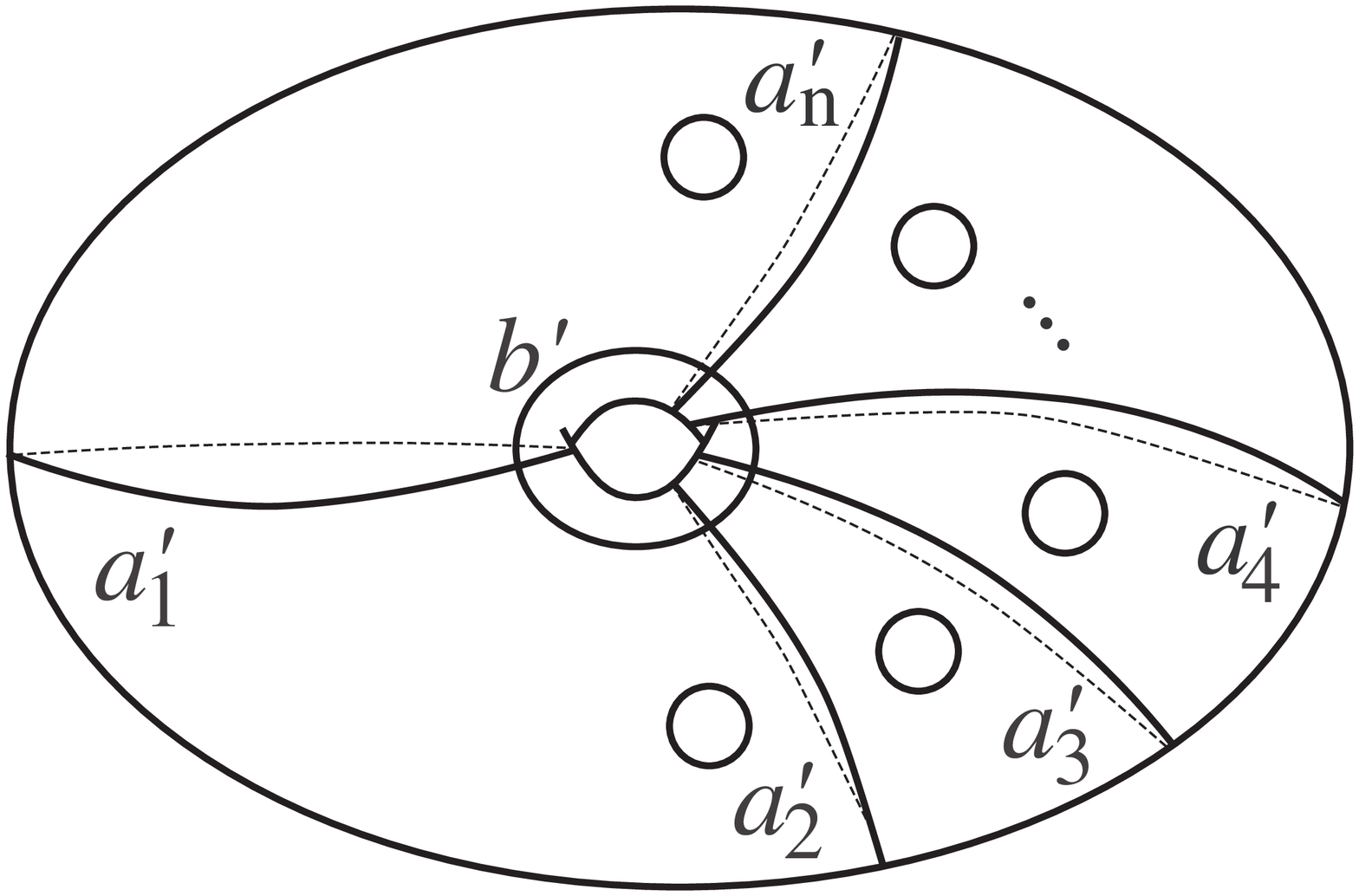}  
		\caption{Curves} \label{fig-9s}
	\end{center} 
\end{figure}

We will now show that $h([m_i]) = \lambda([m_i])$ for all $i= 1, 2, \cdots, n$. The curve $m_1$ is the unique nontrivial curve up to isotopy that is disjoint from all the curves in
$\{a_2, a_3, \cdots, a_n, b\}$, since we know that $h([x]) = \lambda([x])$ for all these curves and $\lambda$ is edge preserving, we have $h([m_1]) = \lambda([m_1])$. The curve $m_2$ is the unique nontrivial curve up to isotopy that is disjoint from all the curves in $\{a_3, a_4, \cdots, a_n, a_1, b\}$, since we know that $h([x]) = \lambda([x])$ for all these curves and $\lambda$ is edge preserving, we have $h([m_2]) = \lambda([m_2])$. Similarly, we have $h([m_i]) = \lambda([m_i])$ for all $i=3, 4, \cdots, n$.

The curve $v_2=m_2$, so $h([v_2]) = \lambda([v_2])$. The curve $v_3$ is the unique nontrivial curve up to isotopy that is disjoint from all the curves in
$\{a_4, a_5, \cdots, a_n, a_1, b, m_2, m_3\}$. Since we know that $h([x]) = \lambda([x])$ for all these curves and $\lambda$ is edge preserving, we have $h([v_3]) = \lambda([v_3])$. The curve $v_4$ is the unique nontrivial curve up to isotopy that is disjoint from all the curves in $\{a_5, a_6, \cdots, a_n, a_1, b, m_2, m_3, m_4\}$. Since we know that $h([x]) = \lambda([x])$ for all these curves and $\lambda$ is edge preserving, we have $h([v_4]) = \lambda([v_4])$. Similarly, we have $h([v_i]) = \lambda([v_i])$ for all $i=5, 6, \cdots, n$.

Consider the curve $w_1$ as shown in the Figure \ref{fig-8s} (ii). There exists a
homeomorphism $\phi : R \rightarrow R$ of order two such that the map $\phi_{*}$ induced by $\phi$ on $\mathcal{C}(R)$ sends the isotopy class of each curve in $\{a_1, a_2, \cdots, a_n, m_1, m_2, \cdots, m_n\}$ to itself and switches $[r_1]$ and $[w_1]$. We can see that $\lambda([r_1]) \neq \lambda([w_1])$ as follows: Consider the curve $y$ we had in Lemma \ref{intersection}. We will first prove that $i(\lambda([y]), \lambda([w_1])) \neq 0$.  
We complete $y$ to a pants decomposition $P$ on $R$ such that $i([w_1], [x]) = 0$ for every $x \in P \setminus \{y\}$, see Figure \ref{fig-10s} (i). Then we will have $i(\lambda([w_1]), \lambda([x])) = 0$ for every $x \in P \setminus \{y\}$. So, either $i(\lambda([w_1]), \lambda([y])) \neq 0$ or $\lambda([w_1]) = \lambda([y])$. By Lemma \ref{intersection} we know that $i(\lambda([y]), \lambda([c_{n-1}])) \neq 0$, see Figure \ref{fig-10s} (ii). We also see that $i(\lambda([w_1]), \lambda([c_{n-1}])) = 0$. So, 
$\lambda([w_1]) \neq \lambda([y])$. Hence, 
$i(\lambda([y]), \lambda([w_1])) \neq 0$. 
Since $i(\lambda([y]), \lambda([r_1])) = 0$ and $i(\lambda([y]), \lambda([w_1])) \neq 0$, we see that $\lambda([r_1]) \neq \lambda([w_1])$.
 
 \begin{figure} 
 	\begin{center}
 		\epsfxsize=2.2in \epsfbox{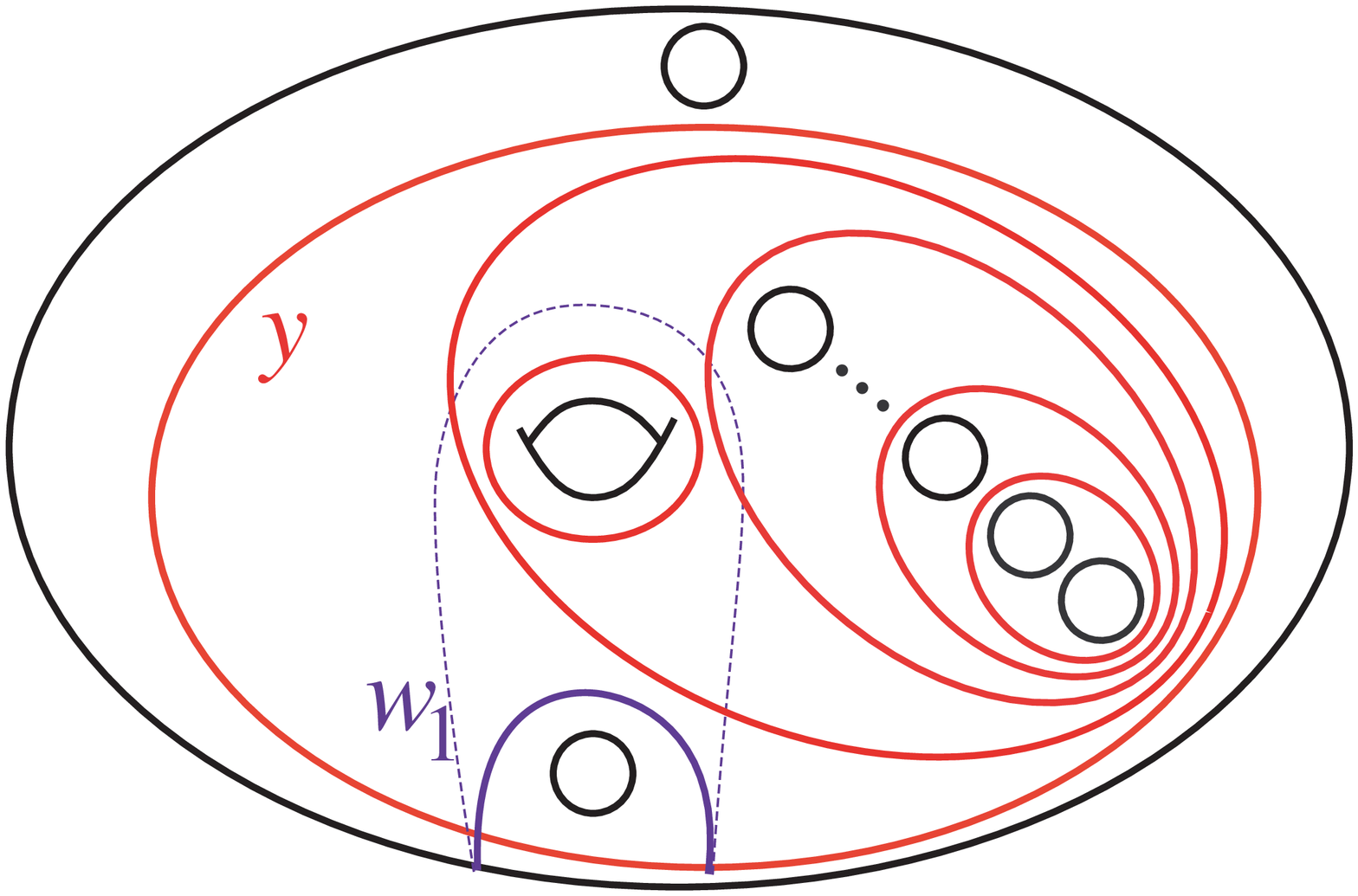} \hspace{0.2cm} 	\epsfxsize=2.2in \epsfbox{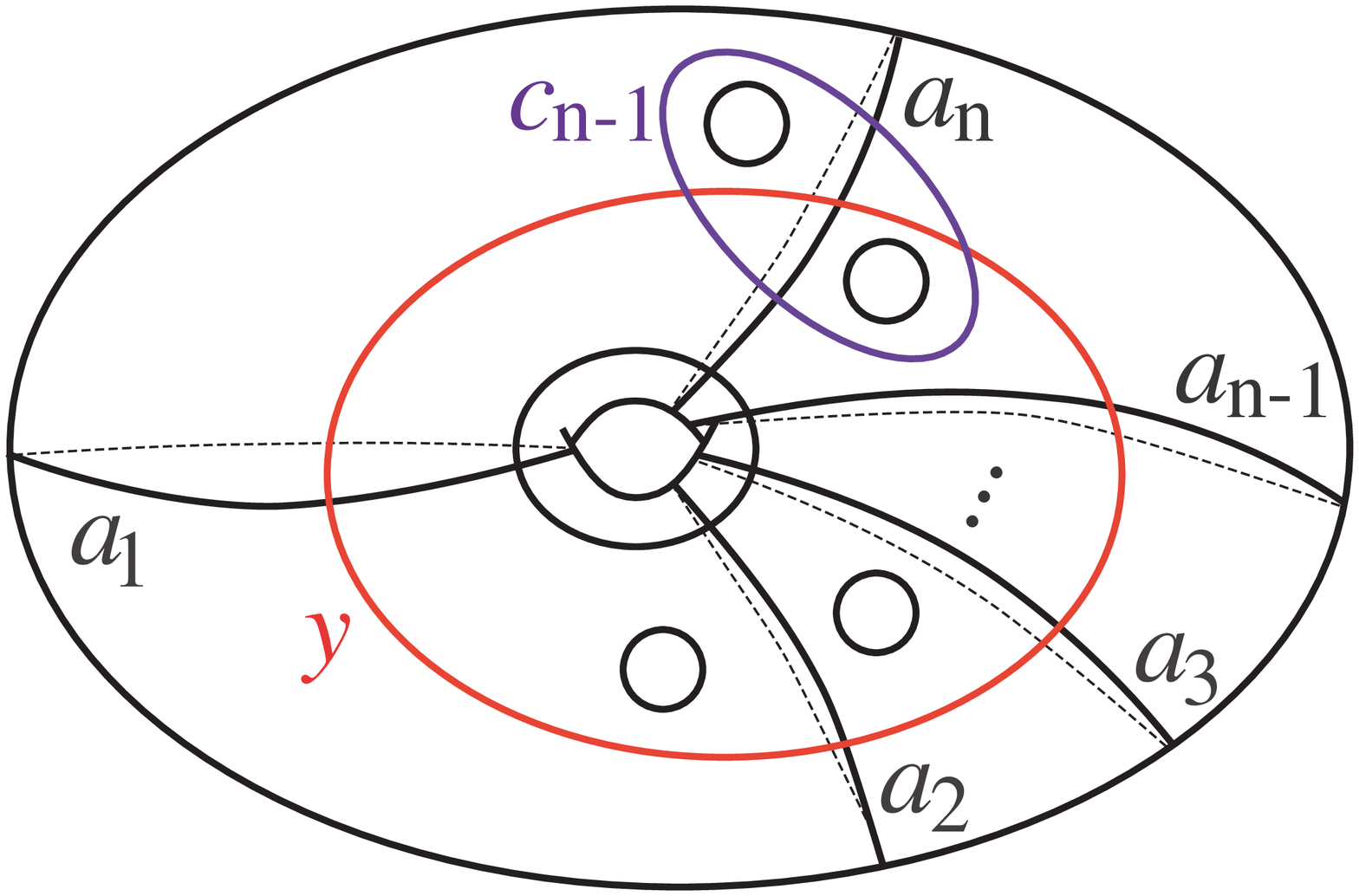}
 		
 		(i)   \hspace{5.6cm} (ii)  \vspace{0.3cm}
 		\caption{Curves } \label{fig-10s}
 	\end{center} 
 \end{figure}
 
There are only two nontrivial curves, namely $r_1$ and $w_1$, up to isotopy that are disjoint from each of $m_3, m_4, \cdots, m_n$, bounds a pair of pants with $b$ and a boundary component of $R$, 
and intersects each of $a_1, a_2, \cdots, a_{n}$ once. Since we know that $h([x]) = \lambda([x])$ for all these curves, $\lambda$ preserves these properties by Lemma \ref{int-one}, and $\lambda([r_1]) \neq \lambda([w_1])$, by replacing $\lambda$ with $\lambda \circ \phi_{*}$ if necessary, we can assume that we have $h([r_1]) = \lambda([r_1])$ and 
$h([w_1]) = \lambda([w_1])$. To get the proof of the lemma, it is enough to prove the result for
this $\lambda$. The curve $r_2$ is the unique nontrivial curve up to isotopy that is disjoint from each of $m_4, m_5, \cdots, m_n, m_1, w_1$, bounds a pair of pants with $b$ and a boundary component of $R$ 
and intersects each of $a_1, a_2, \cdots, a_{n}$ once.
Since we know that $h([x]) = \lambda([x])$ for all these curves and $\lambda$ preserves these properties, we see that $h([r_2]) = \lambda([r_2])$. Similarly, we get  $h([r_i]) = \lambda([r_i])$ $ \forall i= 3, 4, \cdots, n$.
Hence, $h([x]) = \lambda([x])$ $\forall \ x \in \mathcal{C}$.\end{proof}\\

\begin{figure} 
	\begin{center}	
		\epsfxsize=2.2in \epsfbox{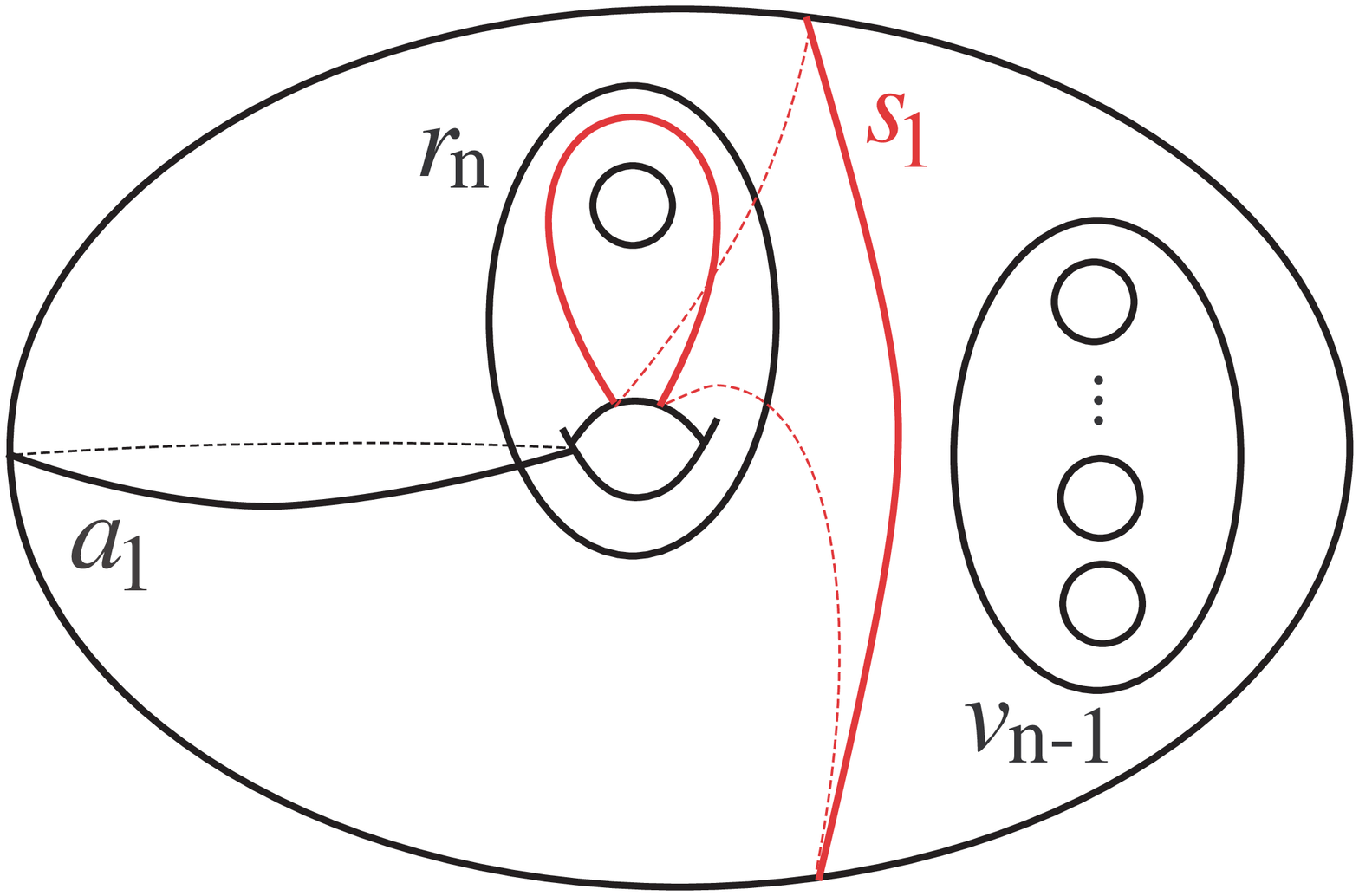} \hspace{0.2cm} 
		\epsfxsize=2.2in \epsfbox{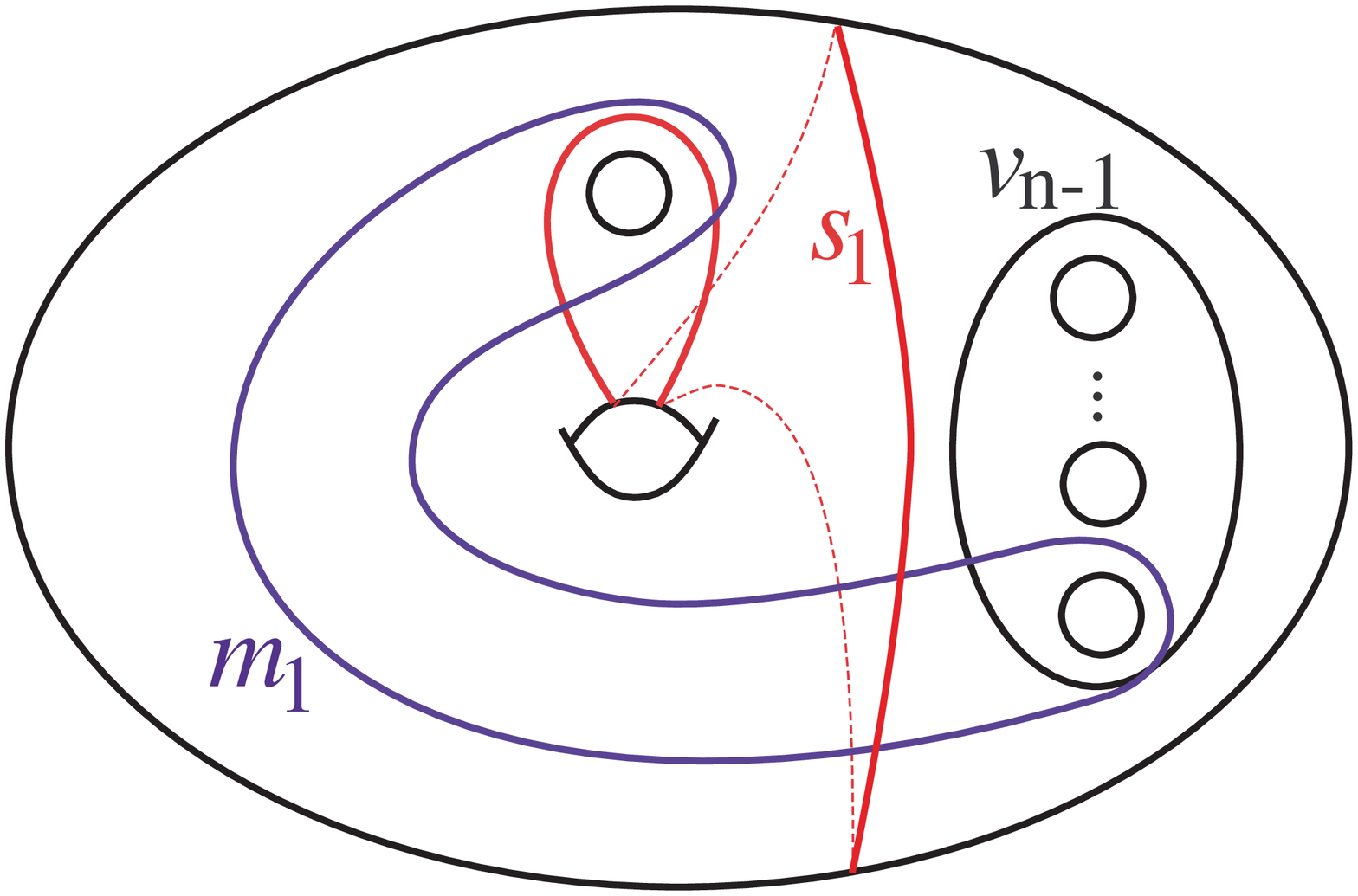}
		
		\vspace{0.3cm}
		
		(i)   \hspace{5.6cm} (ii)  \vspace{0.3cm}
		
		\epsfxsize=2.2in \epsfbox{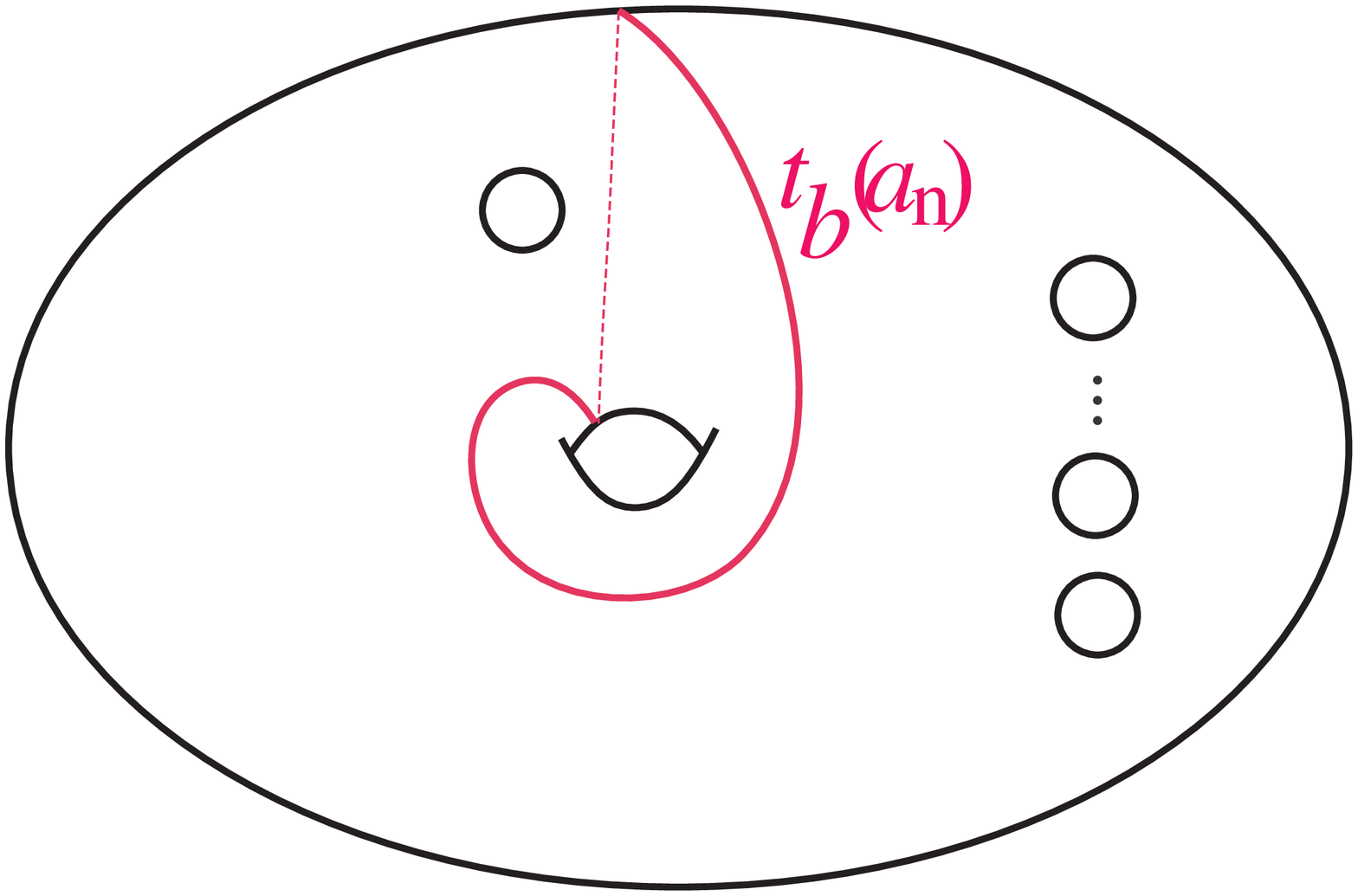} \hspace{0.2cm} 
		\epsfxsize=2.2in \epsfbox{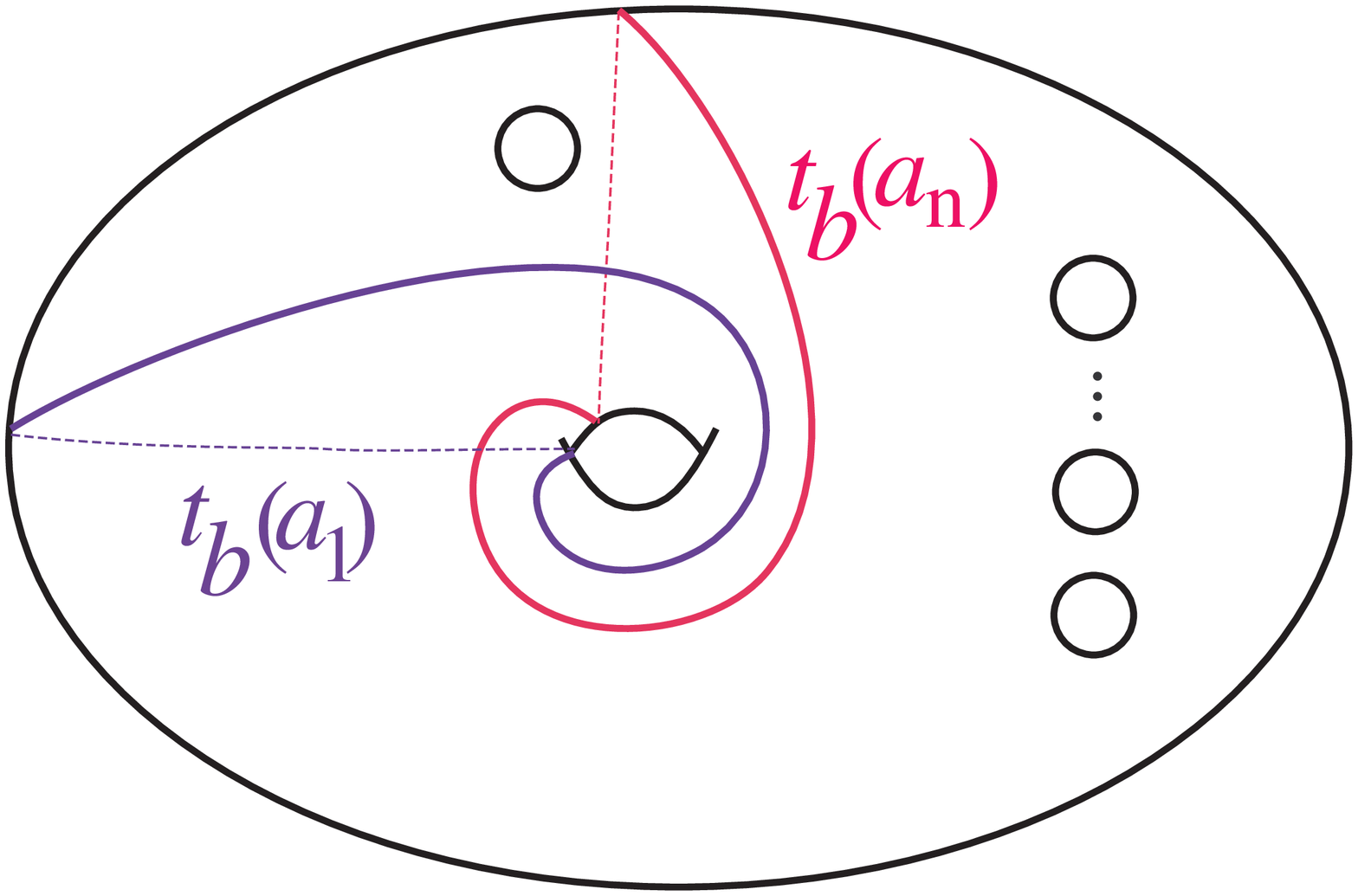}
		
		\vspace{0.3cm}
		
		(iii)   \hspace{5.6cm} (iv) \vspace{0.3cm}   
		
		\epsfxsize=2.2in \epsfbox{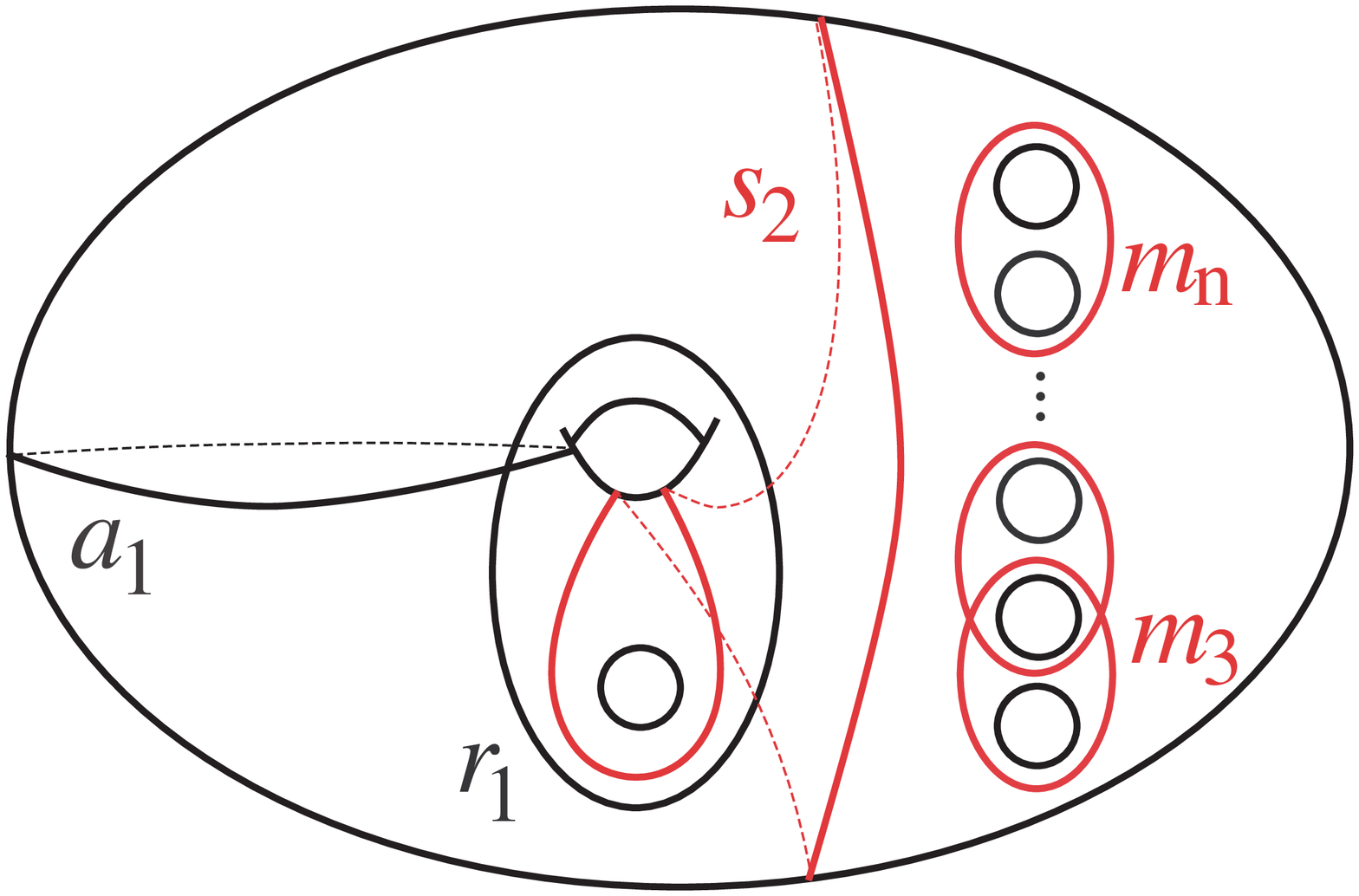} \hspace{0.2cm}  \epsfxsize=2.2in \epsfbox{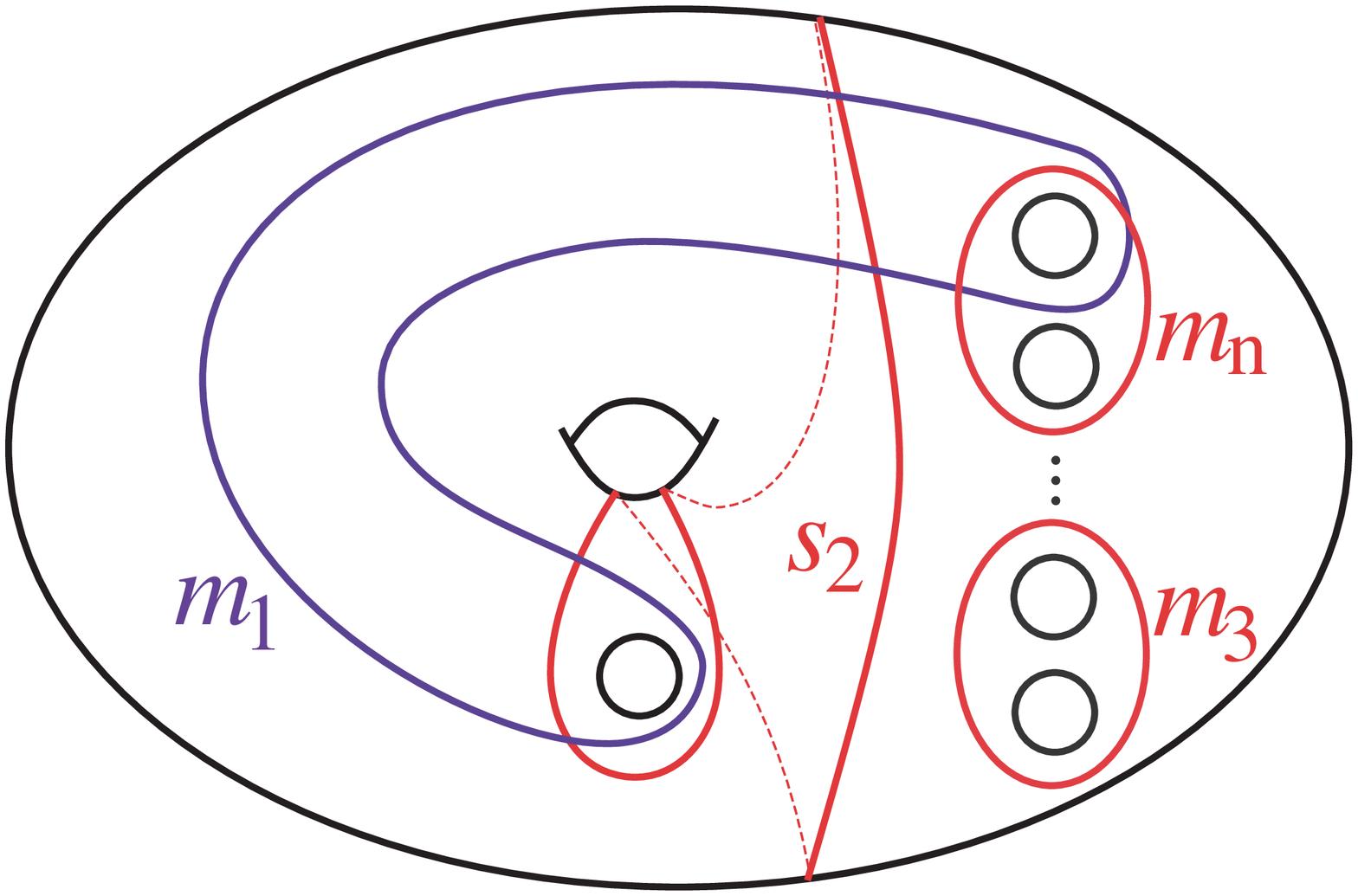}
		
		\vspace{0.3cm}
		
		(v)   \hspace{5.6cm} (vi)  \vspace{0.3cm}
		
		\epsfxsize=2.2in \epsfbox{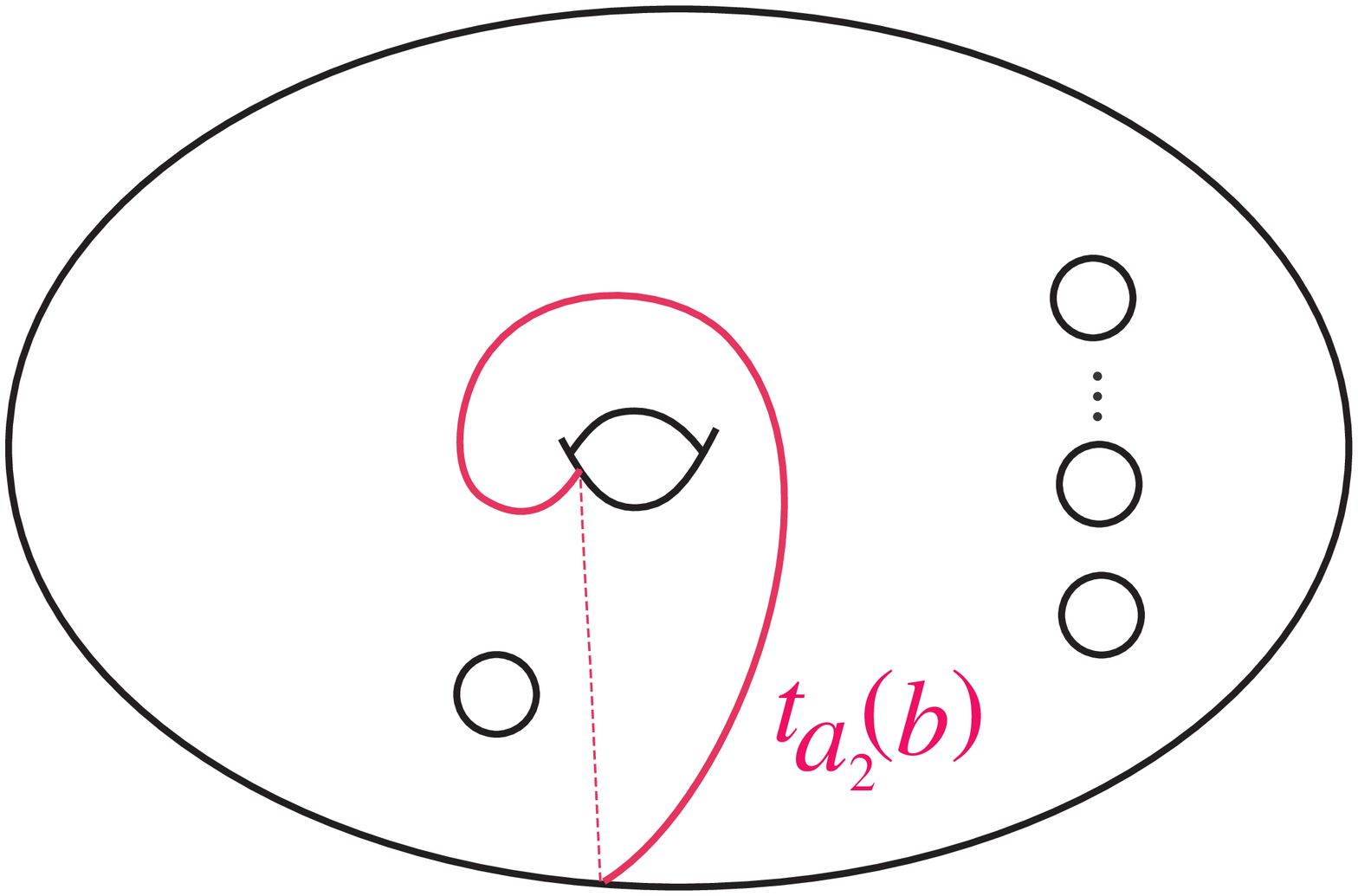} \hspace{0.2cm} 
		\epsfxsize=2.2in \epsfbox{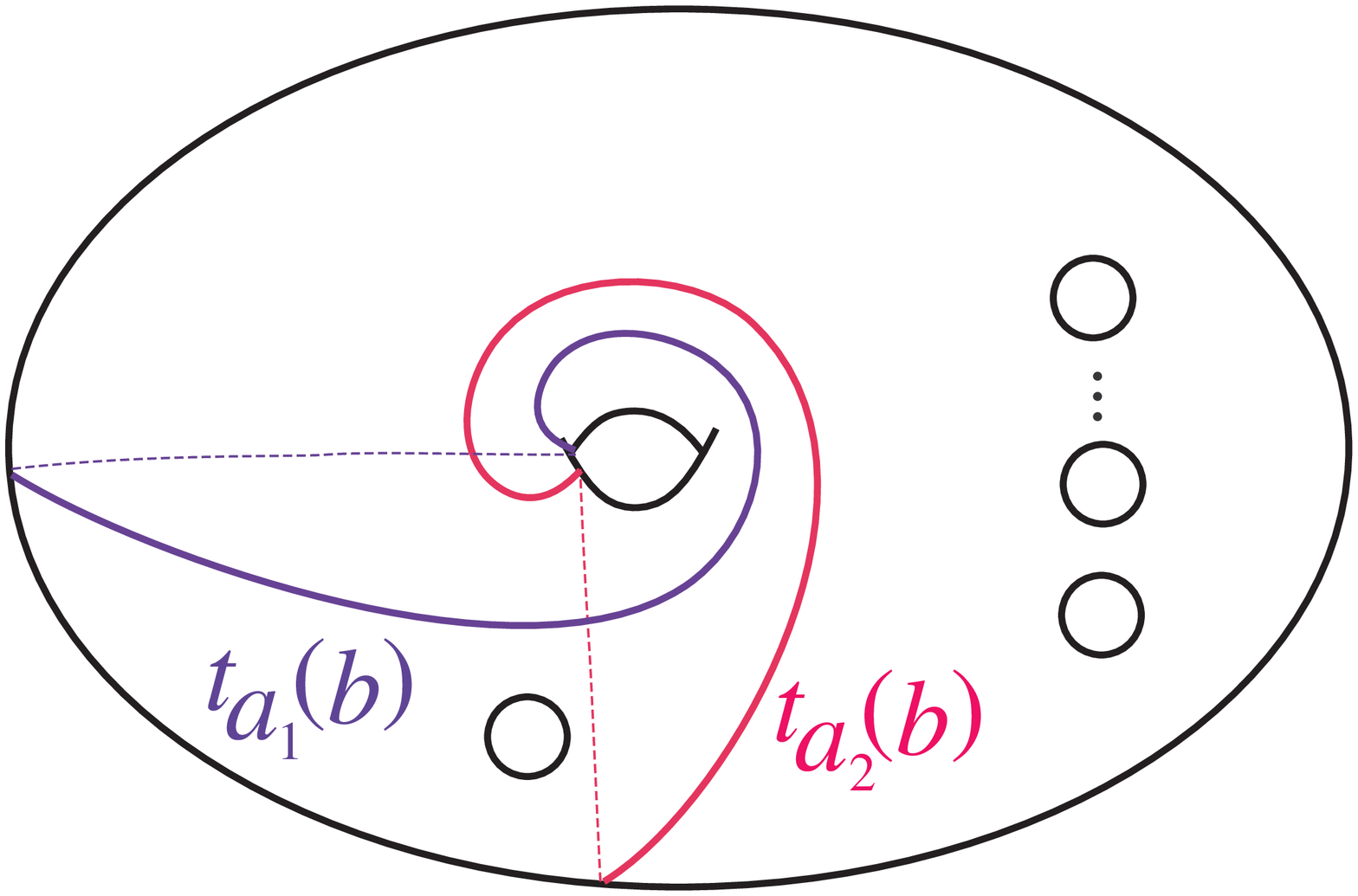}
		
		\vspace{0.3cm}	
		
		(vii)   \hspace{5.6cm} (viii)  
		
		\caption {Twists}
		\label{fig-11s}
	\end{center}
\end{figure}

Consider the curves given in Figure \ref{fig-8s} (i). Let $t_x$ be the Dehn twist about $x$. Let $\sigma_i$ be the half twist along 
$m_i$. The mapping class group $Mod_R$ can be generated by $\{t_x: x \in \{a_1, a_2, \cdots, a_{n}, b\} \} \cup \{\sigma_2,$ $\sigma_3, \cdots, \sigma_{n}\}$, see Corollary 4.15 in \cite{FM}. 
Let $G = \{t_x: x \in \{a_1, a_2, \cdots, a_{n}, b\} \} \cup \{\sigma_2, \sigma_3, \cdots, \sigma_{n}\}$. 
Let $h: R \rightarrow R$ be a homeomorphism which satisfies the statement of Lemma \ref{curves}. We know $h([x]) = \lambda([x])$ $\forall \ x \in \mathcal{C}$. 
We will follow the techniques given by Irmak-Paris \cite{IrP2} to obtain the homeomorphism we want. We will say that a subset $\mathcal{A} \subset \mathcal{C}(R)$ has trivial stabilizer if we have the following: $h \in Mod^*_R$, $h([x]) = [x]$ for every vertex $x \in A$ implies $h$ is identity.

\begin{lemma} \label{prop-imp}   $\forall \ f \in G$, 
	$\exists$ a set $L_f \subset \mathcal{C}(R)$
	such that $\lambda([x])= h([x])$ $\ \forall \ x \in L_f \cup f(L_f)$. $L_f$ can be chosen to have trivial stabilizer. \end{lemma}

\begin{proof} We have $h([x]) = \lambda([x])$ $\forall \ x \in  \mathcal{C}$ by Lemma \ref{curves}. Let $f \in G$. For $f=t_b$, let $L_f = \{a_1, a_2, \cdots, a_{n}, b, r_1\}$. The set $L_f$ has trivial stabilizer. We know $\lambda([x])= h([x])$
$\forall \ x \in L_f$. We need to check the equation for $t_{b}(a_i)$, the other curves in $L_f$ are fixed by $t_{b}$. We will first check the equation for $t_{b}(a_n)$.
Consider the curves given in Figure \ref{fig-11s}. 
The curve $s_1$ is the unique nontrivial curve up to isotopy that is disjoint from all the curves in
$\{a_1, r_n, v_{n-1}\}$. Since we know that $h([x]) = \lambda([x])$ for all
these curves and $\lambda$ is edge preserving, we have $h([s_1]) = \lambda([s_1])$.  
The curve $t_b(a_n)$ is the unique nontrivial curve up to isotopy that is disjoint from all the curves in
$\{m_1, s_1, v_{n-1}\}$. Since we know that $h([x]) = \lambda([x])$ for all
these curves and $\lambda$ is edge preserving, we have $h([t_b(a_n)]) = \lambda([t_b(a_n)])$. The curve $t_b(a_1)$ is the unique nontrivial curve up to isotopy that is disjoint from $t_b(a_n)$ and $v_{n}$, and also that intersects each of $a_1, b$ nontrivially once. Since we know that $h([x]) = \lambda([x])$ for all
these curves, $\lambda$ is edge preserving and it preserves intersection one, we have $h([t_b(a_1)]) = \lambda([t_b(a_1)])$. The curve $t_b(a_2)$ is the unique nontrivial curve up to isotopy that is disjoint from all the curves in $\{t_b(a_n), m_1, m_3, m_4, \cdots, m_n\}$ and also that intersects each of $a_2, b$ nontrivially once. Since we know that $h([x]) = \lambda([x])$ for all
these curves, $\lambda$ is edge preserving and it preserves intersection one, we have $h([t_b(a_2)]) = \lambda([t_b(a_2)])$. Similarly, we get $h([t_b(a_i)]) = \lambda([t_b(a_i)])$ for all $i = 3, 4, \cdots, n-1$. This proves the statement of the lemma for $f=t_b$. 
 
\begin{figure}
	\begin{center}
		\epsfxsize=2.2in \epsfbox{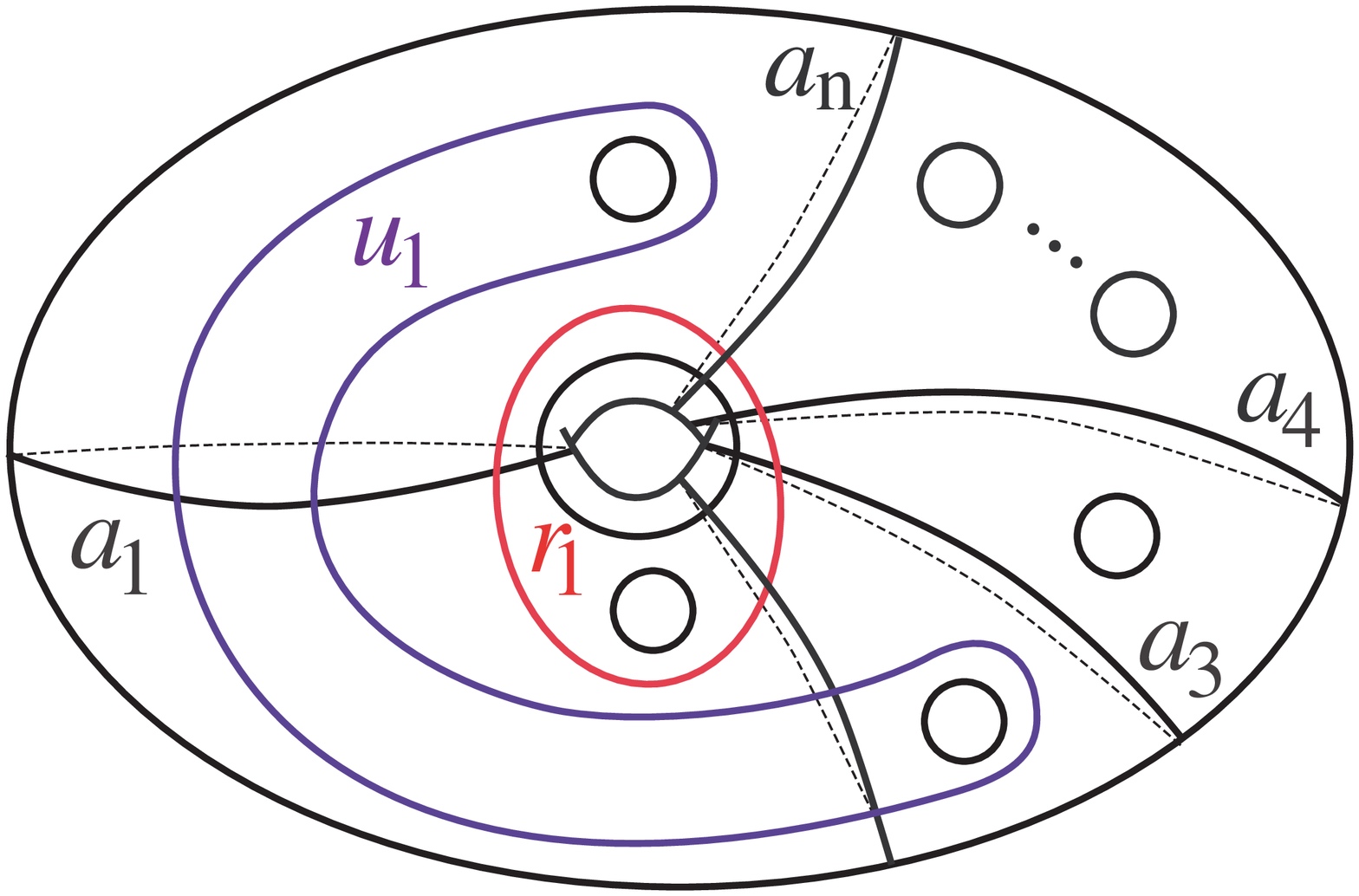} \hspace{0.2cm} 
		\epsfxsize=2.2in \epsfbox{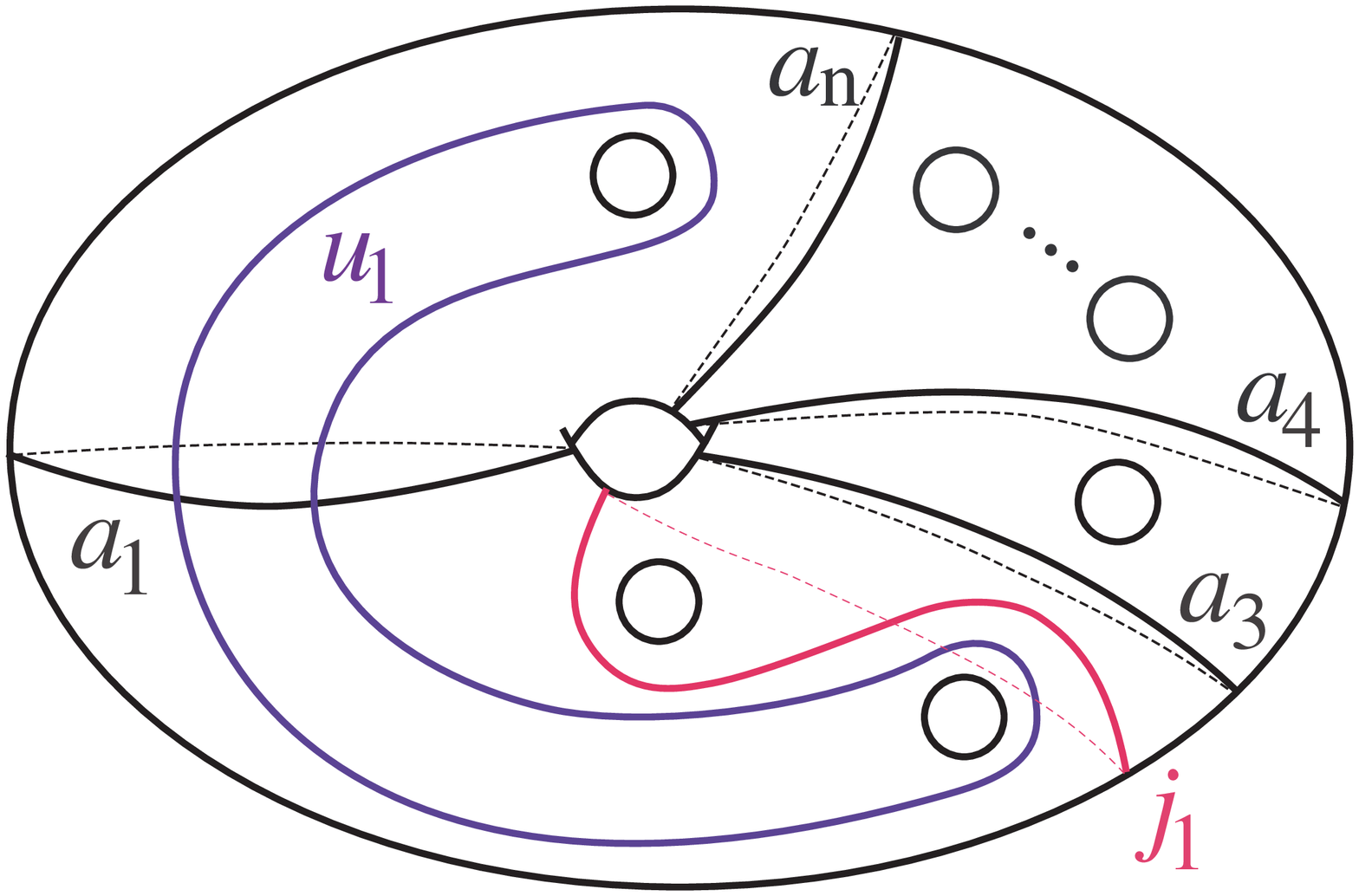}   
		
		(i)   \hspace{5.6cm} (ii) 
		
		\vspace{0.3cm}
		
		\epsfxsize=2.2in \epsfbox{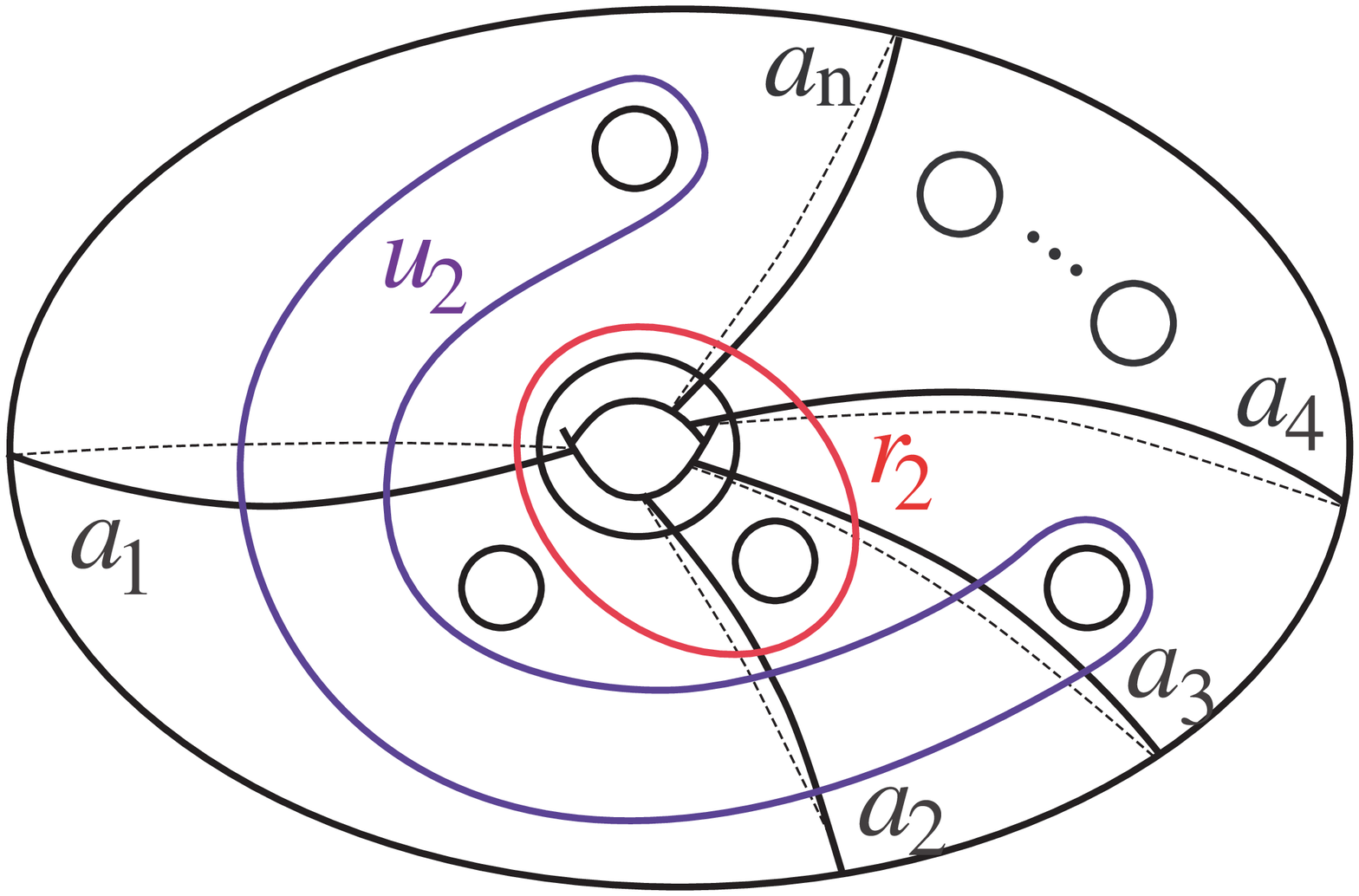} \hspace{0.2cm} 
		\epsfxsize=2.2in \epsfbox{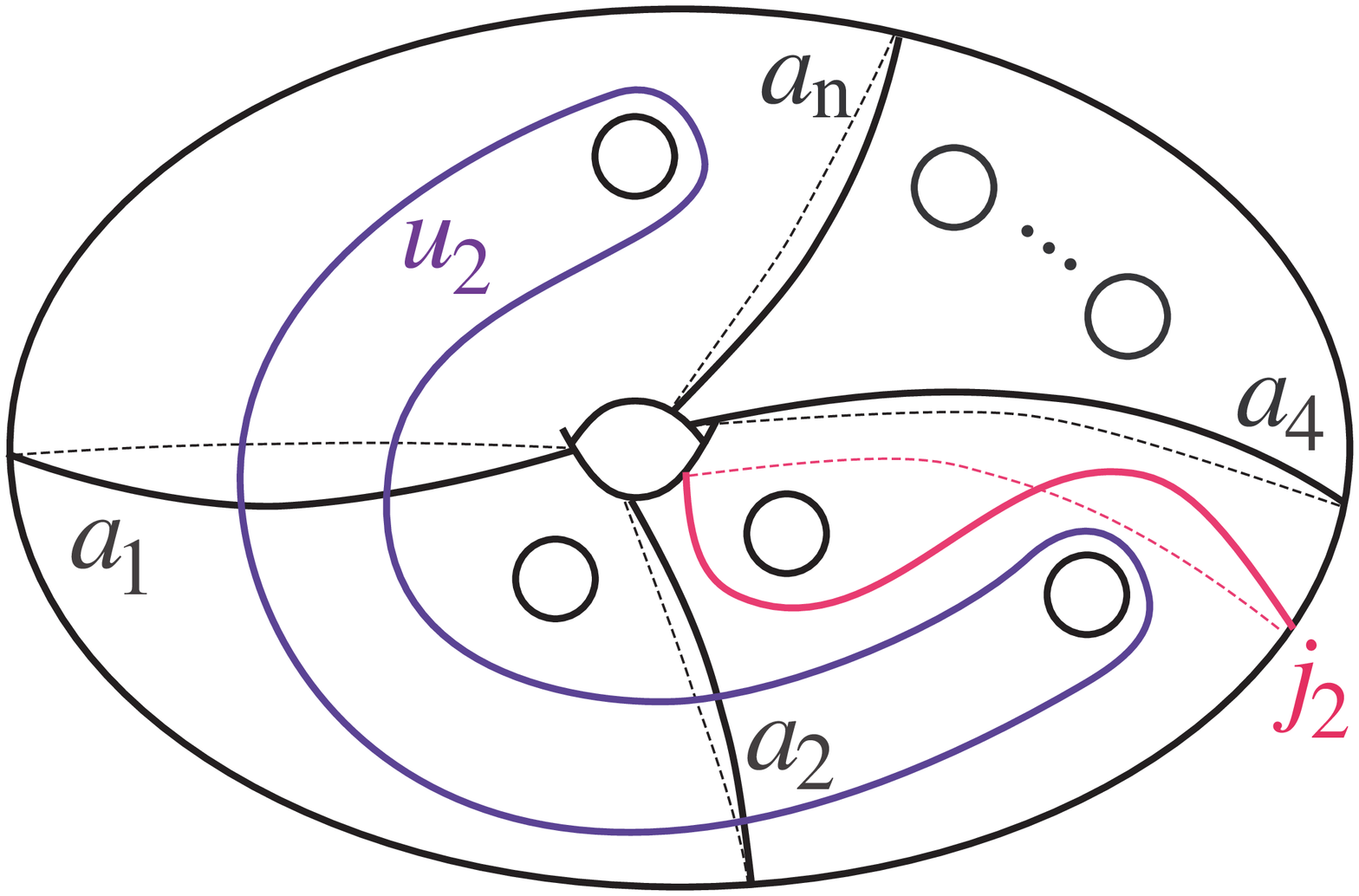}
		
		(iii)   \hspace{5.6cm} (iv) 	
		\caption {Twists, half-twists}
		\label{fig-12s}
	\end{center}
\end{figure} 
 
For $f=t_{a_2}$, let $L_f = \{a_1, a_2, \cdots, a_{n}, b, r_2\}$. The set $L_f$ has trivial stabilizer. We know $\lambda([x])= h([x])$
$\forall \ x \in L_f$. We just need to check the equation for $t_{a_2}(b)$ and 
$t_{a_2}(r_2)$ since the other curves in $L_f$ are fixed by $t_{a_2}$. 
Consider the curves given in Figure \ref{fig-11s} (v). 
The curve $s_2$ is the unique nontrivial curve up to isotopy that is disjoint from all the curves in
$\{a_1, r_1, m_3, m_4, \cdots, m_{n}\}$. Since we know that $h([x]) = \lambda([x])$ for all
these curves and $\lambda$ is edge preserving, we have $h([s_2]) = \lambda([s_2])$.  
The curve $t_{a_2}(b)$ is the unique nontrivial curve up to isotopy that is disjoint from all the curves in
$\{m_1, s_2, m_3, m_4, \cdots, m_{n}\}$. Since we know that $h([x]) = \lambda([x])$ for all
these curves and $\lambda$ is edge preserving, we have $h([t_{a_2}(b)]) = \lambda([t_{a_2}(b)])$. The curve $t_{a_1}(b)$ is the unique nontrivial curve up to isotopy that is disjoint from $t_{a_2}(b)$ and $v_{n}$. Since we know that $h([x]) = \lambda([x])$ for all
these curves and $\lambda$ is edge preserving we have $h([t_{a_1}(b)]) = \lambda([t_{a_1}(b)])$. Similarly, we get $h([t_{a_i}(b)]) = \lambda([t_{a_i}(b)])$ for all $i =3, 4, \cdots, n$. 
The curve $t_{a_2}(r_2)$ is the unique nontrivial curve up to isotopy that is disjoint from each of $m_1, m_2, m_4, m_5, \cdots, m_n, t_{a_1}(b)$. Since we know that $h([x]) = \lambda([x])$ for all
these curves and $\lambda$ is edge preserving, we have $h([t_{a_2}(r_2)]) = \lambda([t_{a_2}(r_2)])$. 
This proves the statement of the lemma for $f=t_{a_2}$. 

Similarly, for $f=t_{a_j}$ when $j \in \{1, 3, 4, \cdots, n\}$, let $L_f = \{a_1, a_2, \cdots, a_{n}, b, r_j\}$. The set $L_f$ has trivial stabilizer. We know $\lambda([x])= h([x])$
$\forall \ x \in L_f$. We just need to check the equation for $t_{a_j}(b)$ and 
$t_{a_j}(r_j)$ since the other curves in $L_f$ are fixed by $t_{a_j}$. 
In the above paragraph we already obtained that $h([t_{a_j}(b)]) = \lambda([t_{a_j}(b)])$.  
When $j < n$, the curve $t_{a_j}(r_j)$ is the unique nontrivial curve up to isotopy that is disjoint from each of 
$m_1, m_2, \cdots, m_j, m_{j+2}, m_{j+3}, \cdots, m_n, t_{a_1}(b)$. 
Since we know that $h([x]) = \lambda([x])$ for all
these curves and $\lambda$ is edge preserving, we have $h([t_{a_j}(r_j)]) = \lambda([t_{a_j}(r_j)])$ when $j < n$. 
The curve $t_{a_n}(r_n)$ is the unique nontrivial curve up to isotopy that is disjoint from each of $m_2, m_3, \cdots, m_n, t_{a_1}(b)$. Since we know that $h([x]) = \lambda([x])$ for all
these curves and $\lambda$ is edge preserving, we have $h([t_{a_n}(r_n)]) = \lambda([t_{a_n}(r_n)])$. 
Hence, we obtain the statement of the lemma for $f=t_{a_j}$ for all $j \in \{1, 2, \cdots, n\}$.
  
For $f = \sigma_i$, where $i \in \{2, 3, \cdots, n\}$, we let $L_f = \{a_1, a_2, \cdots, a_{n}, b, r_o\}$ where $r_o \in \{r_1, r_2, \cdots, r_n\}$ such that $r_o$ is disjoint from $m_i$. We know that $\lambda ([x]) = h ([x])$ for all $x \in L_f$. We just need to check $h ([ \sigma_i (a_i) ]) = \lambda ([ \sigma_i (a_i) ])$ for each $i$ since the other curves in $L_f$ are fixed by 
$\sigma_i$. For $i=2$, we use the curve $u_1$ shown in Figure \ref{fig-12s} (i). The curve $u_1$ is the unique nontrivial curve up to isotopy that is disjoint from $a_3, a_4, \cdots, a_n, b, r_1$. Since we know that $h([x]) = \lambda([x])$ for all
these curves and $\lambda$ is edge preserving, we have $h([u_1]) = \lambda([u_1])$. The curve $\sigma_2 (a_2)$, which is shown as $j_1$ in Figure \ref{fig-12s} (ii), is the unique curve up to isotopy disjoint from $a_1, a_3, a_4, \cdots, a_n, u_1$ which intersects $b$ once and nonisotopic to $a_3$.
Since we know that $h([x]) = \lambda([x])$ for all these curves and $\lambda$ preserves these properties, we see that 
$h([\sigma_2 (a_2)]) = \lambda ([\sigma_2 (a_2)])$. For $i=3$, we use the curve $u_2$ shown in Figure \ref{fig-12s} (iii). The curve $u_2$ is the unique nontrivial curve up to isotopy that is disjoint from $a_4, a_5, \cdots, a_n, b, r_1, r_2$. Since we know that $h([x]) = \lambda([x])$ for all
these curves and $\lambda$ is edge preserving, we have $h([u_2]) = \lambda([u_2])$. The curve $\sigma_3 (a_3)$, which is shown as $j_2$ in Figure \ref{fig-12s} (iv), is the unique curve up to isotopy disjoint from $a_1, a_2, a_4, a_5, \cdots, a_n, u_2$ which intersects $b$ once and nonisotopic to $a_4$.
Since we know that $h([x]) = \lambda([x])$ for all these curves and $\lambda$ preserves these properties, we see that 
$h([\sigma_3 (a_3)]) = \lambda ([\sigma_3 (a_3)])$. 
Similarly we get $h ([ \sigma_i (a_i) ]) = \lambda ([ \sigma_i (a_i) ])$ for each $i = 4, 5, \cdots, n$.\end{proof}

\begin{theorem} \label{A} There exists a homeomorphism $h : R \rightarrow R$ such that $H(\alpha) = \lambda(\alpha)$ for every vertex $\alpha$ in $\mathcal{C}(R)$ where $H=[h]$ and this homeomorphism is unique up to isotopy.\end{theorem}
 
\begin{proof} Let $f \in G$. There exists $L_f \subset \mathcal{C}(R)$ which satisfies the statement of Lemma \ref{prop-imp}. Consider $\mathcal{C}$ given in Lemma \ref{curves}. Let $\mathcal{X}= \mathcal{C} \cup \big( \bigcup _{f \in G} (L_f \cup f(L_f) \big ))$. For each vertex $x$ in the curve complex, there exists $r \in Mod_R$ and a vertex $y$ in the set $\mathcal{X}$ such that $r(y)=x$. By following the construction given in \cite{IrP1}, we let $\mathcal{X}_1 = \mathcal{X}$ and $\mathcal{X}_k = \mathcal{X}_{k-1} \cup (\bigcup _{f \in G} (f(\mathcal{X}_{k-1}) \cup f^{-1}(\mathcal{X}_{k-1})))$ when $k \geq 2$. We observe that 
$\mathcal{C}(R) = \bigcup _{k=1} ^{\infty} \mathcal{X}_k$. We will prove that $h([x]) = \lambda([x])$ for all $x \in \mathcal{X}_{k}$ for each $k \geq 1$. We will give the proof by induction on $k$. By using Lemma \ref{curves} and Lemma \ref{prop-imp}, we see that $h([x]) = \lambda([x])$ for each
$x \in \mathcal{X}_1$. Assume that $h([x]) = \lambda([x])$ for all $x \in \mathcal{X}_{k-1}$ for some $k \geq 2$. Let $f \in G$. There exists a homeomorphism $h_f$ of $R$ such that $h_f([x]) = \lambda([x])$ 
for all $x \in f(\mathcal{X}_{k-1})$. We have $f(L_f) \subset \mathcal{X}_{k-1} \cap f(\mathcal{X}_{k-1})$. This implies that  we have $h_f = h$ since $f(L_f)$ has trivial stabilizer. Similarly, there exists a homeomorphism 
$h'_f$ of $R$ such that $h'_f([x]) = \lambda([x])$ for all $x \in f^{-1}(\mathcal{X}_{k-1})$. We have $L_f \subset \mathcal{X}_{k-1} \cap f^{-1}(\mathcal{X}_{k-1})$. This implies that we have $h'_f = h$ since $L_f$ has trivial stabilizer. So, $h([x]) = \lambda([x])$ for each $x \in \mathcal{X}_k$. Hence, by induction 
$h([x]) = \lambda([x])$ for each $x \in \mathcal{X}_k$ for all $k \geq 1$. Since 
$\mathcal{C}(R) = \bigcup _{k=1} ^{\infty} \mathcal{X}_k$, we have $h([x]) = \lambda([x])$ for every vertex $[x] \in \mathcal{C}(R)$. It is easy to see that this homeomorphism is unique up to isotopy.\end{proof}

\section{Edge Preserving Maps of $\mathcal{C}(R)$ when $g=0, n \geq 5$}

In this section we will always assume that $g = 0$, $n \geq 5$ and $\lambda :\mathcal{C}(R) \rightarrow \mathcal{C}(R)$ is an edge preserving map. As in the second section we have the following two lemmas.

\begin{lemma}
	\label{inj-2} The map $\lambda$ is injective on every set of vertices in $\mathcal{C}(R)$ if each pair in the set has geometric intersection zero.
\end{lemma}
 
\begin{lemma} \label{pd-inj-2} Let $P$ be a pants decomposition on $R$. 
	A set of pairwise disjoint representatives of $\lambda([P])$ is a pants decomposition on $R$.\end{lemma}
 
\begin{figure}
	\begin{center}  \epsfxsize=2.2in \epsfbox{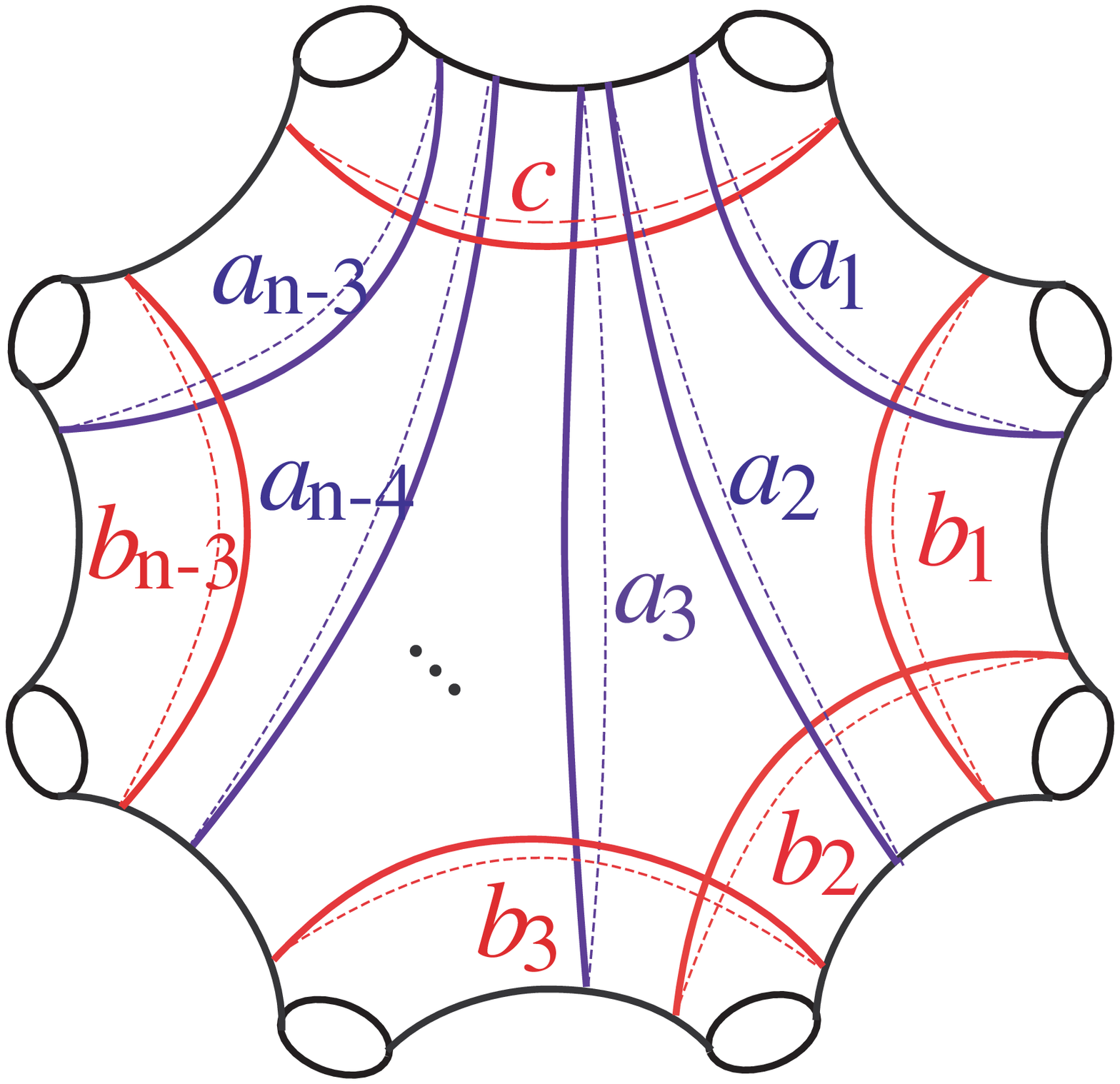} \hspace{0.2cm} 
		\epsfxsize=2.2in \epsfbox{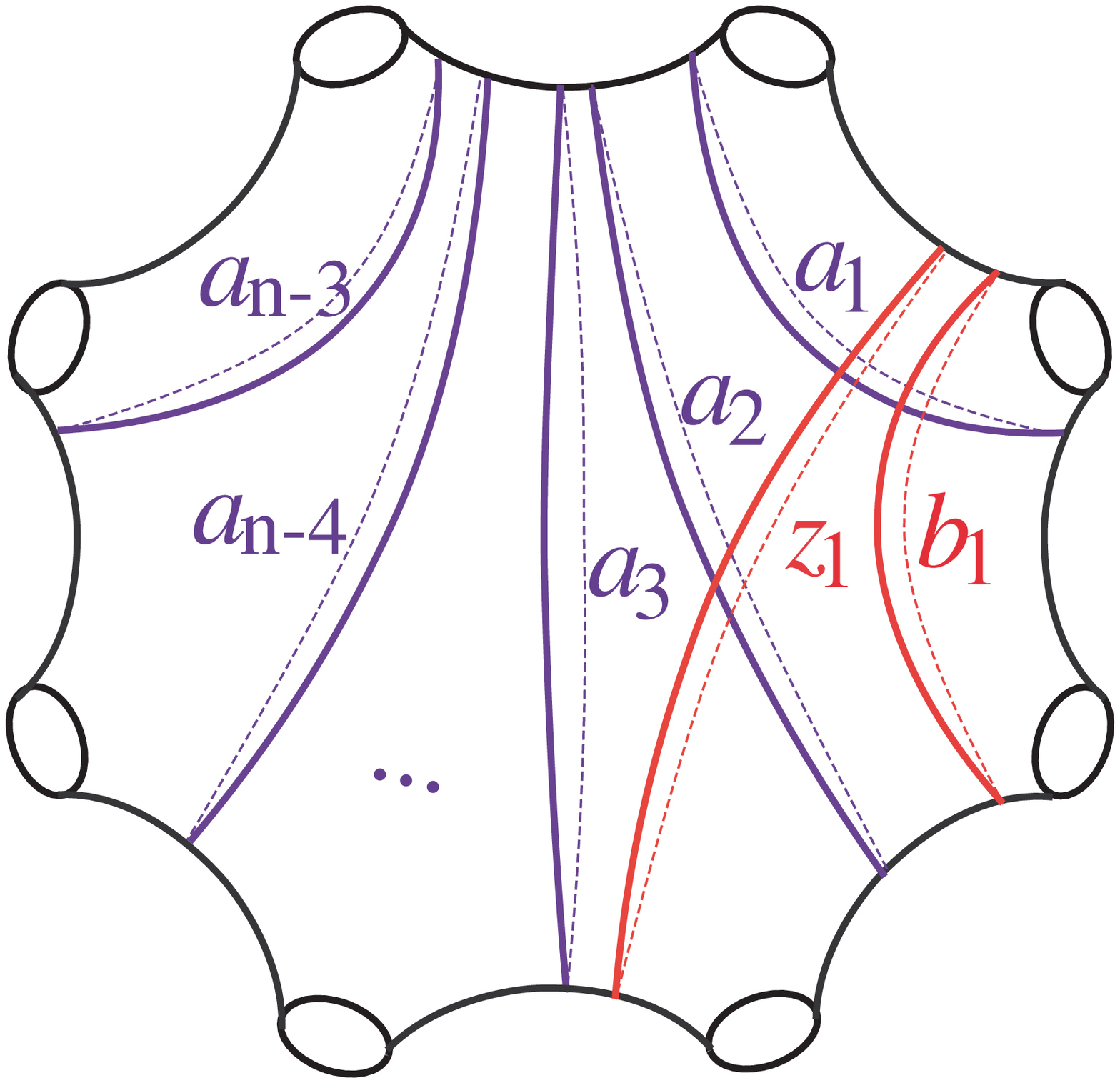}
	
		(i)   \hspace{5.2cm} (ii)  
			
		\epsfxsize=2.2in \epsfbox{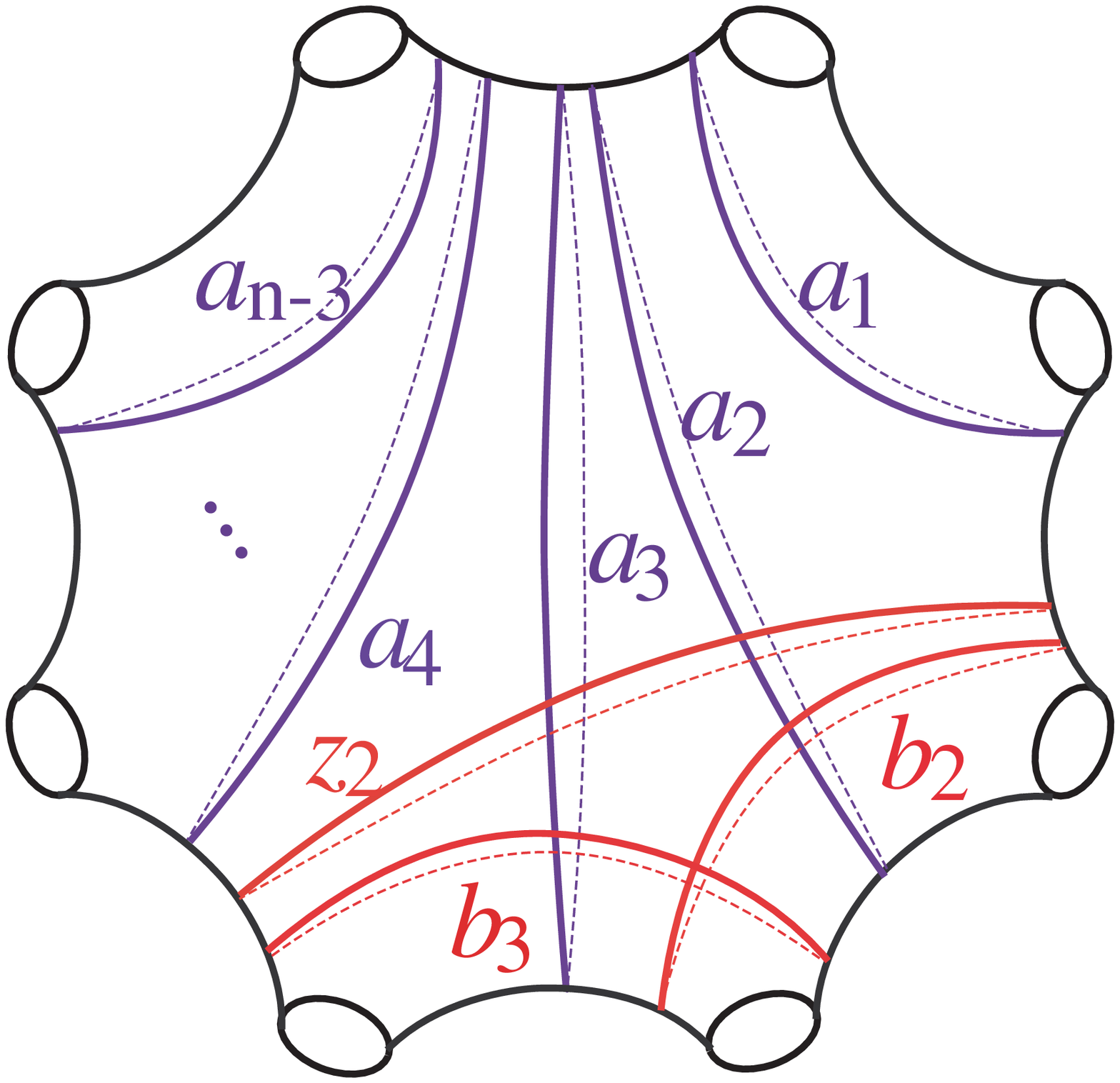}\hspace{0.2cm} 
		\epsfxsize=2.2in \epsfbox{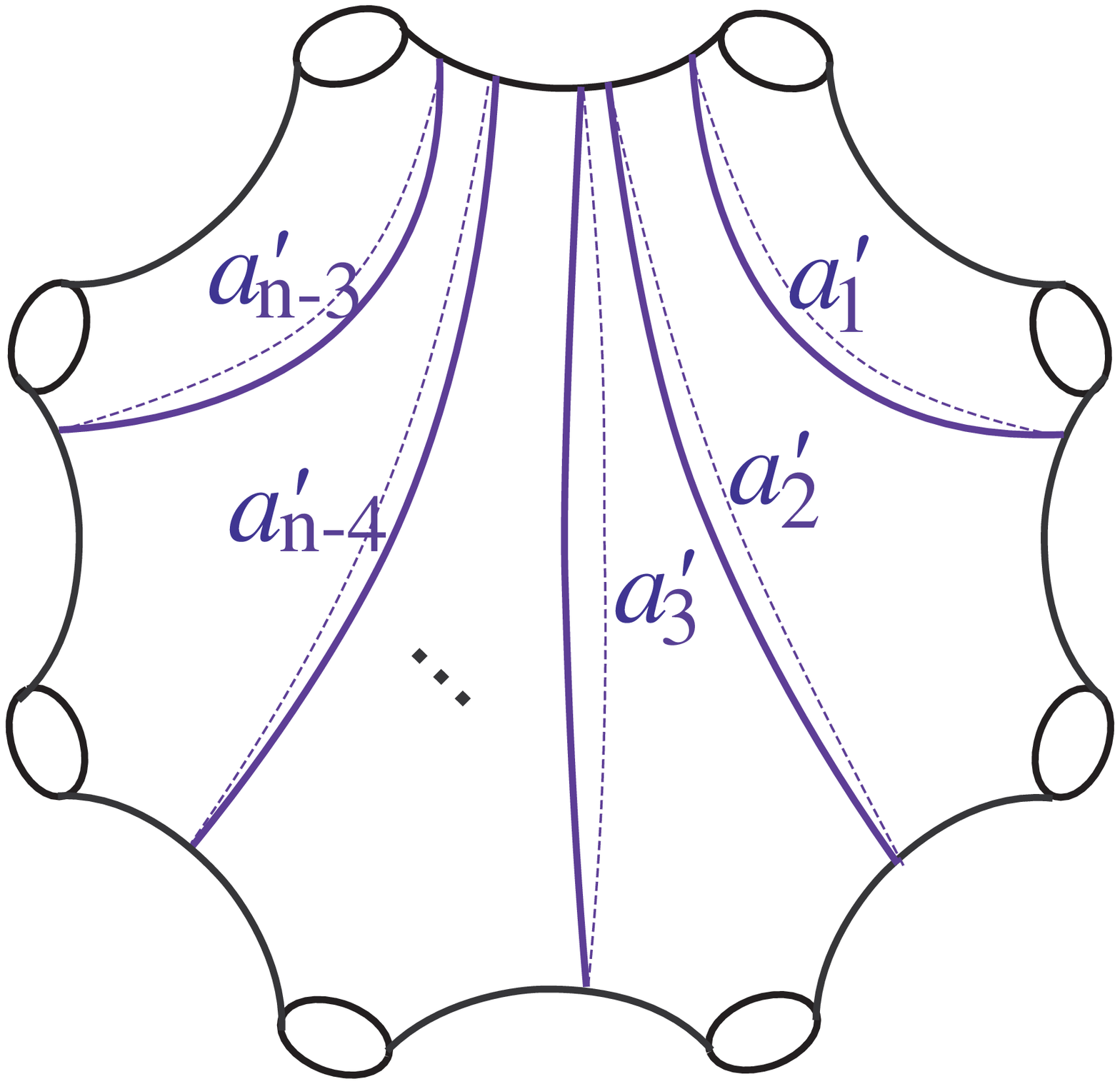} \hspace{0.2cm} 
		
			(iii)   \hspace{5.2cm} (iv) 
			
			\epsfxsize=2.2in \epsfbox{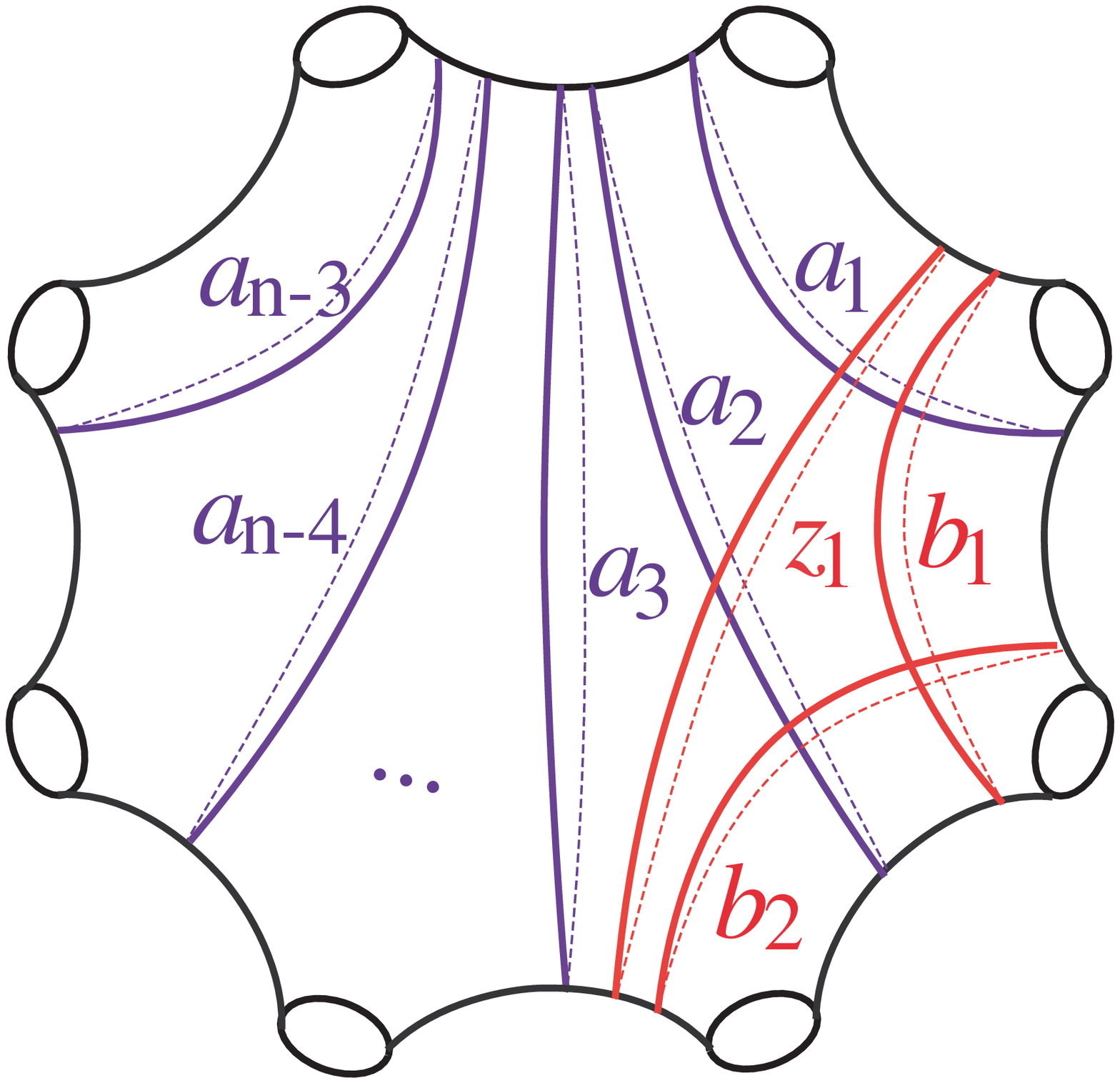} \hspace{0.2cm} 
		\epsfxsize=2.2in \epsfbox{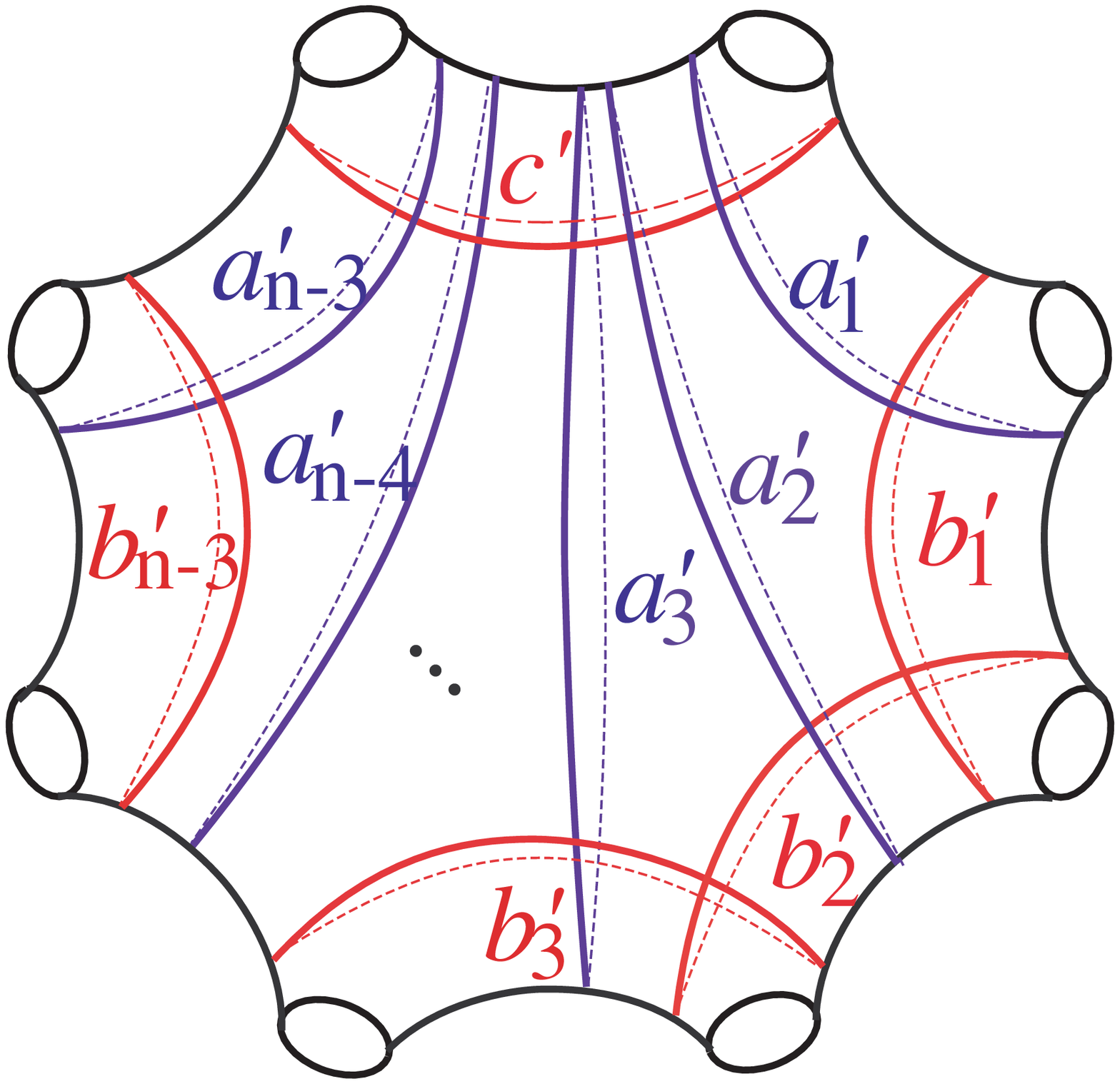}  
		
		(v)   \hspace{5.2cm} (vi)	
		\caption {Curves in $\mathcal{C}_1$} \label{fig-s10}
	\end{center}
\end{figure}

Let $\mathcal{C}_1 = \{a_1, a_2, a_3, \cdots, a_{n-3}, b_1, b_2, b_3, \cdots, b_{n-3}, c\}$ 
where the curves are as shown in Figure \ref{fig-s10} (i). Let $P = \{a_1, a_2, a_3, \cdots, a_{n-3}\}$. Let $P'$ be a pair of pants
decomposition of $R$ such that $\lambda([P]) = [P']$. For all $i$, let $a'_i$ be the representative of $\lambda([a_i])$ in $P'$, $b'_i$ be the representative of $\lambda([b_i])$ that intersects the elements of $P'$ minimally, $c'$ be the representative of $\lambda([c])$ that intersects the elements of $P' \cup \{b'_1, b'_2, b'_3, \cdots, b'_{n-3}\}$ minimally.   
 
\begin{lemma} \label{1} We have $i([a'_i], [b'_i]) \neq 0$, for all $i = 1, 2, \cdots, n-3$.\end{lemma}

\begin{proof} We will first show that $i([a'_1], [b'_1]) \neq 0$. We see that $i([b_1], [x]) = 0$ for all $x \in P \setminus \{a_1\}$ and there is an edge between $[b_1]$ and $[x]$ for all $x \in P \setminus \{a_1\}$. Since $\lambda$ is edge preserving we have $i(\lambda([b_1]), \lambda([x])) = 0$ for all $x \in P \setminus \{a_1\}$ and there is an edge between $\lambda([b_1])$ and $\lambda([x])$ for all $x \in P \setminus \{a_i\}$.
This implies that either $i(\lambda([b_1]), \lambda([a_1])) \neq 0$ or $\lambda([b_1]) = \lambda([a_1])$. With a similar argument we can see that 
$i(\lambda([c]), \lambda([a_1])) \neq 0$ or $\lambda([c]) = \lambda([a_1])$. If $\lambda([b_1]) = \lambda([a_1])$, then we couldn't have $i(\lambda([c]), \lambda([a_1])) \neq 0$ or $\lambda([c]) = \lambda([a_1])$ since $\lambda([c])$ and $\lambda([b_1])$ are connected by an edge. So, $i(\lambda([b_1]), \lambda([a_1])) \neq 0$. To see that 
$i(\lambda([b_2]), \lambda([a_2])) \neq 0$ we do the following: We see that $i([b_2], [x]) = 0$ for all $x \in P \setminus \{a_2\}$ and there is an edge between $[b_2]$ and $[x]$ for all $x \in P \setminus \{a_2\}$. Since $\lambda$ is edge preserving we have $i(\lambda([b_2]), \lambda([x])) = 0$ for all $x \in P \setminus \{a_2\}$ and there is an edge between $\lambda([b_2])$ and $\lambda([x])$ for all $x \in P \setminus \{a_2\}$.
This implies that either $i(\lambda([b_2]), \lambda([a_2])) \neq 0$ or $\lambda([b_2]) = \lambda([a_2])$. Since $i(\lambda([b_1]), \lambda([a_1])) \neq 0$ and there is a homeomorphism sending the pair 
$(a_1, b_1)$ to $(b_1, b_2)$ we can see that $i(\lambda([b_1]), \lambda([b_2])) \neq 0$. If $\lambda([b_2]) = \lambda([a_2])$, then we couldn't have $i(\lambda([b_1]), \lambda([b_2])) \neq 0$ since $\lambda([b_1])$ and $\lambda([a_2])$ are connected by an edge. So, $i(\lambda([b_2]), \lambda([a_2])) \neq 0$.
With similar arguments we get $i(\lambda([b_i]), \lambda([a_i])) \neq 0$ for all $i= 1, 2, \cdots, n-3$.
\end{proof} 

\begin{lemma}
\label{adj-f} The curves $a'_i, a'_{i+1}$ are adjacent to each other w.r.t. $P'$ for all $i = 1, 2, \cdots,$ $n-4$.\end{lemma}

\begin{proof} We will first prove that $a'_1, a'_2$ are adjacent to each other w.r.t. $P'$. Let $z_1$ be the curve shown in Figure \ref{fig-s10} (ii). The set $Q = (P \setminus \{a_1\}) \cup \{b_1\}$ is a pants decomposition on $R$. We see that $i([z_1], [x]) = 0$ for all $x \in Q \setminus \{a_2\}$ and there is an edge between $[z_1]$ and $[x]$ for all $x \in Q \setminus \{a_2\}$. Since $\lambda$ is edge preserving we have $i(\lambda([z_1]), \lambda([x])) = 0$ for all $x \in Q \setminus \{a_2\}$ and there is an edge between $\lambda([z_1])$ and $\lambda([x])$ for all $x \in Q \setminus \{a_2\}$. This implies that either $i(\lambda([z_1]), \lambda([a_2])) \neq 0$ or $\lambda([z_1]) = \lambda([a_2])$. Since $i(\lambda([z_1]), \lambda([b_2])) = 0$ and $i(\lambda([a_2]), \lambda([b_2])) \neq 0$ by Lemma \ref{1}, we cannot have $\lambda([z_1]) = \lambda([a_2])$. So, 
$i(\lambda([z_1]), \lambda([a_2])) \neq 0$. 
Since $i(\lambda([a_2]), \lambda([b_2])) \neq 0$ by Lemma \ref{1} and there is a homeomorphism sending the pair 
$(a_2, b_2)$ to $(z_1, a_1)$ we can see that $i(\lambda([z_1]), \lambda([a_1])) \neq 0$. Since 
$i(\lambda([z_1]), \lambda([a_1])) \neq 0$ and 
$i(\lambda([z_1]), \lambda([a_2])) \neq 0$ and 
$i(\lambda([z_1]), \lambda([x])) = 0$ for all 
$x \in P \setminus \{a_1, a_2\}$, we see that  $a'_1, a'_2$ are adjacent to each other w.r.t. $P'$. 

Consider the curve $z_2$ given in Figure \ref{fig-s10} (iii). 
The set $T = (P \setminus \{a_2\}) \cup \{b_2\}$ is a pants decomposition on $R$. We see that $i([z_2], [x]) = 0$ for all $x \in T \setminus \{a_3\}$ and there is an edge between $[z_2]$ and $[x]$ for all $x \in T \setminus \{a_3\}$. Since $\lambda$ is edge preserving we have $i(\lambda([z_2]), \lambda([x])) = 0$ for all $x \in T \setminus \{a_3\}$ and there is an edge between $\lambda([z_2])$ and $\lambda([x])$ for all $x \in T \setminus \{a_3\}$. This implies that either $i(\lambda([z_2]), \lambda([a_3])) \neq 0$ or $\lambda([z_2]) = \lambda([a_3])$. Since $i(\lambda([z_2]), \lambda([b_3])) = 0$ and $i(\lambda([a_3]), \lambda([b_3])) \neq 0$ by Lemma \ref{1}, we cannot have $\lambda([z_2]) = \lambda([a_3])$. So, 
$i(\lambda([z_2]), \lambda([a_3])) \neq 0$. 
The set $V = (P \setminus \{a_3\}) \cup \{b_3\}$ is a pants decomposition on $R$. We see that $i([z_2], [x]) = 0$ for all $x \in V \setminus \{a_2\}$ and there is an edge between $[z_2]$ and $[x]$ for all $x \in V \setminus \{a_2\}$. Since $\lambda$ is edge preserving we have $i(\lambda([z_2]), \lambda([x])) = 0$ for all $x \in V \setminus \{a_2\}$ and there is an edge between $\lambda([z_2])$ and $\lambda([x])$ for all $x \in V \setminus \{a_2\}$. This implies that either $i(\lambda([z_2]), \lambda([a_2])) \neq 0$ or $\lambda([z_2]) = \lambda([a_2])$. Since $i(\lambda([z_2]), \lambda([b_2])) = 0$ and $i(\lambda([a_2]), \lambda([b_2])) \neq 0$ by Lemma \ref{1}, we cannot have $\lambda([z_2]) = \lambda([a_2])$. So, 
$i(\lambda([z_2]), \lambda([a_2])) \neq 0$.  Since 
$i(\lambda([z_2]), \lambda([a_2])) \neq 0$ and 
$i(\lambda([z_2]), \lambda([a_3])) \neq 0$ and 
$i(\lambda([z_2]), \lambda([x])) = 0$ for all 
$x \in P \setminus \{a_2, a_3\}$, we see that  $a'_2, a'_3$ are adjacent to each other w.r.t. $P'$. Similarly, $a'_i, a'_{i+1}$ are adjacent to each other w.r.t. $P'$ for all $i = 1, 2, \cdots,$ $n-4$.\end{proof}

\begin{lemma} \label{nonadj-f} 
If $x, y \in P$ and $x$ is not adjacent to $y$ w.r.t. $P$, then $\lambda([x]), \lambda([y])$ have representatives in $P'$ which are not adjacent to each other w.r.t. $P'$.\end{lemma}

\begin{proof} It is enough to find two disjoint curves $w, t$ such that $w$ intersects only $x$ nontrivially and not the other curves in $P$, $t$ intersects only $y$ nontrivially and not the other curves in $P$ and $i(\lambda([w]), \lambda([x])) \neq 0$, $i(\lambda([t]), \lambda([y])) \neq 0$,  $i(\lambda([w]), \lambda([q])) = 0$, for all $q \in P \setminus \{ x\}$, $i(\lambda([t]), \lambda([q])) = 0$, for all $q \in P \setminus \{y\}$, $i(\lambda([t]), \lambda([w])) = 0$. By using Lemma \ref{1}, we can see that for the pair $a_i, a_j$, that are not adjacent to each other w.r.t. $P$, the curves $b_i, b_j$ 
would satisfy the above properties. So, we see that nonadjacency is preserved for every nonadjacent pair in $P$.\end{proof}

\begin{lemma} \label{curves-s} There exists a 
homeomorphism $h: R \rightarrow R$ such that $h([x]) = \lambda([x])$ 
$\forall \ x \in P= \{a_1, a_2, \cdots, a_{n-3}\}$.\end{lemma}

\begin{proof} The proof follows from Lemma \ref{adj-f} and Lemma \ref{nonadj-f}, see Figure \ref{fig-s10} (iv).\end{proof} 

\begin{lemma} \label{2} Let $i([a'_1], [b'_1]) =2, i([a'_{n-3}], [b'_{n-3}]) =2, i([a'_{n-3}], [c']) =2, i([c'], [a'_1]) =2$ and $i([b'_i], [b'_{i+1}]) =2  $, for all $i = 1, 2, \cdots, n-4$.\end{lemma}

\begin{proof} We will give the proof when $n \geq 6$. The proof is similar when $n = 5$. We will first show that $i([a'_1], [b'_1]) =2$. 
Consider the curves given in Figure \ref{fig-s10} (v). By Lemma \ref{curves-s}, we have $h([x]) = \lambda([x])$ $\forall \ x \in \{a_1, a_2, \cdots,$ $ a_{n-3}\}$. Let $a'_i$ be as shown in Figure \ref{fig-s10} (iv). Let $M'$ be the connected component of $R_{a'_3}$ (cut surface along $a'_3$) bounded by $a'_3$ and four boundary components of $R$ containing $a'_1$. Let $z'_1$ be a representative of $\lambda([z_1])$ which intersects minimally with all the elements in $\{a'_1, a'_2, a'_3, b'_1, b'_2\}$. By Lemma \ref{1}, we have $i([a'_1], [b'_1]) \neq 0$. 
Since there exists a homeomorphism sending the pair $(a_1, b_1)$ to 
$(b_1, b_2)$ and $i([a'_1], [b'_1]) \neq 0$, by using similar curve configurations we see that $i([b'_1], [b'_2]) \neq 0$. By Lemma \ref{1}, we have $i([a'_2], [b'_2]) \neq 0$. In the proof of Lemma \ref{adj-f}, we showed that $i([z'_1], [a'_1]) \neq 0$ and $i([z'_1], [a'_2]) \neq 0$. By using the intersection information for each pair of curves in $\{a'_1, a'_2, a'_3, b'_1, b'_2, z'_1\}$ and using that $\lambda$ is edge preserving we can see that the curves $a'_1, b'_2, z'_1, b'_1, a'_2$ form a pentagon in $C(R)$, see \cite{K1}. Since $a'_1$ is a curve that separates a pair of pants and there is a homeomorphism sending $a_1$ to $b_2$, by using similar curve configurations we can see that $b'_2$ is a curve that separates a pair of pants. Since $a'_2$ is a curve that separates a genus zero surface with four boundary components and there is a homeomorphism sending $a_2$ to $z_1$, we see that $z'_1$ is a curve that separates a genus zero surface with four boundary components on $R$. Using all this information about these curves and using Korkmaz's Theorem 3.2 in \cite{K1}, we get $i([a'_1], [b'_1]) =2$. 
Since for each of the remaining pair $(x, y)$ in the statement of the lemma there exists a homeomorphism sending the pair $(a_1, b_1)$ to $(x, y)$ and $i([a'_1], [b'_1]) =2$, by using similar curve configurations we get $i([x], [y]) =2$.\end{proof}\\
  
If $f: R \rightarrow R$ is a homeomorphism, then we will use the same notation for $f$ and $[f]$. Recall that $\mathcal{C}_1 = \{a_1, a_2, a_3, \cdots, a_{n-3}, b_1, b_2, b_3, \cdots, b_{n-3}, c\}$ 
where the curves are as shown in Figure \ref{fig-s10} (i). Let $\mathcal{C}_2 = \{w_1, w_2, \cdots, w_n, r_1, r_2, \cdots, r_n\}$ where the curves are as shown in Figure \ref{fig-s11}.
 
\begin{lemma} \label{curves-main} There exists a 
	homeomorphism $h: R \rightarrow R$ such that $h([x]) = \lambda([x])$ 
	$\forall \ x \in \mathcal{C}_1$.\end{lemma}

\begin{proof} The proof follows from Lemma \ref{curves-s}, Lemma \ref{2} and using that $\lambda$ is edge preserving, see Figure \ref{fig-s10} (vi).\end{proof}

\begin{figure}
	\begin{center}  \epsfxsize=2.2in \epsfbox{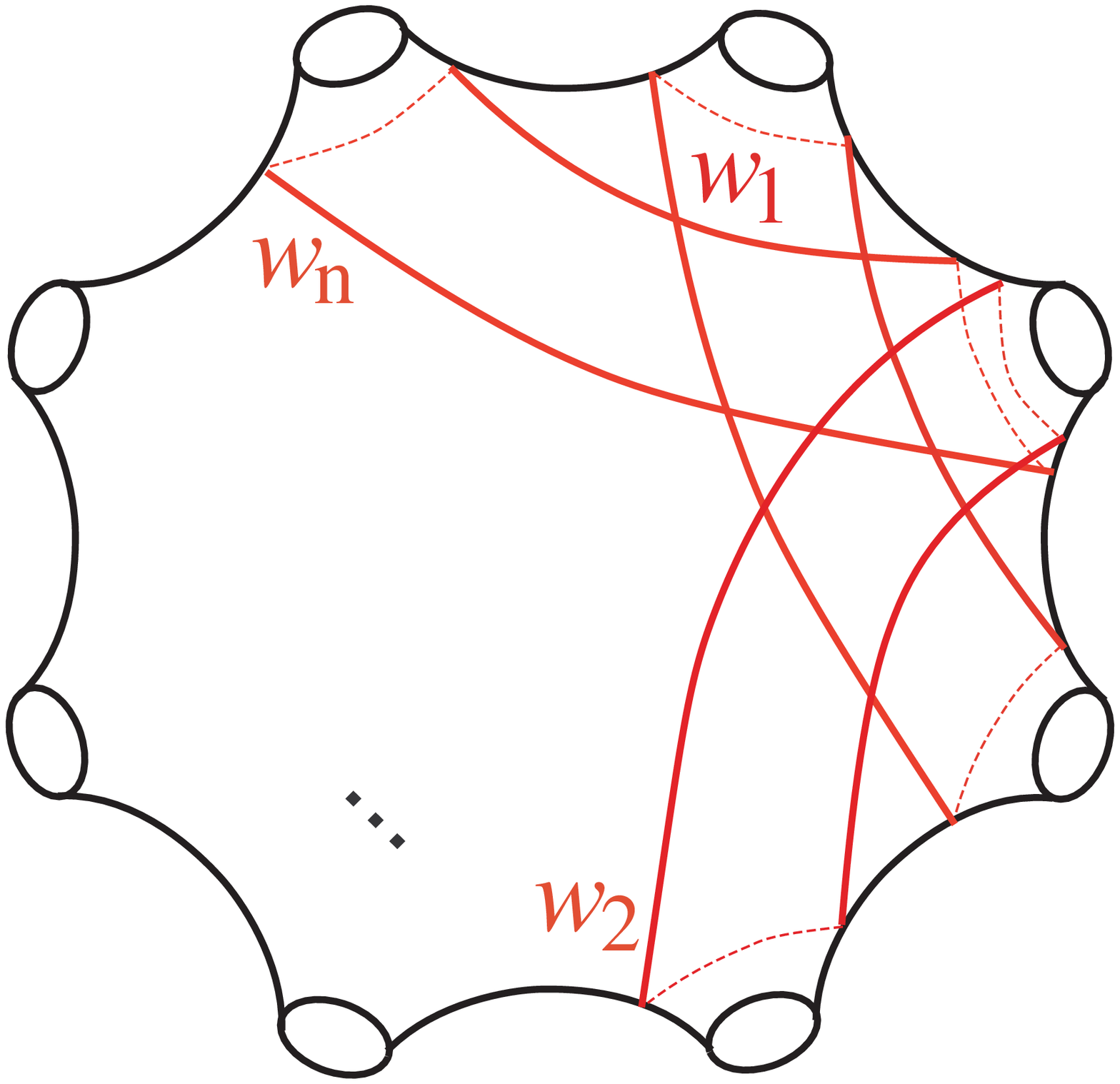} \hspace{0.2cm} 
		\epsfxsize=2.2in \epsfbox{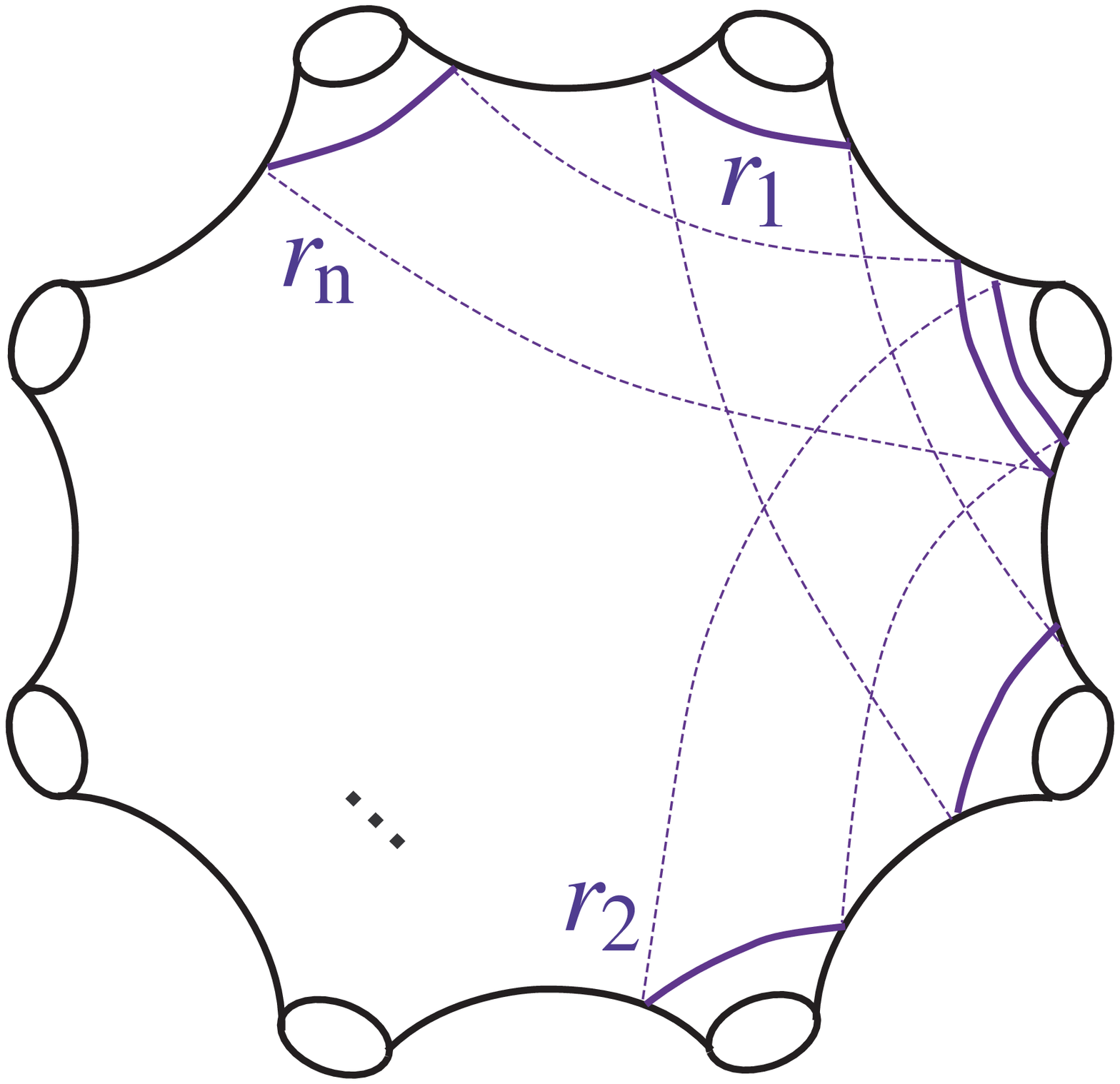}
		
		(i)   \hspace{5.2cm} (ii)  
		
		\caption {Curves in $\mathcal{C}_2$} \label{fig-s11}
	\end{center}
\end{figure}

\begin{lemma} \label{curves-ii-main} There exists a 
	homeomorphism $h: R \rightarrow R$ such that $h([x]) = \lambda([x])$ 
	$\forall \ x \in \mathcal{C}_1 \cup \mathcal{C}_2$.\end{lemma}

\begin{proof} Let $h: R \rightarrow R$ be a homeomorphism which satisfies the statement of Lemma \ref{curves-main}. We will give the proof when $n \geq 6$. The proof for $n=5$ is similar. Consider the curves in $\mathcal{C}_2$ given in Figure \ref{fig-s11}. There exists a
homeomorphism $\phi : R \rightarrow R$ of order two such that the map $\phi_{*}$ induced by $\phi$ on $\mathcal{C}(R)$ sends the isotopy class of each curve in $\mathcal{C}_1$ to itself and switches $[r_1]$ and $[w_1]$.
Since there is a homeomorphism sending the pair $(a_1, b_1)$ to 
$(a_1, r_1)$, by using Lemma \ref{2} we see that $i(\lambda[a_1], \lambda[r_1])= 2$. Similarly, we have $i(\lambda[a_1], \lambda[w_1])= 2$, 
$i(\lambda[b_1], \lambda[r_1])= 2$, $i(\lambda[b_1], \lambda[w_1])= 2$.
The curves $r_1$ and $w_1$ are the only nontrivial curves up to isotopy disjoint from $a_2$, 
and intersects each of $a_1, b_1$ nontrivially twice in the four holed sphere cut by $a_2$. Since we know that $h([x]) = \lambda([x])$ for all these curves, $\lambda$ preserves these properties, by replacing $\lambda$ with $\lambda \circ \phi_{*}$ if necessary, we can assume that we have $h([w_1]) = \lambda([w_1])$. 
  
  \begin{figure}
  	\begin{center} 
  		\epsfxsize=2.2in \epsfbox{small-fig-67.eps} \hspace{0.2cm} 	\epsfxsize=2.2in \epsfbox{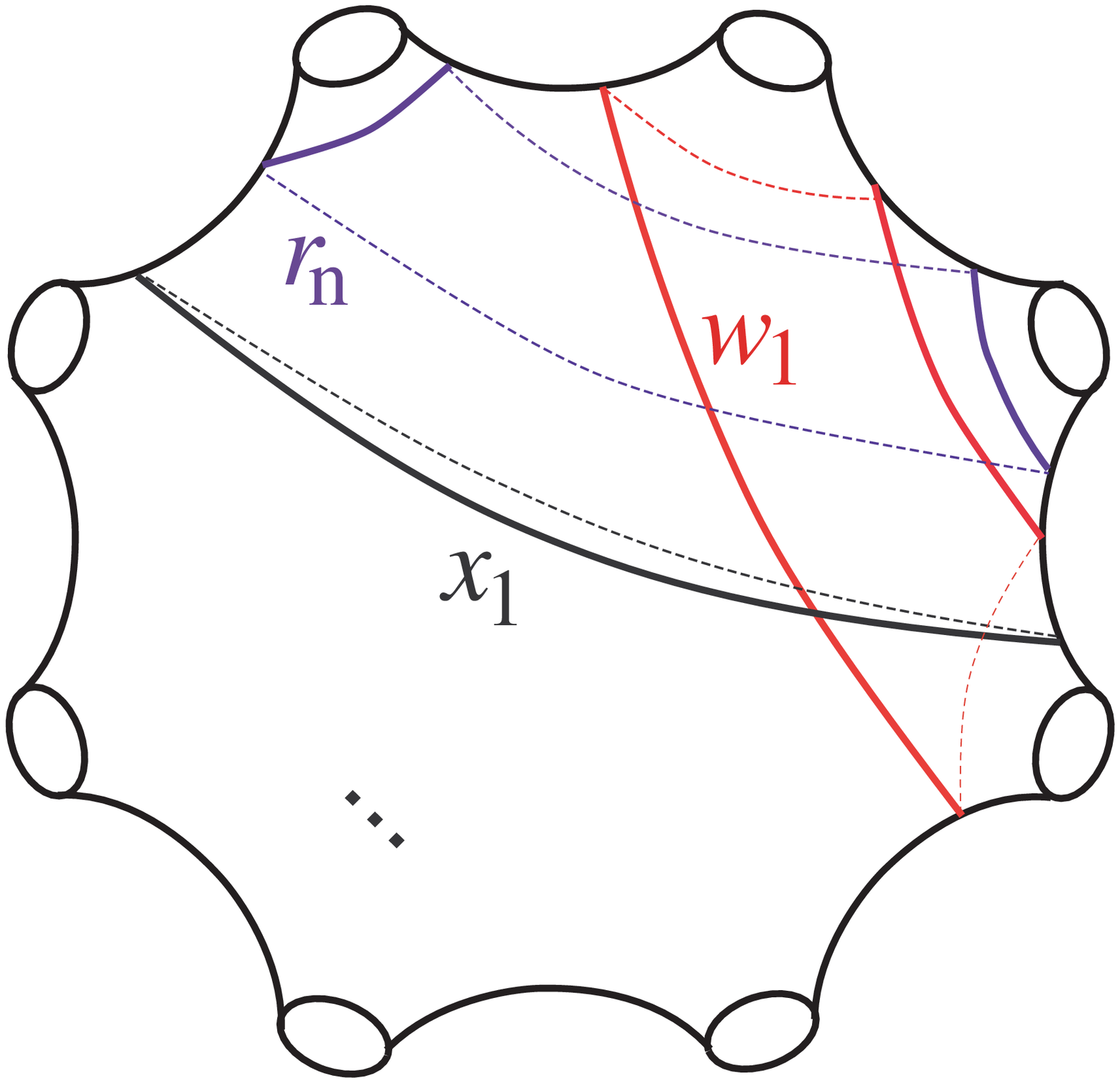}
  		
  		(i)   \hspace{5.2cm} (ii) 
  		
  		\epsfxsize=2.2in \epsfbox{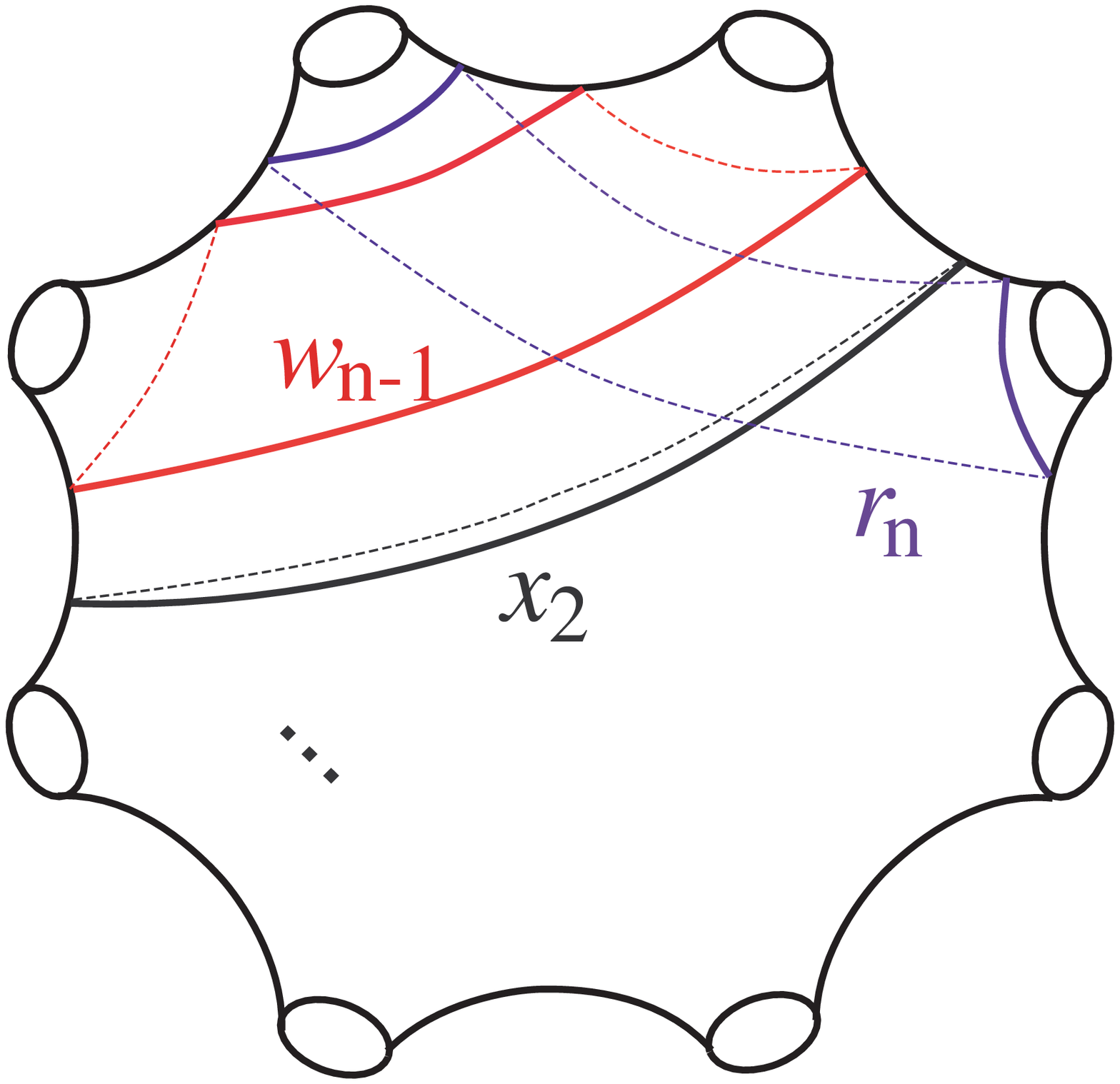} \hspace{0.2cm} 	\epsfxsize=2.2in \epsfbox{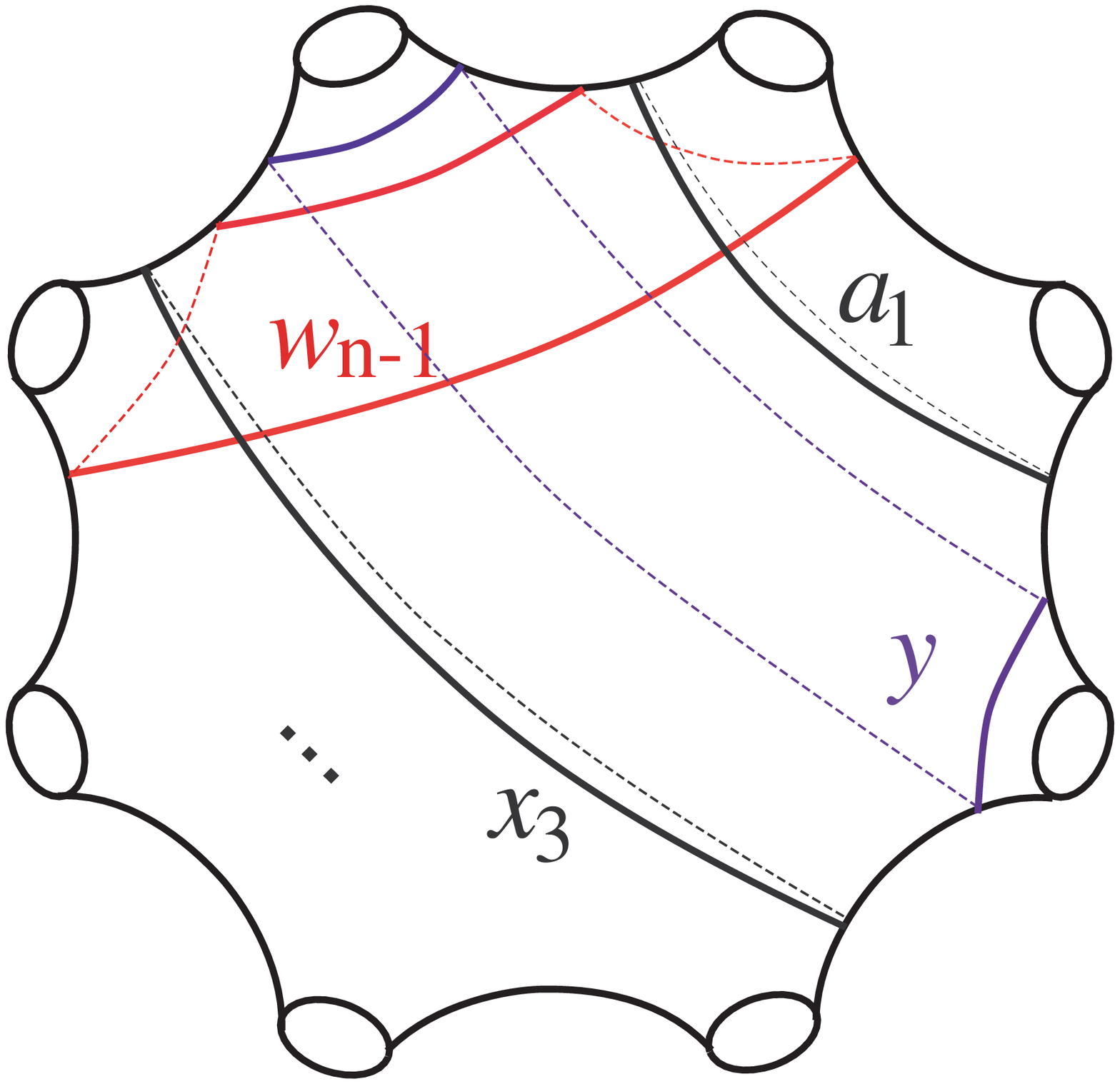}
  		
  		(iii)   \hspace{5.2cm} (iv) 
  		
  		\epsfxsize=2.2in \epsfbox{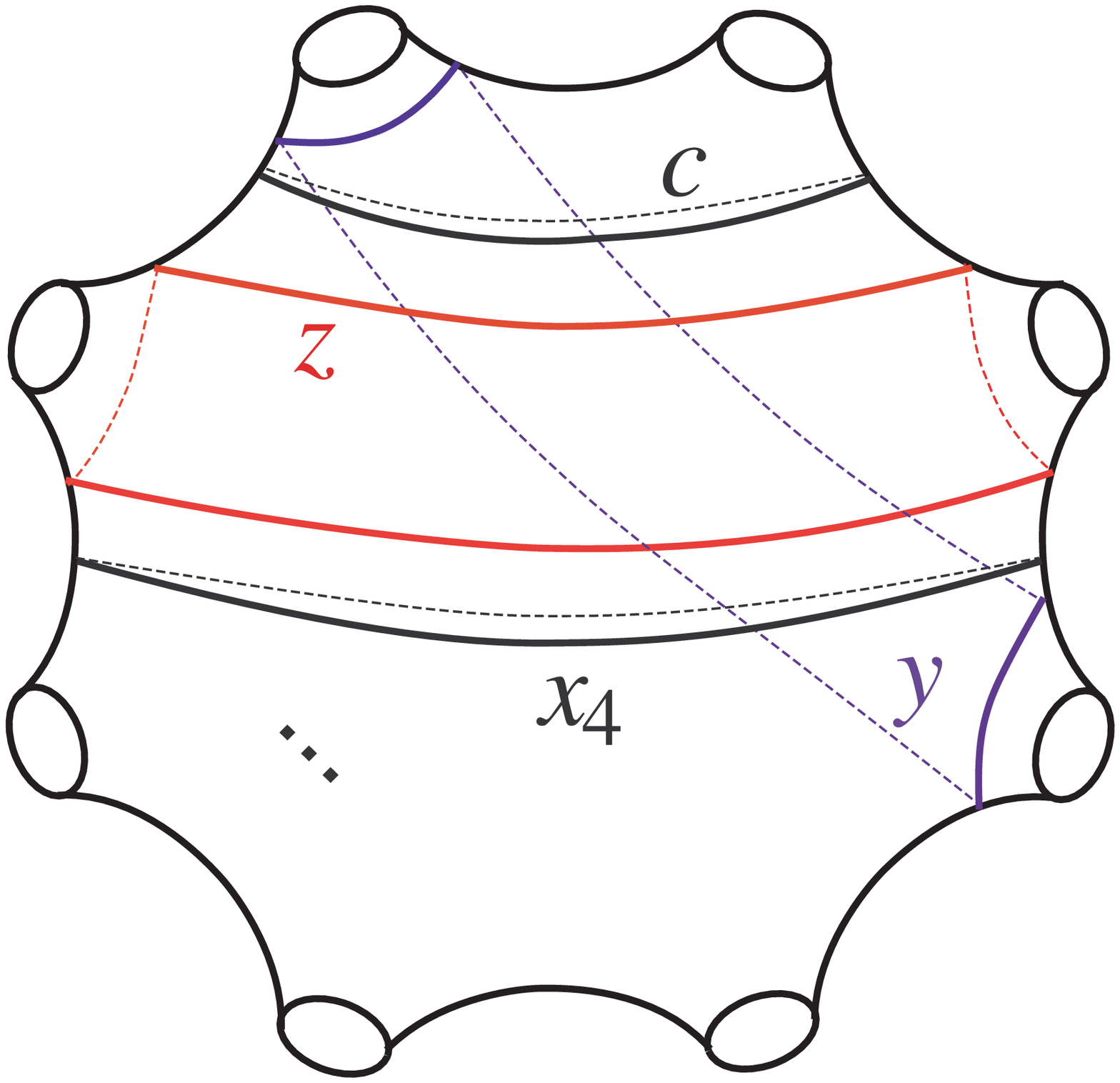} \hspace{0.2cm} 	\epsfxsize=2.2in \epsfbox{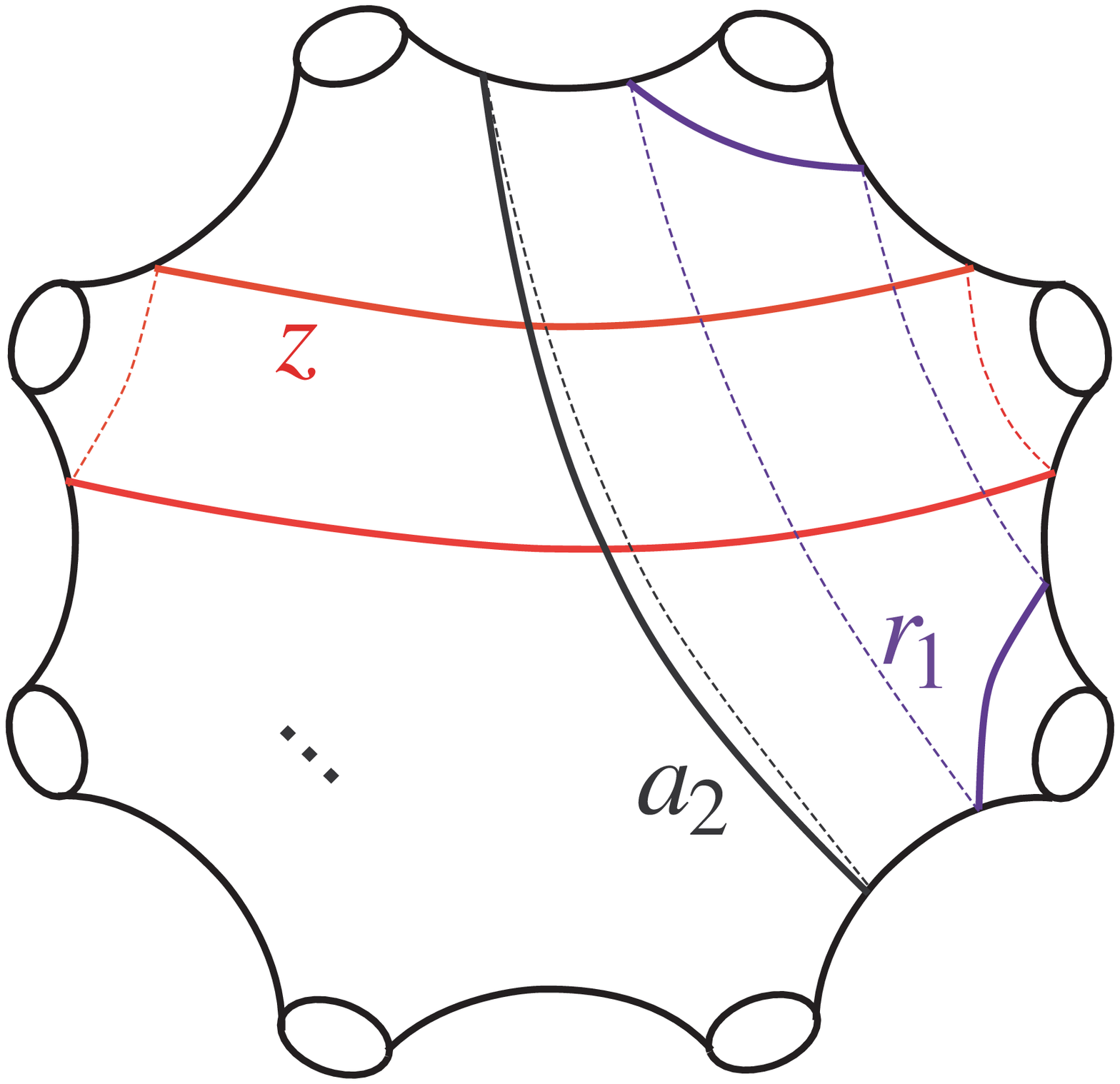}
  		
  		(v)   \hspace{5.2cm} (vi) 
  		\caption {Curves}
  		\label{fig-s12}
  	\end{center}
  \end{figure}

Consider the curves given in Figure \ref{fig-s12}.
The curve $x_1$ is the unique nontrivial curve up to isotopy that is disjoint from all the curves in $\{c, a_1, b_2, b_3, \cdots, b_{n-3}\}$. Since we know that $h([x]) = \lambda([x])$ for all these curves and $\lambda$ is edge preserving, we have $h([x_1]) = \lambda([x_1])$. The curve $r_n$ is the unique nontrivial curve up to isotopy that is nonisotopic to and disjoint from each curve in $\{x_1, w_1, b_2, b_3, \cdots, b_{n-3}\}$. Since we know that $h([x]) = \lambda([x])$ for all these curves and $\lambda$ preserves these properties, we have $h([r_n]) = \lambda([r_n])$.
The curve $x_2$ is the unique nontrivial curve up to isotopy that is disjoint from all the curves in $\{c, a_{n-3}, b_1, b_2, \cdots, b_{n-4}\}$. Since we know that $h([x]) = \lambda([x])$ for all these curves and $\lambda$ is edge preserving, we have $h([x_2]) = \lambda([x_2])$. The curve $w_{n-1}$ is the unique nontrivial curve up to isotopy that is nonisotopic to and disjoint from each curve in $\{x_2, r_n, b_1, b_2, \cdots, b_{n-4}\}$. Since we know that $h([x]) = \lambda([x])$ for all these curves and $\lambda$ preserves these properties, we have $h([w_{n-1}]) = \lambda([w_{n-1}])$. 

The curve $x_3$ is the unique nontrivial curve up to isotopy that is disjoint from all the curves in $\{c, a_1, b_1, b_3, b_4, \cdots, b_{n-3}\}$. Since we know that $h([x]) = \lambda([x])$ for all these curves and $\lambda$ is edge preserving, we have $h([x_3]) = \lambda([x_3])$. The curve $y$ is the unique nontrivial curve up to isotopy that is nonisotopic to and disjoint from all the curves in $\{a_1, x_3, w_{n-1}, b_3, b_4, \cdots, b_{n-3}\}$. Since we know that $h([x]) = \lambda([x])$ for all these curves and $\lambda$ preserves these properties, we have $h([y]) = \lambda([y])$.  
The curve $x_4$ is the unique nontrivial curve up to isotopy that is disjoint from all the curves in $\{c, a_1, b_2, b_3, \cdots, b_{n-4}, a_{n-3}\}$. Since we know that $h([x]) = \lambda([x])$ for all these curves and $\lambda$ is edge preserving, we have $h([x_4]) = \lambda([x_4])$. The curve $z$ is the unique nontrivial curve up to isotopy that is nonisotopic to and disjoint from all the curves in $\{c, y, x_4, b_2, b_3, \cdots, b_{n-4}\}$. Since we know that $h([x]) = \lambda([x])$ for all these curves and $\lambda$ preserves these properties, we have $h([z]) = \lambda([z])$. The curve $r_1$ is the unique nontrivial curve up to isotopy that is nonisotopic to and disjoint from each curve in $\{a_2, z, b_3, b_4, \cdots, b_{n-3}, a_{n-3}\}$. Since we know that $h([x]) = \lambda([x])$ for all these curves and $\lambda$ preserves these properties, we have $h([r_1]) = \lambda([r_1])$.
Hence we have $h([w_1]) = \lambda([w_1])$ and $h([r_1]) = \lambda([r_1])$. Similarly, we get $h([w_i]) = \lambda([w_i])$ for all $i=2, 3, \cdots, n$ and $h([r_i]) = \lambda([r_i])$ for all $i=2, 3, \cdots, n$.\end{proof}\\ 

We will use the notation $h_x$ for the half twist along $x$. Consider the curves in Figure \ref{fig-s10} (i). The group $Mod_R$ can be generated by 
$\{h_x: x \in \{a_1, b_1, b_2, \cdots, b_{n-3}, a_{n-3}, c\} \}$, see Corollary 4.15 in \cite{FM}. 
Let $G = \{h_x: x \in \{a_1, b_1, b_2, \cdots, b_{n-3}, a_{n-3}, c\} \}$. 
Let $h: R \rightarrow R$ be a homeomorphism which satisfies the statement of Lemma \ref{curves-ii-main}. We know $h([x]) = \lambda([x])$ $\forall \ x \in \mathcal{C}_1 \cup \mathcal{C}_2$. 
We will follow the techniques given by Irmak-Paris \cite{IrP2} to obtain the homeomorphism we want.

 \begin{figure}[t]
	\begin{center} 
		\epsfxsize=2.2in \epsfbox{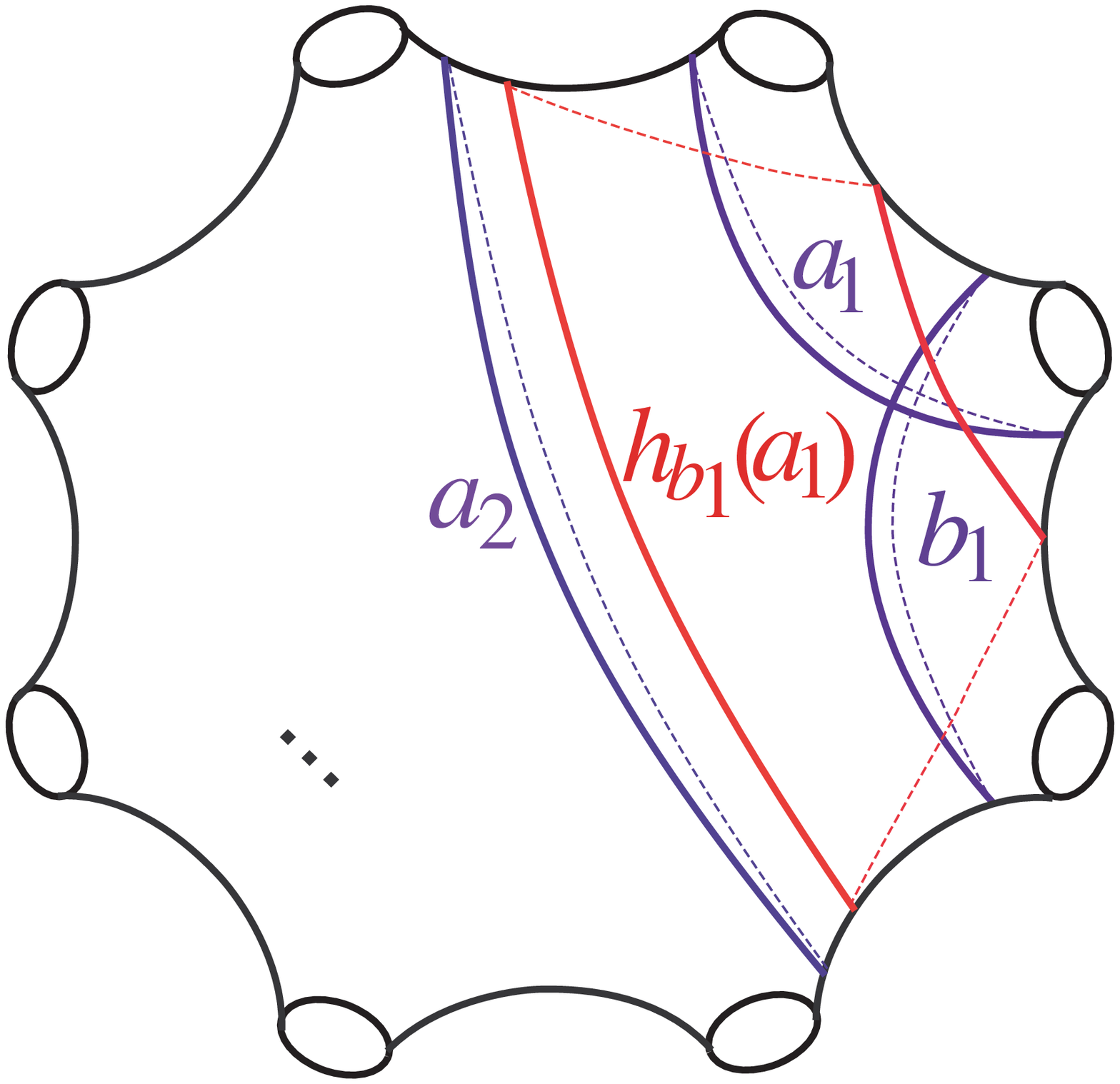} \hspace{0.2cm} 	\epsfxsize=2.2in \epsfbox{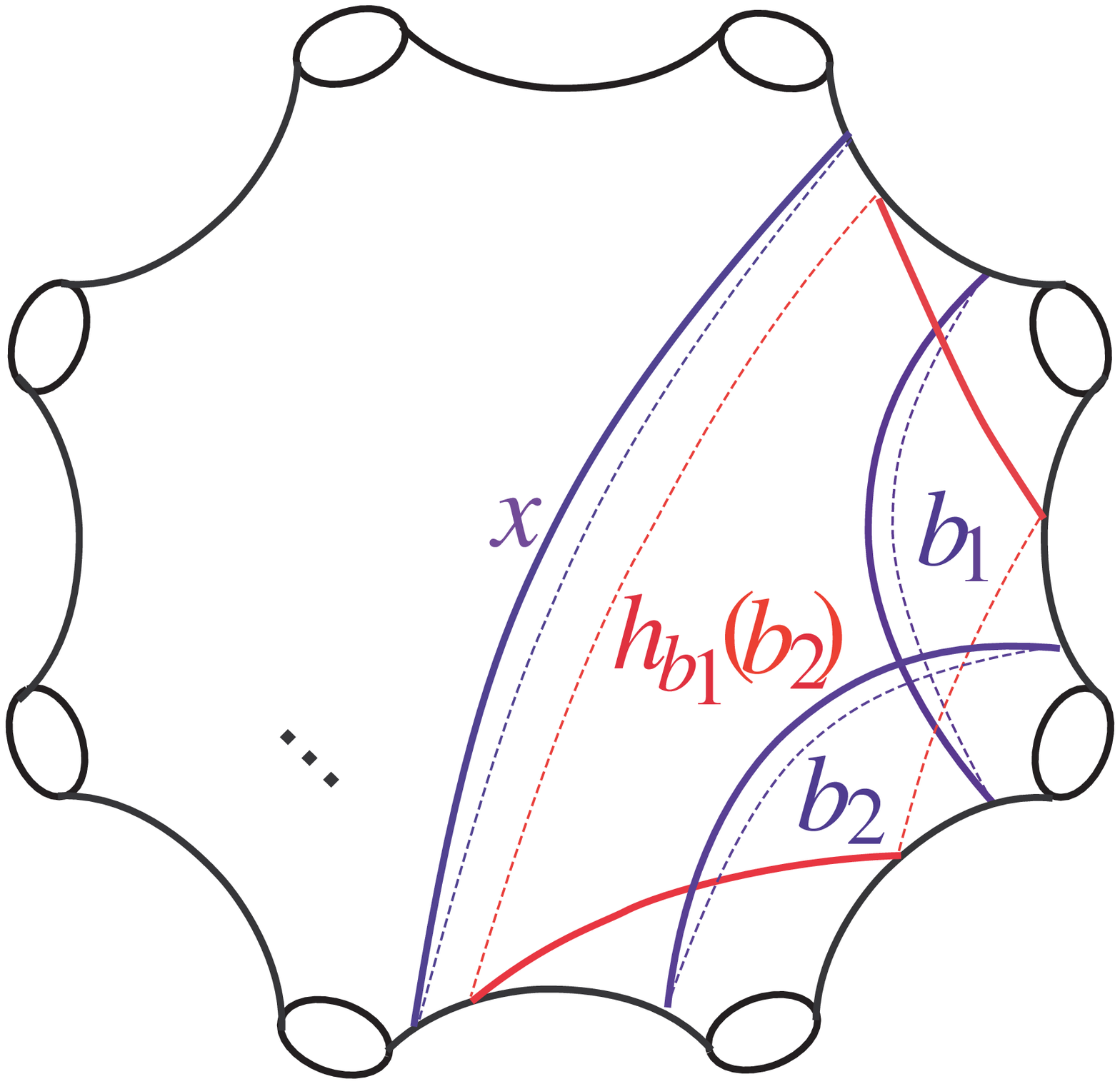}
		
		(i)   \hspace{5.2cm} (ii) 
		
		\epsfxsize=2.2in \epsfbox{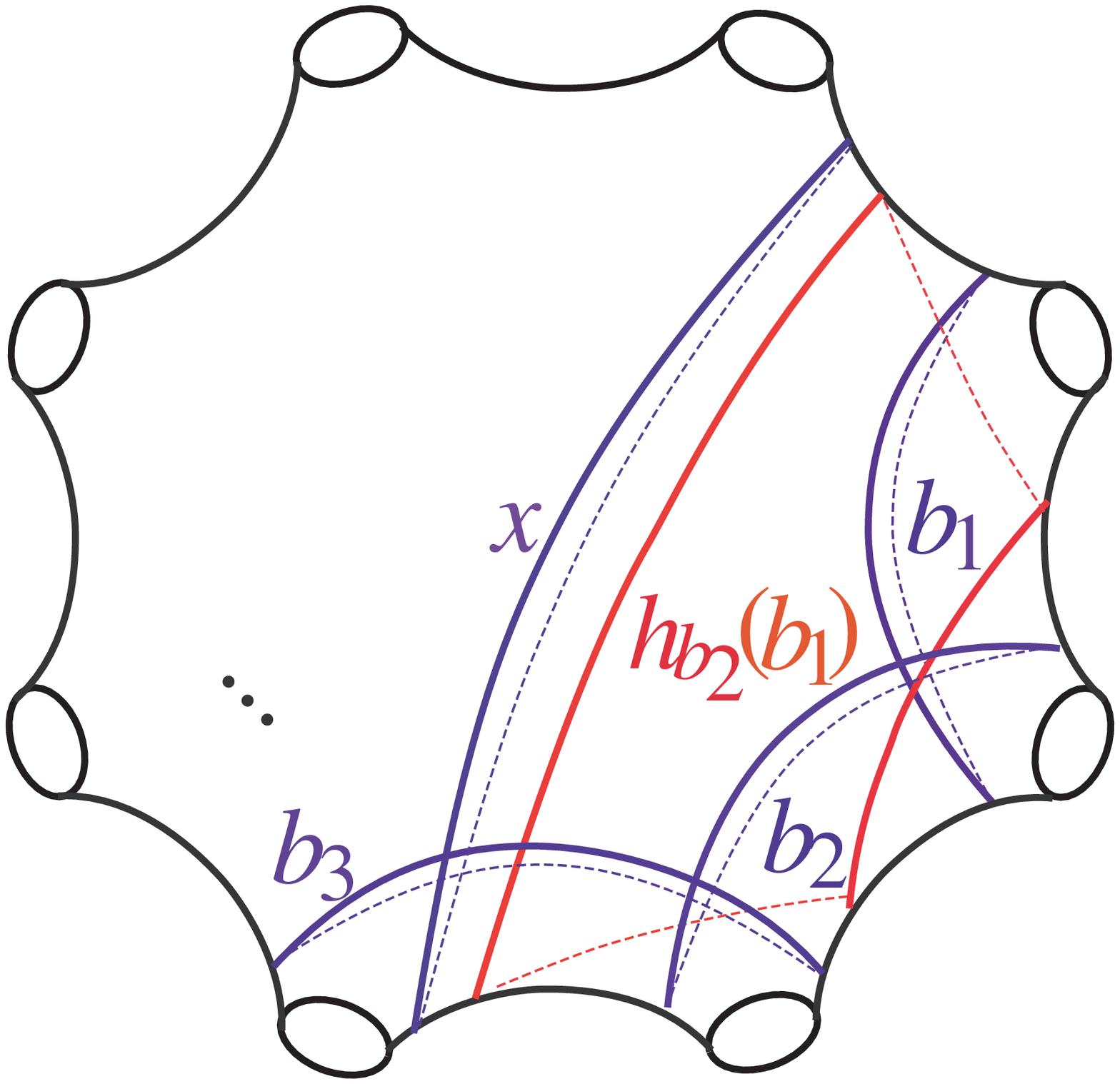} \hspace{0.2cm} 	\epsfxsize=2.2in \epsfbox{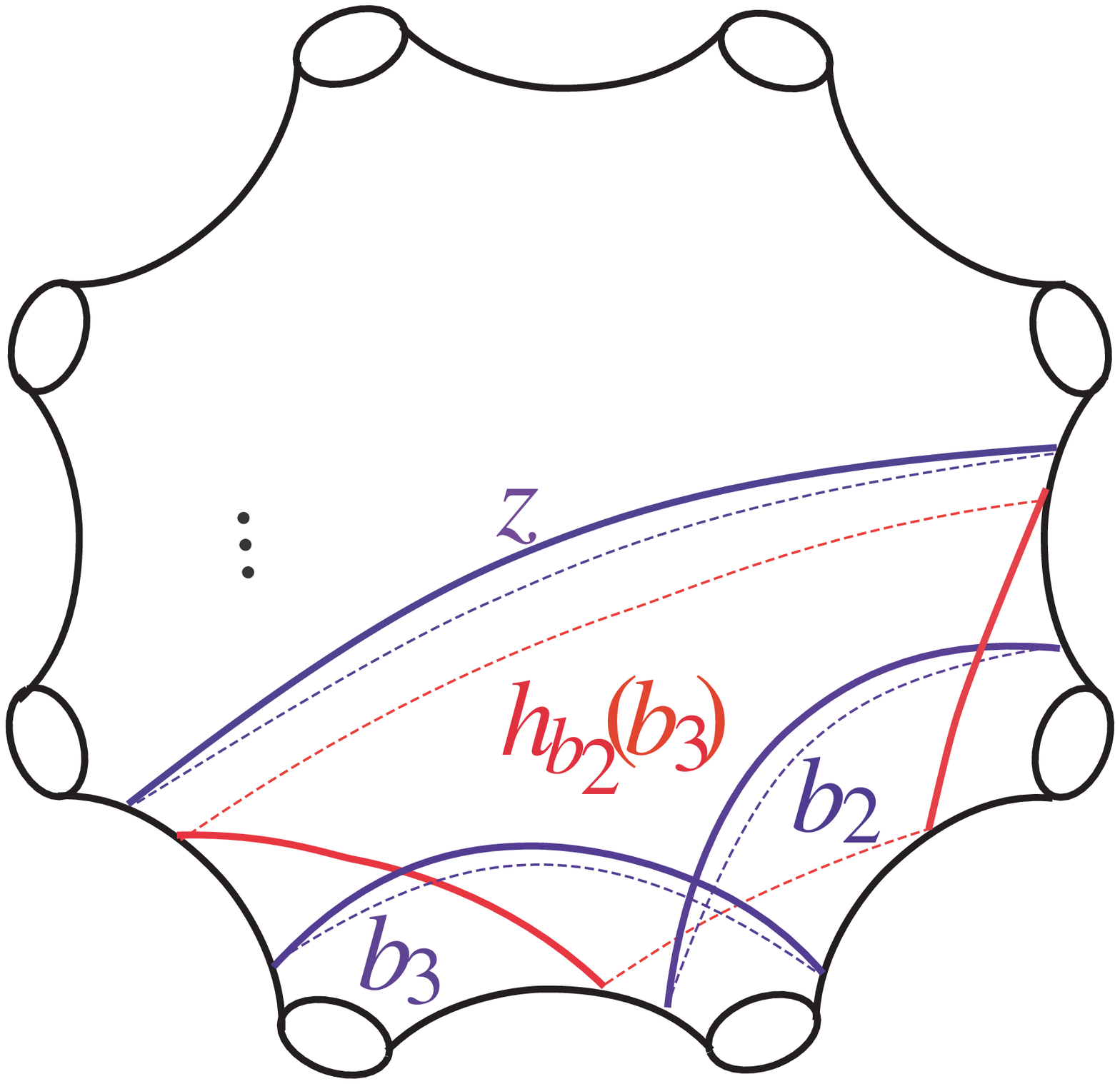}
		
		(iii)   \hspace{5.2cm} (iv) 
		\caption {Half-twists}
		\label{fig-s13}
	\end{center}
\end{figure}

\begin{lemma} \label{prop-imp2} $\forall \ f \in G$, 
$\exists$ a set $L_f \subset \mathcal{C}(R)$
such that $\lambda([x])= h([x])$ $\ \forall \ x \in L_f \cup f(L_f)$. $L_f$ can be chosen to have trivial stabilizer.\end{lemma}

\begin{proof} We have $h([x]) = \lambda([x])$ $\forall \ x \in  \mathcal{C}_1 \cup \mathcal{C}_2$ by Lemma \ref{curves-ii-main}. Let $f \in G$. For $f=h_{b_1}$, let $L_f = \{a_1, b_1, b_2, \cdots, b_{n-3}, a_{n-3}, c, w_{n-1}\}$. The set $L_f$ has trivial stabilizer.
We know $\lambda([x])= h([x])$
$\forall \ x \in L_f$. We will check the equation for $h_{b_1}(a_1)$ and $h_{b_1}(b_2)$ since the other curves in $L_f$ are fixed by $h_{b_1}$. 
Consider the curves given in Figure \ref{fig-s13} (i), (ii). We see that $w_1=h_{b_1}(a_1)$ and $r_2=h_{b_1}(b_2)$. So, by Lemma \ref{curves-ii-main}, we have $\lambda([h_{b_1}(a_1)])= h([h_{b_1}(a_1)])$ and 
$\lambda([h_{b_1}(b_2)])= h([h_{b_1}(b_2)])$. 
So, when $f=h_{b_1}$, we have $\lambda([x])= h([x])$ $\ \forall \ x \in L_f \cup f(L_f)$.  

For $f=h_{b_2}$, let $L_f = \{a_1, b_1, b_2, \cdots, b_{n-3}, a_{n-3}, c, w_{n}\}$. The set $L_f$ has trivial stabilizer.
We know $\lambda([x])= h([x])$
$\forall \ x \in L_f$. We will check the equation for $h_{b_2}(b_1)$ and $h_{b_2}(b_3)$ since the other curves in $L_f$ are fixed by $h_{b_2}$. 
Consider the curves given in Figure \ref{fig-s13} (iii), (iv). We see that $w_2=h_{b_2}(b_1)$ and $r_3=h_{b_2}(b_3)$. So, by Lemma \ref{curves-ii-main}, we have $\lambda([h_{b_2}(b_1)])= h([h_{b_2}(b_1)])$ and $\lambda([h_{b_2}(b_3)])= h([h_{b_2}(b_3)])$. 
So, when $f=h_{b_2}$, we have $\lambda([x])= h([x])$ $\ \forall \ x \in L_f \cup f(L_f)$.  

For $f \in G \setminus \{h_{b_1}, h_{b_2}\}$,  similarly we choose let $L_f = \{a_1, b_1, b_2, \cdots, b_{n-3}, a_{n-3}, c, w_{f}\}$ where $w_f \in \{w_1, w_2, \cdots, w_n\}$ and $w_f$ is fixed by $f$. Similar to the previous cases we have $\lambda([x])= h([x])$ $\ \forall \ x \in L_f \cup f(L_f)$.\end{proof}  
 
\begin{theorem} \label{B} There exists a homeomorphism $h : R \rightarrow R$ such that $H(\alpha) = \lambda(\alpha)$ for every vertex $\alpha$ in $\mathcal{C}(R)$ where $H=[h]$ and this homeomorphism is unique up to isotopy.\end{theorem}

\begin{proof} The proof is similar to the proof of Theorem \ref{A} using Lemma \ref{curves-ii-main} and Lemma \ref{prop-imp2}.  \end{proof}\\
 
{\bf Acknowledgements}\\

We thank Peter Scott for some discussions and his comments about this paper.

{\bf Elmas Irmak} \vspace{0.2cm}

Bowling Green State University, Department of Mathematics and Statistics, Bowling

Green, 43403, OH, USA, e-mail: eirmak@bgsu.edu\vspace{0.2cm}

University of Michigan, Department of Mathematics, Ann Arbor, 48105, MI, USA, 

e-mail: eirmak@umich.edu\\

\end{document}